\documentclass{amsart}

\pdfoutput=1

\usepackage{amsmath, amsthm, amssymb, amsfonts}
\usepackage{enumerate}
\usepackage{latexsym}
\usepackage{mathrsfs}
\usepackage{amscd}
\usepackage{hyperref}
\usepackage{enumitem}
\usepackage{braket}
\usepackage[all]{xy}
\usepackage{graphicx}
\usepackage{thmtools}
\usepackage{thm-restate}
\usepackage{hyperref}
\usepackage{cleveref}
\usepackage{tikz}
\usepackage{mathtools}

\newtheorem{theorem}{Theorem}[section]
\newtheorem{definition}[theorem]{Definition}
\newtheorem{lemma}[theorem]{Lemma}
\newtheorem{proposition}[theorem]{Proposition}
\newtheorem{corollary}[theorem]{Corollary}
\newtheorem{remark}[theorem]{Remark}

\newtheorem{assumption}[theorem]{Assumption}

\numberwithin{equation}{section}

\title{Wrapped Floer cohomology and Lagrangian correspondences}

\author[Yuan Gao]{Yuan Gao\textsuperscript{1}}

\address{\textsuperscript{1}Department of Mathematics, Stony Brook University, Stony Brook NY, 11794, USA}

\email{ygao@math.stonybrook.edu}

\begin{document}

\begin{abstract}
	We study Lagrangian correspondences between Liouville manifolds and construct functors between wrapped Fukaya categories. The study naturally brings up the question on comparing two versions of wrapped Fukaya categories of the product manifold, which we prove quasi-isomorphism on the level of wrapped Floer cohomology. To prove representability of these functors constructed from Lagrangian correspondences, we introduce the geometric compositions of Lagrangian correspondences under wrapping, as new classes of objects in the wrapped Fukaya category, which we prove to represent the functors by establishing a canonical isomorphism of quilted version of wrapped Floer cohomology under geometric composition.
\end{abstract}

\maketitle

\tableofcontents

\section{Introduction}
	Lagrangian Floer theory \cite{FOOO1} has laid the foundation from the symplectic geometry side for Kontsevich's celebrated Homological Mirror Symmetry Conjecture, which predicts certain equivalence between A-model category (the Fukaya category) and B-model category (derived category of coherent sheaves) for a mirror pair of Calabi-Yau manifolds. The story has been extended to more general situations, including Fano manifolds, varieties of general type as symplectic manifolds (for example the genus two curve in \cite{Seidel}), and certain non-compact symplectic manifolds (the punctured spheres in \cite{Abouzaid-Auroux-Efimov-Katzarkov-Orlov}). The corresponding B-model categories are typically triangulated category of singularities of a Landau-Ginzburg potential. \par
	Funtoriality properties are well established for B-model categories, usually defined in algebro-geometric terms, for example derived pushforward/pullback (of sheaves) associated to morphisms between the base varieties. Further thought on homological mirror symmetry should provide some clue to understanding functorial properties of Fukaya categories, and possibly proving functoriality of the homological mirror symmetry itself. However, there are not many morphisms between symplectic manifolds realized by maps in the usual sense, except for the very restrictive classes: symplectomorphisms, coverings and embeddings. To understand functoriality of A-model categories, it is necessary to change the viewpoint from morphisms to correspondences; the natural sources are Lagrangian correspondences between symplectic manifolds. \par
	In \cite{Wehrheim-Woodward2}, Wehrheim and Woodward study functors of Fukaya categories associated to Lagrangian correspondences, in the case of compact monotone symplectic manifolds and compact monotone Lagrangian correspondences. One of their main results is that any compact monotone Lagrangian correspondence from $M$ to $N$ gives rise to an $A_{\infty}$-functor from the compact monotone Fukaya category $\mathcal{F}(M)$ of $M$ to the dg-category of $A_{\infty}$-modules over $\mathcal{F}(N)$ (technically, on the level of cohomology categories). It is expected that Lagrangian correspondences should give rise to $A_{\infty}$-functors between compact unobstructed Fukaya categories, improving the above-mentioned result. Moreover, with a lot of recent developments, the story is expected to be extended to full generality for all unobstructed compact Lagrangian correspondences. Part of the main ideas are discussed in \cite{Fukaya2}, by using Lagrangian Floer theory for immersed Lagrangians, which is developed in \cite{Akaho-Joyce}. \par
	In the case of certain non-compact symplectic manifolds which satisfy appropriate convexity conditions, there is a variant of Fukaya category, called the wrapped Fukaya category, whose objects also include certain class of non-compact Lagrangian submanifolds, and whose morphisms spaces. This is introduced in \cite{Abouzaid-Seidel}, \cite{Abouzaid1} for Liouville manifolds. The objects of the wrapped Fukaya category $\mathcal{W}(M)$ of a Liouville manifold are closed exact Lagrangian submanifolds in the compact part of $M$, as well as non-compact exact Lagrangian submanifolds that are invariant under the Liouville flow over the cylindrical ends, called conical Lagrangian submanifolds. Compared to the Fukaya category of closed Lagrangian submanifolds in $M$, the wrapped Fukaya category captures additional information about Reeb dynamics on the contact manifold at infinity. From the standpoint of homological mirror symmetry, the wrapped Fukaya category seems more suitable for the role of A-model category for a non-compact symplectic manifold (see \cite{Abouzaid-Auroux-Efimov-Katzarkov-Orlov}) which is mirror to coherent sheaves on the mirror variety/family. \par
	Also, there are many interesting affine varieties that are related by quotients or birational transformations. Some of these affine varieties appear as the divisor complements of log Calabi-Yau pairs \cite{Auroux}, \cite{Ganatra-Pomerleano}, \cite{Gross-Hacking-Keel}, \cite{Pascaleff}, for which the structure of the symplectic cohomology and the wrapped Fukaya category can be understood well to some extend. This motivates us to find a symplectic analogue of the relations among these affine varieties in terms of Lagrangian correspondences, and try to understand how their wrapped Fukaya categories are related. \par
	In this series of reseach, we begin by concerning the foundational matters regarding $A_{\infty}$-functors between wrapped Fukaya categories from the viewpoint of Lagrangian correspondences. The very first step is to adapt the quilted Floer cohomology \cite{Wehrheim-Woodward1}, \cite{Wehrheim-Woodward2} developed by Wehrheim-Woodward to our setting. For each admissible Lagrangian correspondence $\mathcal{L} \subset M^{-} \times N$ in the sense of wrapped Floer theory (Definition \ref{def: admissible Lagrangian submanifolds in the product manifold}), the principle is that we should be able to construct a canonical $A_{\infty}$-functor between wrapped Fukaya categories. To carry out the construction, we introduce a quilted version of wrapped Floer cohomology, whose definition is a straightforward modification of quilted Floer cohomology \cite{Wehrheim-Woodward1} with large Hamiltonian perturbation. However, there is one essential difference: wrapped Floer theory in the product manifold involves certain class of non-compact Lagrangian submanifolds, and there are certain issues with what objects can be and should be included. Such problems require careful inspection and will be main portion of our discussion. \par
	In this paper, we focus on foundational matters, while staying on the cohomology-level. Chain-level refinements and the full $A_{\infty}$-structures will be dealt with in the upcoming work \cite{Gao}. Further applications will also be discussed there. \par

\subsection{Floer theory in product manifolds}
	The story of Lagrangian correspondences begins with studying Floer theory in product manifolds. There are several technical issues with wrapped Floer theory for the product symplectic manifold $M \times N$. First, the standard definition of wrapped Floer cohomology depends crucially on the convexity property of the symplectic manifold. However, the product of two convex symplectic manifolds might no longer be convex. It is the reason for which we consider only Liouville manifolds, which behave nicely under products - the product of two Liouville manifolds is again a Liouville manifold, and is therefore convex at infinity. Second, for the usual wrapped Floer theory to work, we need to make a choice of a cylindrical end $\Sigma \times [1, +\infty)$ for $M \times N$. There is a natural choice, as observed in \cite{Oancea}, which will be described in section \ref{product manifold}. Third, the sum of the chosen two admissible Hamiltonian functions $H_{M, N} = \pi_{M}^{*}H_{M} + \pi_{N}^{*}H_{N}$, which we call the split Hamiltonian, is a priori not admissible. There is a similar issue with the product almost complex structure $J_{M, N} = J_{M} \times J_{N}$. This is the main obstacle to studying Floer theory in the product Liouville manifold, and has been an annoying issue for quite a while. One of the main results of this paper is to prove well-definedness of wrapped Floer cohomology in the product manifold and show the following invariance property: \par

\begin{theorem}
\label{invariance for wrapped Floer cohomology in the product}
	Suppose $\mathcal{L}_{0}, \mathcal{L}_{1}$ are admissible Lagrangian submanifolds of $M \times N$, i.e. they are either product Lagrangian submanifolds, or conical Lagrangian submanifolds with respect to the cylindrical end $\Sigma \times [1, +\infty)$. The wrapped Floer cohomology with respect to the split Hamiltonian and the product almost complex structure
\begin{equation*}
HW^{*}(\mathcal{L}_{0}, \mathcal{L}_{1}; H_{M, N}, J_{M, N}),
\end{equation*}
and the one defined using an admissible admissible $K$ and an admissible almost complex structure $J$ with respect to the cylindrical end $\Sigma \times [1, +\infty)$
\begin{equation*}
HW^{*}(\mathcal{L}_{0}, \mathcal{L}_{1}; K, J),
\end{equation*}
are both well-defined. \par
	Moreover, given a split Hamiltonian $H_{M, N}$ and a product almost complex structure $J_{M, N}$, there is an admissible Hamiltonian $K$ and an admissible almost complex structure $J$ with respect to the cylindrical end $\Sigma \times [1, +\infty)$, as well as a cochain map
\begin{equation}
R: CW^{*}(\mathcal{L}_{0}, \mathcal{L}_{1}; H_{M, N}, J_{M, N}) \to CW^{*}(\mathcal{L}_{0}, \mathcal{L}_{1}; K, J)
\end{equation}
which respects the action filtration and induces an isomorphism on cohomology. 
\end{theorem}

	A problem of the same nature was considered in \cite{Oancea} in which he studied the K\"{u}nneth formula for symplectic cohomology. We adapt his strategy and extend the argument to the case involving Lagrangian submanifolds. However, there is a small technical difference. In \cite{Oancea} the definition of symplectic cohomology uses linear Hamiltonians, so the split Hamiltonian is indeed admissible; however, symplectic cohomology defined using linear Hamiltonians involves a limit, and it is not clear that the double iterated limit is equivalent to a single limit. For the purpose of defining a quilted version of wrapped Floer cohomology and carrying out the chain level construction of $A_{\infty}$-structure in the wrapped Fukaya category, it is more convenient to work with quadratic Hamiltonians, which do not involve taking limits. But the difficulty is transferred to the non-admissibility of split Hamiltonians, which we resolve in this paper. In addition to that, we also study multiplication structures on wrapped Floer cohomology and prove the isomorphism in Theorem \ref{invariance for wrapped Floer cohomology in the product} intertwines multiplication structures. \par

\begin{theorem}
\label{preserving multiplicative structure}
The action-restriction map $R$ preserves the multiplication structure on wrapped Floer cohomology. More concretely, there exists a cochain map
\begin{equation}
R^{2}: CW^{*}(\mathcal{L}_{0}, \mathcal{L}_{1}; H_{M, N}) \otimes CW^{*}(\mathcal{L}_{0}, \mathcal{L}_{1}; H_{M, N}) \to CW^{*-1}(\mathcal{L}_{0}, \mathcal{L}_{1}; K)
\end{equation}
of degree $-1$ such that the following diagram
\begin{equation}
\begin{CD}
CW^{*}(\mathcal{L}_{1}, \mathcal{L}_{2}; H_{M, N}) \otimes CW^{*}(\mathcal{L}_{0}, \mathcal{L}_{1}; H_{M, N}) @>m^{2}>> CW^{*}(\mathcal{L}_{0}, \mathcal{L}_{2}; H_{M, N})\\
@VV R \otimes R V @VV R V\\
CW^{*}(\mathcal{L}_{1}, \mathcal{L}_{2}; K) \otimes CW^{*}(\mathcal{L}_{0}, \mathcal{L}_{1}; K) @>m^{2}>> CW^{*}(\mathcal{L}_{0}, \mathcal{L}_{2}; K)
\end{CD}
\end{equation}
is homotopy commutative, with the cochain homotopy between the two possible compositions precisely given by $R^{2}$.
\end{theorem}

	An immediate corollary is the K\"{u}nneth formula for wrapped Floer cohomology: \par

\begin{corollary}
	Let $L_{0}, L_{1} \subset M$ and $L'_{0}, L'_{1}$ be conical Lagrangian submanifolds. Then the wrapped Floer cohomology
\begin{equation*}
HW^{*}(L_{0} \times L'_{0}, L_{1} \times L'_{1}; K, J)
\end{equation*}
is well-defined with respect to any admissible Hamiltonian $K$ and any admissible almost complex structure $J$ with respect to the cylindrical end $\Sigma \times [1, +\infty)$ of $M \times N$. And there is a quasi-isomorphism
\begin{equation}
CW^{*}(L_{0}, L_{1}; H_{M}, J_{M}) \otimes CW^{*}(L'_{0}, L'_{1}; H_{N}, J_{N}) \to CW^{*}(L_{0} \times L'_{0}, L_{1} \times L'_{1}; K, J),
\end{equation}
where the differential on the left-hand side is the tensor product differential.
\end{corollary}

	The statement of K\"{u}nneth formula can be improved to an $A_{\infty}$-version, identifying the $A_{\infty}$-tensor product of wrapped Fukaya categories $\mathcal{W}(M) \otimes \mathcal{W}(N)$ with that of the product manifold $\mathcal{W}(M \times N)$, under additional assumptions. In particular, all the $A_{\infty}$-structure maps are suitably packaged, and this equivalence respects such structures. Such improvement will be discussed in \cite{Gao}. \par

\subsection{Functors}
	One important output of the study of wrapped Floer theory in the product $M^{-} \times N$ is a quilted version of wrapped Floer cohomology for Lagrangian correspondences, discussed in section \ref{wrapped Floer theory for Lagrangian correspondences}. The quilted wrapped Floer cohomology is defined $HW^{*}(L, \mathcal{L}, L')$ for the usual class of Lagrangian submanifolds $L \subset M, L' \subset N$ and an admissible Lagrangian correspondence $\mathcal{L} \subset M^{-} \times N$. It allows us to build an $A_{\infty}$-functor from $\mathcal{W}(M)$ to $\mathcal{W}(N)^{l-mod}$, the dg-category of left $A_{\infty}$-modules over $\mathcal{W}(N)$, associated to an admissible Lagrangian correspondence $\mathcal{L} \subset M^{-} \times N$. While leaving the $A_{\infty}$-structures in the upcoming paper \cite{Gao}, we carry out the construction on the level of cohomology categories in section \ref{section: module-valued functors associated to Lagrangian correspondences}. \par

\begin{theorem} \label{functors associated to Lagrangian correspondences}
	Associated to each admissible Lagrangian correspondence $\mathcal{L} \subset M^{-} \times N$, quilted wrapped Floer cohomology gives rise to a functor
\begin{equation}\label{functor associated to a single Lagrangian correspondence}
\Phi_{\mathcal{L}}: H(\mathcal{W}(M)) \to l-Mod(H(\mathcal{W}(N)))
\end{equation}
to the category of left-modules over $H(\mathcal{W}(N))$. Moreover, the construction is functorial in the product manifold $M^{-} \times N$, meaning that this can be improved to a functor
\begin{equation}\label{functor in a categorical level}
\Phi: H(\mathcal{W}(M^{-} \times N)) \to Func(H(\mathcal{W}(M)), l-Mod(H(\mathcal{W}(N))))
\end{equation}
to the category of functors. In addition, both functors are cohomologically unital.
\end{theorem}

	Attempting to obtain a functor to the honest wrapped Fukaya category $\mathcal{W}(N)$, we ask for representability of these functors, in the sense of \cite{Fukaya1}. The natural candidate to represent the module-valued functor is the geometric composition of Lagrangian correspondences. However, the geometric composition is not in general a conical Lagrangian submanifold/correspondence even if we assume it is embedded. Thus well-definedness of the wrapped Floer cohomology of the geometric composition is an essential difficulty. In section \ref{section: well-definedness of wrapped Floer cohomology of the geometric composition} we will explain well-definedness of wrapped Floer cohomology of the geometric compositions, and prove isomorphism of wrapped Floer cohomology under geometric composition. \par

\begin{theorem}
\label{geometric composition isomorphism}
	Suppose we are given Lagrangian submanifolds $L \subset M, L' \subset N$, as well as an admissible Lagrangian correspondence $\mathcal{L} \subset M^{-} \times N$. Suppose that the geometric composition $L \circ_{H_{M}} \mathcal{L}$ is a properly embedded Lagrangian submanifold of $N$. Then:
\begin{enumerate}[label=(\roman*)]

\item There is a well-defined wrapped Floer cohomology group $HW^{*}(L \circ_{H_{M}} \mathcal{L}, L')$, whose underlying cochain complex $CW^{*}(L \circ_{H_{M}} \mathcal{L}, L')$ and differential are defined as usual;

\item There is a quasi-isomorphism
\begin{equation}\label{isomorphism under geometric composition}
gc: CW^{*}(L, \mathcal{L}, L') \to CW^{*}(L \circ_{H_{M}} \mathcal{L}, L').
\end{equation}
\end{enumerate}

\end{theorem}

	Based on these results, we are going to prove the representability of the module-valued $A_{\infty}$-functors associated to Lagrangian correspondences in the upcoming work \cite{Gao}, in which certain non-compact immersed Lagrangian submanifolds are taken care of, and chain-level $A_{\infty}$-structures are constructed. This, combined with an $A_{\infty}$-analogue of Yoneda lemma, will imply that we may regard these $A_{\infty}$-functors as landing in $\mathcal{W}(N)$, up to quasi-isomorphism, as long as we allow certain immersed Lagrangian submanifolds as objects. \par
	Naturally, pushing the story one step further, we may want to study the compositions of these functors as we have introduced geometric compositions of Lagrangian correspondences. However, unlike in the case of compact Lagrangian submanifolds, wrapped Floer theory in multiple products of Liouville manifolds encounters more serious difficulty regarding classes of objects to be included, which have not yet been overcome and can be topics of future research. \par

\paragraph{\textbf{Acknowledgements}}
	This work is part of a project during the author's PhD studies at Stony Brook University. The author is grateful for his advisor Kenji Fukaya for guidance and support, sharing his ideas in Lagrangian correspondences and motivating the project. The author would also like to thank Mark McLean for sharing his knowledge on symplectic cohomology, and helpful conversations regarding certain technical issues in this work. \par

\section{Geometry of Liouville manifolds and Lagrangian submanifolds}

\subsection{Liouville manifolds}
	Let $(M, \lambda)$ be a Liouville manifold, that is, $\omega = d\lambda$ is symplectic, and the vector field $Z$ defined by
\begin{displaymath}
i_{Z}\omega = \lambda
\end{displaymath}
generates a complete expanding flow. \par
	For practical purposes, we assume that $M$ is the completion of a Liouville domain $M_{0}$. That is, inside $M$ there is an open submanifold $int(M_{0})$ whose closure is a compact manifold with boundary, such that $Z$ points outward near $\partial M_{0}$, and the flow of $Z$ identifies $M \setminus int(M_{0})$ with the positive symplectization $\partial M_{0} \times [1, +\infty)$, where $\partial M_{0}$ is equipped with the contact form $\lambda|_{\partial M_{0}}$. \par
	By abuse of notation, we write $\partial M = \partial M_{0}$, and also
\begin{displaymath}
M = M_{0} \cup_{\partial M} \partial M \times [1, +\infty).
\end{displaymath}
Denote by $r$ the radial coordinate on $\partial M \times [1, +\infty)$, so that over the cylindrical end, $\lambda = r(\lambda|_{\partial M})$. In such a case, we shall sometimes call $M_{0}$ the interior part of $M$ (only in this paper), although this is not a standard terminology. \par
	Let us make the following assumption on the Liouville form $\lambda$: \par

\begin{assumption}
	All periodic Reeb orbits of $\lambda|_{\partial M}$ on $\partial M$ are non-degenerate.
\end{assumption}

	This is a generic condition, and for such a choice of $1$-form $\lambda$, there will be only finitely many Reeb chords on $\partial M$ shorter than any given constant. \par

\subsection{Product Liouville manifolds} \label{product manifold}
	As noted before, the product $M \times N$ of two Liouville manifolds is again a Liouville manifold, carrying the product symplectic form $\omega = \omega_{M} \oplus \omega_{N}$, the product Liouville form $\lambda = \lambda_{M} \times \lambda_{N}$, and the product Liouville vector field $\nu = \nu_{M} \oplus \nu_{N}$. Suppose that we have chosen cylindrical ends for $M$ and $N$ individually. Then there is a natural choice of cylindrical end for $M \times N$, as observed in \cite{Oancea}. \par
	We consider the following three subsets of $M \times N$:
\begin{align}
U_{1} &= \partial M \times [1, +\infty) \times \partial N \times [1, +\infty), \\
U_{2} &= M \times \partial N \times [1, +\infty), \\
U_{3} &= \partial M \times [1, +\infty) \times N.
\end{align}

	We still denote the radial coordinate on $\partial M \times [1, +\infty)$ by $r_1$, and that on $\partial N \times [1, +\infty)$ by $r_2$. Then $r_1, r_2$ can be thought of as functions on these regions (via pull back). Let $\Sigma$ be a hypersurface that is transverse to the Liouville vector field $\nu$, such that:
\begin{align*}
r_{1}|_{\Sigma \cap U_3} &\equiv \alpha, & r_{1}|_{\Sigma \cap U_1} &\in [1, \alpha], \\
r_{2}|_{\Sigma \cap U_2} &\equiv \beta & r_{2}|_{\Sigma \cap U_1} &\in [1, \beta]
\end{align*}
for some $\alpha, \beta >1$. The reason that we choose $\alpha, \beta$ to be bigger than $1$ is because we do not want the hypersurface $\Sigma$ to have corners (note $\partial (M_{0} \times N_{0}) = \partial M \times N_{0} \cup M_{0} \times \partial N$ has a corner because it is the boundary of the product of two manifolds with boundary). \par
	Let $\phi_{M}^{s_{1}}$ be the time-$(\ln{s_{1}})$ Liouville flow on $M$, and $\phi_{N}^{s_{2}}$ be the time-$(\ln{s_{2}})$ Liouville flow on $N$. Let $int(\Sigma)$ denote the compact part that $\Sigma$ separates from the non-compact part of $M \times N$. Define a cylindrical end of $M \times N$ by the following parametrization:
\begin{equation}
F: \Sigma \times [1, +\infty) \to (M \times N) \setminus{int(\Sigma)}
\end{equation}
\begin{equation}
F(z, r) = (\phi^{r}_{M}(\pi_{M}(z)), \phi^{r}_{N}(\pi_{N}(z)).
\end{equation}

\subsection{Lagrangian submanifolds}
	The Lagrangian submanifolds we are going to consider are either exact closed Lagrangian submanifolds in the interior of $M$ or non-compact exact Lagrangian submanifolds $L$ that intersect $\partial M$ transversely with boundary $l = \partial L = L \cap \partial M$ being Legendrian submanifolds in $\partial M$, such that over the cylindrical end, $L \cap \partial M \times [1, +\infty)$ is of the form $l \times [1, +\infty)$. We call such a non-compact Lagrangian submanifold a conical Lagrangian submanifold. These two kinds of Lagrangian submanifolds are said to be admissible. We will mostly focus on conical Lagrangian submanifolds since only they involve non-trivial wrapping. \par

\begin{remark}
More generally, we will allow some additional Lagrangian submanifolds in the wrapped Fukaya category, and also call them admissible, in the sense that wrapped Floer cohomology for those Lagrangian submanifolds are well-defined. See section \ref{section: well-definedness of wrapped Floer cohomology of the geometric composition}.
\end{remark}

	For a conical Lagrangian submanifold $L$, we choose a primitive $f_{L}$ for the restriction of $\lambda$ to $L$, i.e. $df_{L} = \lambda|_{L}$. Since $\lambda$ vanishes on $\partial L$, we can choose $f_{L}$ such that on the cylindrical end $l \times [1, +\infty)$ of $L$ the function $f_L$ is locally constant, in particular independent of the radial coordinate $r$. \par

\subsection{Spin structures and gradings}
	The coefficients for the wrapped Floer cohomology groups of a pair of Lagrangian submanifolds will only be $\mathbb{Z}/2$ if we do not impose any conditions on the Lagrangian submanifolds. In order to have coefficients being $\mathbb{Z}$, we need to study orientations on various moduli spaces of pseudoholomorphic disks used in the definition of Floer cochain complexes. It is by now standard that a choice of spin structures on relevant Lagrangian submanifolds determines coherent orientations on all moduli spaces of (inhomogeneous) pseudoholomorphic disks. For a proof, see Chapter 8 of \cite{FOOO2} for the compact case, and section 11 of \cite{Seidel} for the exact case. \par
	The relevant Floer cohomology groups a priori carry only a $\mathbb{Z}/2$-grading. If we desire a $\mathbb{Z}$-grading, we need a trivialization of the square of the anti-canonical bundle of $M$, as well as gradings on Lagrangian submanifolds. A choice of gradings on relevant Lagrangian submanifolds determines $\mathbb{Z}$-valued gradings on generators of various Floer cochain complexes, i.e. integral lift of Maslov indices of Hamiltonians chords between a pair of Lagrangian submanifolds. For this matter, also see section 11 of \cite{Seidel}. \par
	Let us make the following assumption on the Lagrangian submanifolds in consideration:

\begin{assumption} \label{relative first Chern class vanishes}
$2c_1(M, L) \in H^{2}(M, L; \mathbb{Z})$ vanishes. 
\end{assumption}

	Under Assumption \ref{relative first Chern class vanishes}, $L$ admits both spin structures and gradings. We will fix a choice of spin structure and grading for every Lagrangian submanifold that we are going to look at, and will not repeat this later. The reader is referred to \cite{Seidel} or \cite{Abouzaid1} for a discussion on the orientations on the moduli spaces of inhomogeneous pseudoholomorphic disks and the signs they determine. \par

\subsection{Hamiltonian functions}
	We will work with a restricted class of Hamiltonian functions which have rigid behaviour near infinity. \par
	
\begin{definition}
	 A (time-independent) Hamiltonian $H: M \to \mathbb{R}$ is called admissible, if over the cylindrical end $\partial M \times [1, +\infty)$, $H$ depends only on the radial coordinate, and is quadratic in the radial coordinate outside a compact set containing the interior $M_{0}$, i.e. $H$ takes the form $H(y, r) = r^{2}$ for $r \ge r_{0}$ for some number $r_{0} > 1$. Denote by $\mathcal{H}(M)$ the space of admissible Hamiltonians.
\end{definition}

	The reason that we use time-independent Hamiltonians to setup wrapped Floer theory is that it simplifies many estimates required for proving compactness results for the relevant moduli spaces, in particular when we define the quilted version of wrapped Floer cohomology (sections \ref{section: wrapped Floer theory in the product}, \ref{section: definition of quilted wrapped Floer cohomology}), and introduce geometric compositions into wrapped Floer theory (section \ref{section: well-definedness of wrapped Floer cohomology of the geometric composition}). \par
	In practice, we may use a Hamiltonian $H$ that is quadratic for $r \ge 1+\epsilon$, and is $C^2$-small in the interior of $M$, takes values in $[0, \epsilon]$ there with the $C^1$-norm of $H$ in the interior of $M$ also being in $[0, \epsilon]$. Additionally, on a collar neighborhood of $\partial M$ in $M$, in particular at the level $r = 1$, we require the Hamiltonian takes values between $[-\epsilon, \epsilon]$. Moreover, we require the derivative of $H$ with respect to $r$ be less than $\frac{2}{\epsilon}$, namely $H$ does not grow too fast in $[1, 1+\epsilon]$. The additional requirements on admissible Hamiltonians will not affect the resulting Floer theory, up to quasi-isomorphism, by a standard continuation argument. This works because it suffices to interpolate the two almost complex structures on a compact set, so the usual $C^{0}$-estimates hold. \par
	These additional conditions imply the following lemma, the proof of which is a straightforward calculation. \par
	
\begin{lemma} \label{area estimate}
Let $H$ be an admissible Hamiltonian satisfying the additional requirements described above, and $X_{H}$ its Hamiltonian vector field. The restriction of $X_H$ on each level hypersurface $\partial M \times \{r\}$ is $2r$ times the Reeb vector field. And for each time-$T$ $X_{H}$-chord $\gamma$, we have 

\begin{enumerate}
\item If $\gamma$ lies in the interior of $M$, then
\begin{equation}
\int_{0}^{T} \gamma^{*}\lambda_{M} \le T\epsilon;
\end{equation}

\item If $\gamma$ is a Reeb chord on the level hypersurface $\partial M \times \{r\}$ for some $r \in [1, 1+\epsilon]$, then
\begin{equation}
\int_{0}^{T} \gamma^{*}\lambda_{M} \le \frac{4T}{\epsilon};
\end{equation}

\item If $\gamma$ is a Reeb chord on the level hypersurface $\partial M \times \{r\}$ for some $r \ge 1+\epsilon$, then
\begin{equation}
\int_{0}^{T} \gamma^{*}\lambda_{M} = 2Tr^{2}.
\end{equation}

\end{enumerate}

\end{lemma}

	Now let $L_{0}, L_{1}$ be two conical Lagrangian submanifolds of $M$. We may also make the following assumption on the admissible Hamiltonian $H$ that we are going to work with, which can be achieved in a generic situation. \par
	
\begin{assumption}
	All time-one chords of $H$ from $L_{0}$ to $L_{1}$ are non-degenerate.
\end{assumption}

\subsection{Almost complex structures}
	One of the technical advantages of Liouville manifolds in studying pseudoholomorphic curves and studying Floer theory is that they are convex for a natural class of almost complex structures. These almost complex structures are compatible with the symplectic structures, and over the cylindrical end transform the Liouville vector field to Reeb vector fields on level hypersurfaces $\partial M \times \{R\}$, making $\partial M$ as a convex boundary of $M$. Furthermore, the restriction of such an almost complex structure to the hyperplane distribution of the contact structure on each level hypersurface $\partial M \times \{r\}$ induces an almost complex structure compatible with the symplectic structure on that hyperplane distribution induced by the differential of the contact form. We call these almost complex structures of contact type. This can be made concise by saying that $J$ is of contact type if it is compatible with the symplectic structure and over the cylindrical end it satisfies 
\begin{equation}
\lambda_{M} \circ J = dr.
\end{equation}

	In fact, this condition can be loosen to the one which only requires the almost complex structure to be of contact type away from a compact set. The Floer theory defined using almost complex structures of the latter type will be quasi-isomorphic to the one defined using the previous ones, by a standard continuation argument, for the reason similar to that with Hamiltonians. We call this larger class of almost complex structures admissible, and denote the space of all admissible almost complex structures on $M$ by $\mathcal{J}(M)$. \par

\subsection{Floer's equation over the strip}
	Suppose we are given an admissible Hamiltonian $H$ as well as a one-parameter family of almost complex structures $J_{t}$ of contact type parametrized by $t \in [0,1]$. For conical Lagrangian submanifolds $L_{0}, L_{1}$ of $M$, and two time-one chords $\gamma_0, \gamma_1$ of $X_{H}$ from $L_{0}$ to $L_{1}$, we can consider Floer's equation:
\begin{equation}
\begin{cases}
      u: Z = (-\infty, +\infty) \times [0,1] \to M\\
      \partial_{s}u + J_{t}(\partial_{t}u - X_{H}) = 0\\
      \lim\limits_{s \to -\infty}u(s, \cdot) = \gamma_0(\cdot), \lim\limits_{s \to +\infty}u(s, \cdot) = \gamma_1(\cdot)\\
      u(s, 0) \in L_{0}, u(s, 1) \in L_{1}\\
\end{cases} 
\end{equation}

	In the third equation above, we require the convergence is exponentially fast in $|s|$. This is equivalent to the condition that the solutions have finite energy. We denote by $\tilde{\mathcal{M}}(\gamma_{0}, \gamma_{1})$ the set of solutions to Floer's equation as above, and call it the (parametrized) moduli space of inhomogeneous pseudoholomorphic strips (which we also call Floer trajectories) between $\gamma_0$ and $\gamma_1$. Since the equation is invariant under translation in the $s$-variable, there is an $\mathbb{R}$-action on $\tilde{\mathcal{M}}(\gamma_0, \gamma_1)$, which is free whenever the solutions are not constant maps. We denote by $\mathcal{M}(\gamma_0, \gamma_1)$ the quotient $\tilde{\mathcal{M}}(\gamma_0, \gamma_1)/\mathbb{R}$, and call it the (unparametrized) moduli space. The following lemma regarding transversality is standard. However, since we are using a time-independent Hamiltonian $H$, there does not seem to be a clear proof in the literature, so we include one here. \par

\begin{lemma} \label{transversality for strips}
Suppose that $\dim M \ge 4$. For a generic one-parameter family of admissible almost complex structures $J_{t}$, the moduli space $\tilde{\mathcal{M}}(\gamma_0, \gamma_1)$ is a smooth manifold whose dimension is equal to $\deg(\gamma_0) - \deg(\gamma_1)$.
If $\dim M = 2$, for a generic one-parameter family of admissible almost complex structures and a generic one-parameter family of Hamiltonians, the moduli space is a smooth manifold also.
\end{lemma}
\begin{proof}
There are three cases to consider. 
\begin{enumerate}[label=(\roman*)]

\item The moduli space consists of trivial solutions, i.e. those $u$ satisfying $du = X_{H}(u)$ identically. By a traditional trick of Gromov, this corresponds to a constant pseudoholomorphic section of some locally trivial Hamiltonian fibration with respect to a distinguished almost complex structure determined by $J$ and $H$. For these solutions, the linearized Cauchy-Riemann operator is the standard Cauchy-Riemann operator (perturbed by a constant vector) with linear Lagrangian boundary conditions, which is then surjective because the domain has genus zero.

\item One of $\gamma_{0}, \gamma_{1}$ is a Hamiltonian chord in the interior of $M$, and $u$ is non-trivial. In this case, due to the convexity, $u \in \tilde{\mathcal{M}}(\gamma_{0}, \gamma_{1})$ has to have image entirely contained in the interior of $M$, in particular the other chord also has to be inside the interior. Over there we are freely allowed to perturb $J_{t}$, so is not difficult to show that the linearization of the universal Cauchy-Riemann operator (including the almost complex structure as a variable) is surjective. To set it up, let $p>2$ and consider the universal Cauchy-Riemann operator as a section of the Banach bundle $\mathcal{E} \to \mathcal{B}$, where
\begin{equation}
\mathcal{B} = \Set{(u, J) |
\begin{cases}
u: Z \to M \text{ is of class } L^{1, p}, \text{ for some } p>2\\
u(s, 0) \in L_{0}, u(s, 1) \in L_{1},\\
\lim\limits_{s \to -\infty}u(s, \cdot) = \gamma_{0}(\cdot), \lim\limits_{s \to +\infty}u(s, \cdot) = \gamma_{1}(\cdot),\\
J \text{ is an almost complex structure compatible with } \omega,
\end{cases}
}
\end{equation}

and the fiber of $\mathcal{E}$ over $(u, J) \in \mathcal{B}$ is
\begin{equation}
\mathcal{E}_{(u, J)} = L^{p}(u^{*}TM \otimes_{J} \Lambda^{0,1}_{Z}).
\end{equation}
The inhomogeneous Cauchy-Riemann operator is then regarded as a section, denoted by $\bar{\partial}_{H}$, which sends $(u, J)$ to $\frac{1}{2}(du - dt \otimes X_{H}) + \frac{1}{2} J \circ (du - dt \otimes X_{H}) \circ j$, where $j$ is the fixed complex structure on the domain $Z$. The linearization of this is then the following Fredholm operator:
\begin{equation}
D_{(u, J)}\bar{\partial}_{H}: \mathcal{Y} \oplus L^{1, p}(u^{*}TM; u^{*}TL_{0}, u^{*}TL_{1}; ) \to L^{p}(u^{*}TM \otimes_{J} \Lambda^{0,1}_{Z})
\end{equation}
\begin{equation}
D_{(u, J)}\bar{\partial}_{H}(Y, \xi) = \frac{1}{2}Y \circ (du - dt \otimes X_{H}) \circ j + D_{u}\bar{\partial}_{J, H}(\xi),
\end{equation}
where $\mathcal{Y}$ is the tangent space of the space of compatible almost complex structures, and $\bar{\partial}_{J, H}$ is the inhomogeneous Cauchy-Riemann operator for the single $J$ and $H$. We want to prove that the universal linearized operator $D_{(u, J)}\bar{\partial}_{H}$ is surjective. Suppose the contrary. Then there exists a nonzero $\eta \in \ker D_{u}^{*}\bar{\partial}_{J, H}$, which is automatically smooth by elliptic regularity, and vanishes at most at a discrete set of points, such that
$$\langle Y \circ (du - dt \otimes X_{H}) \circ j, \eta \rangle = 0.$$
But since $u$ is not a trivial solution,  $du - dt \otimes X_{H}$ vanishes at most at a discrete set of points on the domain. In particular, we can choose $Y$ such that at some point $z \in Z$ which is not in the union of these two discrete sets,
$$\langle Y_{z} \circ (du_{z} - dt \otimes X_{H}(u(z))) \circ j_{z}, \eta_{z} \rangle \ne 0.$$
This is a contradiction, which implies that $D_{(u, J)}\bar{\partial}_{H}$ is surjective. Considering the projection $\ker D_{(u, J)}\bar{\partial}_{H} \to \mathcal{J}(M)$, we get the conclusion by applying Sard-Smale theorem.

\item Both $\gamma_{0}, \gamma_{1}$ are Reeb chords contained in level hypersurfaces of the cylindrical end $\partial M \times [1, +\infty)$, and they are contained in different levels, say $\partial M \times \{a_0\}$ and $\partial M \times \{a_1\}$. In this case, we need a SFT-type transversality argument. Over the cylindrical end the almost complex structure can be assumed to be of contact type, which means that it is the direct sum of the almost complex structure $J_{\xi}$ on the contact hyperplane distribution $\xi$ of the contact manifold $\partial M$ with the standard conjugate complex structure on $\mathbb{C} = \mathbb{R}^{2}$, where the first $\mathbb{R}$ factor is the Reeb direction. Therefore, writing $u = (v, l): Z \to \partial M \times [1, +\infty)$ we can split this equation to a SFT-type system of equations:
\begin{equation}
\begin{cases}
\pi_{\xi}(\frac{\partial v}{\partial s}) + J_{\xi} \circ \pi_{\xi}(\frac{\partial v}{\partial t}) = 0\\
\frac{\partial l}{\partial s} - \pi_{Reeb}(\frac{\partial v}{\partial t}) + H_{S}'(l) = 0\\
\frac{\partial l}{\partial t} + \pi_{Reeb}(\frac{\partial v}{\partial s}) = 0.
\end{cases}
\end{equation}
And there are boundary conditions for $\pi_{\xi}(dv)$: the two boundary components lie in the Lagrangian subbundles of $\xi$ which are the tangent bundles of the Legendrian submanifolds $l_{0}$ and $l_{1}$ respectively. \par
	Suppose now $\dim M \ge 4$. Then the first equation is non-trivial, and is a Cauchy-Riemann equation on the contact hyperplane distribution with Lagrangian boundary conditions, and $J_{\xi}$ can be any almost complex structure compatible with the symplectic form on $\xi$, hence the linearized operator can be made surjective by a generic choice of $J_{\xi}$. Note this is true even if the projection $\pi_{\xi}(dv)$ is constant (namely when $dv$ is contained in the Reeb direction), because that reduces to the usual linear Cauchy-Riemann equation with linear Lagrangian boundary conditions. \par
	In order to prove that the linearization of the full equation can be made surjective by generic choice of $J_{\xi}$, we need to show that the component of the linearization in the $\mathbb{R}^2$-direction is indeed regular, because writing down the linearization explicitly shows that there is nothing to perturb. Let us try to solve the equation
\begin{equation*}
D_{u}\bar{\partial}_{J, H}(\eta) = 0.
\end{equation*}
Let us write $\eta = (\eta_{\xi}, \eta_{Reeb}, \eta_{r})$ as the decompositon according to the decomposition of the tangent bundle of $\partial M \times [1, +\infty)$. As indicated above, the component in $\xi$ can be made surjective, so we obtain some solution $\eta_{\xi}$ of Sobolev class $L^{1, p}$, which is in fact in $L^{k, p}$ for any $k$ by elliptic bootstrapping. The remaining component of the equation has the form:
\begin{equation}
\begin{cases}
\frac{\partial \eta_{r}}{\partial s} - \frac{\partial \eta_{Reeb}}{\partial t} + 2 \eta_{r} + f(\eta_{\xi}) = 0\\
\frac{\partial \eta_{r}}{\partial t} + \frac{\partial \eta_{Reeb}}{\partial s} + g(\eta_{\xi}) = 0\\
\eta_{Reeb}(s, 0) = 0\\
\eta_{Reeb}(s, 1) = 0
\end{cases}
\end{equation}
where $f(\eta_{\xi}), g(\eta_{\xi})$ are functions of $\eta_{\xi}$ depending only on the contact structure on $\partial M$, as well as on the point $u$ where we linearize the equation, and are in fact polynomial in $\eta_{\xi}$, but do not depend on the derivatives of it. For $s \gg 0$, $\eta_{\xi}$ decays exponentially in $s$, because of the Sobolev condition. By elementary method in PDE (for example using Fourier series), the only solution has exponential growth with the same rate at both $+\infty$ and $-\infty$, and hence is not of class $L^{k, p}$ for any $k, p$.
\end{enumerate}
\end{proof}

	As a consequence, the unparametrized moduli space $\mathcal{M}(\gamma_0, \gamma_1)$ is a smooth manifold of dimension $\deg(\gamma_0) - \deg(\gamma_1) - 1$ whenever the $\mathbb{R}$-action is free, otherwise the parametrized moduli space consists of constant maps and $\mathcal{M}(\gamma_0, \gamma_1)$ is regarded as empty. \par
	We introduce a (partial) compactification of $\mathcal{M}(\gamma_0, \gamma_1)$ by adding strata consisting of broken strips:
\begin{equation} \label{compactification of moduli space of strips}
\bar{\mathcal{M}}(\gamma_0, \gamma_1) = \coprod \mathcal{M}(\gamma_0, \theta_1) \times \mathcal{M}(\theta_1, \theta_2) \times \cdots \times \mathcal{M}(\theta_{k-1}, \theta_k) \times \mathcal{M}(\theta_k, \gamma_1).
\end{equation}

	Because the symplectic manifold $M$ is exact, and the Lagrangian submanifolds $L_0, L_1$ are exact, there cannot be any sphere bubbling or disk bubbling. The maximum principle implies that this is indeed a compactification. However, the method of proof we are going to use is slightly different, though essentially equivalent. \par
	
\begin{lemma} \label{regularity + compactness}
For a generic one-parameter family of admissible almost complex structures $J_{t}$, $\bar{\mathcal{M}}(\gamma_0, \gamma_1)$ is a compact manifold with boundary and corners of dimension $\deg(\gamma_0) - \deg(\gamma_1)$. The codimension one boundary stratum is covered by images of the natural inclusions
$$\mathcal{M}(\gamma_0, \theta_1) \times \mathcal{M}(\theta_1, \gamma_1) \to \bar{\mathcal{M}}(\gamma_0, \gamma_1)$$
\end{lemma}
\begin{proof}
	This is a direct consequence of the action-energy equality, which expresses the energy of a pseudoholomorphic strip between two chords in terms of the action of these two chords:
\begin{equation}\label{action-energy equality for strips}
\int_{Z} \frac{1}{2}|du - X_{H}(u)|^2 = \mathcal{A}_{H, L_0, L_1}(\gamma_0) - \mathcal{A}_{H, L_0, L_1}(\gamma_1)
\end{equation}
It follows that there cannot be pseudoholomorphic strips escaping to infinity, because the energy of the pseudoholomorphic strips between given two chords is fixed by the action of them. On the other hand, once $\gamma_0$ and $\gamma_1$ are fixed, there can only be finitely many non-trivial broken strips, because each non-trivial pseudoholomorphic strip picks up some energy which is positive and uniformly bounded from below by a constant which depends only on the background geometry - the symplectic manifold, the Lagrangian submanifolds and the Hamiltonian function. Moreoever, there cannot be sphere bubbles or disk bubbles as mentioned before.
\end{proof}

	However, the set of $X_{H}$-chords between $L_0$ and $L_1$ is in general infinite. But still we have the following finiteness result, which also follows from the action-energy equality. \par
	
\begin{lemma} \label{compactness 2}
For each time-one $X_{H}$-chord $\gamma_1$ from $L_0$ to $L_1$, the compactified moduli space $\bar{\mathcal{M}}(\gamma_0, \gamma_1)$ is empty for all but finitely many chords $\gamma_0$.
\end{lemma}

	The above two lemmata are the key geometric ingredients that ensure finiteness in the definition of Floer's differential, the statements and most of the ideas of which, except the SFT-type argument, are learned from \cite{Abouzaid1}, \cite{Abouzaid-Seidel}. \par
	
\section{Wrapped Floer cohomology}

\subsection{Basic definition}
	Let us briefly recall the definition of wrapped Floer cohomology, basically taken from \cite{Abouzaid1}. We include it here for the convenience of the reader, and use it to fix notations. Given a pair of admissible Lagrangian submanifolds $L_{0}, L_{1}$, we define a graded $\mathbb{Z}$-module:
\begin{equation}
CW^{*}(L_0, L_1; H, J_{t}) = \bigoplus_{\gamma \in \mathcal{X}(L_{0}, L_{1}; H)} \mathbb{Z} |o_{\gamma}|
\end{equation}
where $\mathcal{X}(L_0, L_1)$ is the set of all time-one $H$-chords from $L_0$ to $L_1$, and $|o_{\gamma}|$ is the canonical orientation line associated to the chord $\gamma$, (see [Seidel]). \par
	This as a graded module is in fact independent of $J_{t}$. We then define a differential $m^{1}=m^{1}_{H, J_{t}}$ on $CW^{*}(L_0, L_1; H, J_{t})$ by counting isolated elements in the moduli space of inhomogeneous pseudoholomorphic strips. More precisely, we consider the case where $\deg(\gamma_0) = \deg(\gamma_1) + 1$. Using the orientation on the moduli space $\mathcal{M}(\gamma_0, \gamma_1)$ determined by the chosen spin structures on the Lagrangian submanifolds, we obtain an isomorphism $o_{\gamma_1} \to o_{\gamma_0}$. We denote by $m^{1}_{u}: |o_{\gamma_1}| \to |o_{\gamma_0}|$ the induced map on orientation lines, and define
\begin{equation}
m^{1}: CW^{i}(L_0, L_1; H, J_{t}) \to CW^{i+1}(L_0, L_1; H, J_{t})
\end{equation}
\begin{equation}
m^{1}([\gamma_1]) = (-1)^{i} \sum_{\substack{u \in \bar{\mathcal{M}}(\gamma_0, \gamma_1)\\ \deg(\gamma_0) = \deg(\gamma_1) + 1}} m^{1}_{u}([\gamma_1]).
\end{equation}

	By Lemma \ref{compactness 2}, the sum is finite. The fact that $m^{1}$ squares to zero comes from the structure of the boundary of one-dimensional moduli space, described in Lemma \ref{regularity + compactness}, plus a standard gluing argument. We call the resulting cohomology group the wrapped Floer cohomology of $L_0, L_1$ with respect to $(H, J_{t})$ and denote it by $HW^{*}(L_0, L_1; H, J_{t})$. The wrapped Floer cochain complex does not depend on admissible $(H, J_{t})$ up to canonical chain homotopy up to higher chain homotopies, the proof of which uses continuation maps. Thus the wrapped Floer cohomology is independent of admissible pairs $(H, J_{t})$. \par

\begin{remark}
	From now on, when we define various kinds of moduli spaces, we always start with one-parameter families of almost complex structures that are of contact type. Instead of emphasizing this point every time, we will only denote them as $J, J_{M}, J_{N}$, etc. But these symbols should be understood as one-parameter families of almost complex structures parametrized by $t \in [0, 1]$, unless otherwise specified.
\end{remark}

\subsection{Multiplicative structure}
	We review the multiplicative structure on wrapped Floer cohomology. The chain level operation (which is in general not necessarily associative)
\begin{equation}
m^{2}: CW^{*}(L_1, L_2; H, J_{t}) \otimes CW^{*}(L_0, L_1; H, J_{t}) \to CW^{*}(L_0, L_2; H, J_{t})
\end{equation}
is defined to be the composition of
\begin{equation} \label{pre-multiplication}
CW^{*}(L_1, L_2; H, J_{t}) \otimes CW^{*}(L_0, L_1; H, J_{t}) \to CW^{*}(\phi^{2}L_0, \phi^{2}L_1; \frac{H}{2} \circ \phi^{2}, (\phi^{2})^{*}J_{t})
\end{equation}
with
\begin{equation}
CW^{*}(\phi^{2}L_0, \phi^{2}L_1; \frac{H}{2} \circ \phi^{2}, (\phi^{2})^{*}J_{t}) \to CW^{*}(L_0, L_2; H, J_{t})
\end{equation}
where the first map is defined by counting inhomogeneous pseudoholomorphic triangles, and the second map is induced by conjugating by the Liouville flow, which is an isomorphism. Both maps are cochain maps. See \cite{Abouzaid1} for the original construction of all $A_{\infty}$-structure maps, but let us review some of the definitions to fix terminology and notations we shall use. \par
	Let $\phi^{s}$ be the time-$(\ln{s})$ Liouville flow. We first consider the map \eqref{pre-multiplication}, whose definition uses pseudoholomorphic triangles. Let $S$ be a disk with three boundary punctures $\xi_0, \xi_1, \xi_2$ where $\xi_0$ is a negative puncture while the other two are positive punctures. We then need to choose a smooth function $\rho_{S}: \partial D^2 \to [1, 2]$ which is $1$ near $\xi_1, \xi_2$ and is $2$ near $\xi_0$. We call such a function a time-shifting function. Let $\mathcal{H}(M)$ denote the space of admissible Hamiltonians on $M$, and $\mathcal{J}(M)$ the space of admissible almost complex structures on $M$. The definition of the inhomogeneous pseudoholomorphic triangles also involves additional data as described below. \par

\begin{definition}\label{Floer datum}
A Floer datum for $S$ consists of
\begin{enumerate}[label = (\alph*)]

\item A basic one-form $\alpha_{S}$ on $S$, which is closed ($d \alpha_{S} = 0$) and vanishes on the boundary of $S$, and whose pullback by $\epsilon^{k}$ is $2dt$ if $k=0$, and $dt$ if $k=1, 2$;

\item A $S$-dependent family of admissible Hamiltonians on $M$, i.e. a smooth map $H_{S}: S \to \mathcal{H}(M)$, whose pullback under $\epsilon^{k}$ agrees with $\frac{H}{4} \circ \phi^{2}$ if $k=0$, and with $H$ if $k=1, 2$;

\item A $S$-dependent family of admissible almost complex structures on $M$, i.e. a smooth map $J_{S}: S \to \mathcal{J}(M)$, whose pullback under $\epsilon^{k}$ agrees with $(\phi^{2})^{*}J_{t}$ if $k=0$, and with $J_{t}$ if $k=1, 2$.

\end{enumerate}
\end{definition}

	Let $\mathcal{M}_{2}(\gamma_{0}; \gamma_{1}, \gamma_{2})$ be the moduli space of inhomogeneous pseudoholomorphic maps $u: S \to M$ which satisfies the following conditions:
\begin{equation}
\begin{cases}
(du - \alpha_{S} \otimes X_{H_{S}})^{0,1} = 0\\
u(z) \in \phi^{\rho_{S}(z)}L_0 &\text{if $z \in \partial S$ lies between $\xi_0$ and $\xi_1$}\\
u(z) \in \phi^{\rho_{S}(z)}L_1 &\text{if $z \in \partial S$ lies between $\xi_1$ and $\xi_2$}\\
u(z) \in \phi^{\rho_{S}(z)}L_2 &\text{if $z \in \partial S$ lies between $\xi_2$ and $\xi_0$}\\
\lim\limits_{s \to -\infty} u \circ \epsilon^{0}(s, \cdot) = \phi^{2}\gamma_{0}(\cdot)\\
\lim\limits_{s \to +\infty} u \circ \epsilon^{k}(s, \cdot) = \gamma_{k}(\cdot) &\text{if $k=1, 2$}.
\end{cases}
\end{equation}
In the above equation, the $(0,1)$-part of the differential is defined with respect to the complex structure on $S$ and the $S$-parametrized family of almost complex structures $J_{S}$ on $M$. We have the following transversality result :

\begin{lemma} \label{transversality for 3-pointed disks}
	For a generic family of admissible almost complex structures $J_{S}$ as part of the Floer datum defined in Definition \ref{Floer datum}, the moduli space $\mathcal{M}_{2}(\gamma_{0}; \gamma_{1}, \gamma_{2})$ is a smooth manifold of dimension $\deg(\gamma_0) - \deg(\gamma_1) - \deg(\gamma_2)$.
\end{lemma}
\begin{proof}
	If one of the chords lie in the interior of $M$, then the proof argues with the linearization of the universal Cauchy-Riemann operator, in which $J_{S}$ is regarded as a variable. It is not difficult to show that this universal Fredholm operator is surjective, essentially because we are allowed to perturb $J$ in the interior of $M$ in an arbitrary way (unlike the case that $J$ has to be of contact-type over the cylindrical end), then a Sard-Smale argument finishes the proof. We omit the detail here, which is in fact the same as the proof given for Lemma \ref{transversality for strips}. \par
	If all the chords $\gamma_{0}, \gamma_{1}, \gamma_{2}$ lie in the cylindrical end where the almost complex structures are only allowed to be perturbed in the space of admissible (or more restrictively contact-type) almost complex structures, it is not clear a priori the universal Fredholm operator is still surjective. For this we need a SFT-type as in the proof of Lemma \ref{transversality for strips}. Since all the three chords lie in the cylindrical end, the inhomogeneous pseudoholomorphic triangle $u$ must be also contained in the cylindrical end, because the collar neighborhood of $\partial M$ in the interior of $M$ is concave. So in the inhomogeneous Cauchy-Riemann equation that $u$ satisfies, the almost complex structure is of contact type, which means that it is the direct sum of the almost complex structure $J_{\xi}$ on the contact hyperplane distribution $\xi$ of the contact manifold $\partial M$ with the standard conjugate complex structure on $\mathbb{C} = \mathbb{R}^{2}$, where the first $\mathbb{R}$ factor is the Reeb direction. Therefore, writing $u = (v, l): S \to \partial M \times [1, +\infty)$ we can split this equation to a SFT-type system of equations:
\begin{equation}
\begin{cases}
\pi_{\xi}(dv) + J_{\xi} \circ \pi_{\xi}(dv) \circ j_{S} = 0\\
dl - \pi_{Reeb}(dv) \circ j + \alpha_{S} \circ j_{S} \otimes H_{S}'(l) = 0.
\end{cases}
\end{equation}
Here $\pi_{\xi}: T(\partial M) \to \xi$ is the projection to the contact hyperplane distribution, and $\pi_{Reeb}: T(\partial M) \to \partial M \times \mathbb{R}$ is the projection to the Reeb direction. Following the same kind of SFT-type argument as in the proof of Lemma \ref{transversality for strips} by going to strip-like ends, we obtain the desired conclusion. In this case, that is the domain is a disk with three or more boundary punctures, the proof even also works directly for dimension two. This is because we have used a generic domain-dependent family of almost complex structures that can vary freely away from the strip-like ends, and consequently multiply-covered solutions will never appear.
\end{proof}

	Since there cannot be sphere bubbles or disk bubbles, the moduli space $\mathcal{M}_{2}(\gamma_{0}; \gamma_{1}, \gamma_{2})$ has a natural compactification $\bar{\mathcal{M}}_{2}(\gamma_{0}; \gamma_{1}, \gamma_{2})$ whose codimension one boundary stratum is covered by 
\begin{equation}
\begin{split}
\coprod \mathcal{M}(\theta_{0}, \gamma_{0}) \times \mathcal{M}_{2}(\theta_{0}; \gamma_{1}, \gamma_{2}) &\cup \coprod \mathcal{M}_{2}(\gamma_{0}; \theta_{1}, \gamma_{2}) \times \mathcal{M}(\gamma_{1}, \theta_{1})\\
&\cup \coprod \mathcal{M}_{2}(\gamma_{0}; \gamma_{1}, \theta_{2}) \times \mathcal{M}(\gamma_{2}, \theta_{2}).
\end{split}
\end{equation}
$\bar{\mathcal{M}}_{2}(\gamma_{0}; \gamma_{1}, \gamma_{2})$ is a compact manifold, because we have a similar result to Lemma \ref{compactness 2}, which in this case says that for a fixed pair $(\gamma_{1}, \gamma_{2})$, the moduli space $\mathcal{M}_{2}(\gamma_{0}; \gamma_{1}, \gamma_{2})$ is empty for all but finitely many $\gamma_{0}$. \par
	The compactness result uses the following action-energy equality 
\begin{equation}\label{action-energy equality for multiplication structure}
\int_{S}\frac{1}{2}|du - \alpha_{S} \otimes X_{S}|^2 = \mathcal{A}_{\frac{H}{2} \circ \phi^{2}}(\phi^{2}\gamma_{0}) - \mathcal{A}_{H}(\gamma_{1}) - \mathcal{A}_{H}(\gamma_{2}) + (\text{curvature term})
\end{equation}
where the curvature term is $\int_{S} u^{*}H_{S} d\alpha_{S}$ in general, which is in fact zero here by our choice of $\alpha_{S}$ being closed. Using this, and by an argument that is completely similar to the ones for Lemma \ref{regularity + compactness} and Lemma \ref{compactness 2}. \par

\subsection{The cohomological unit} \label{cohomological unit}
	It is well-known that cohomology-level multiplication on wrapped Floer cohomology has a unit. For the convenience of the reader, and also for the purpose of using the representative explicitly later in the proof of Theorem \ref{geometric composition isomorphism}, we include a description here. To save notations we consider the case of a single Lagrangian submanifold $L$, and the multiplication:
\begin{equation*}
m^{2}: HW^{*}(L, L) \otimes HW^{*}(L, L) \to HW^{*}(L, L).
\end{equation*}
The unit $e_{L} \in HW^{0}(L, L)$ for the multiplication is the image of $1 \in H^{0}(L)$ under the canonical map
\begin{equation} \label{PSS homomorphism}
PSS: H^{*}(L) \to HW^{*}(L, L),
\end{equation}
called the Piunikhin-Salamon-Schwarz (PSS) homomorphism, which is defined as follows. Consider the disk with one negative boundary puncture $D^{2} \setminus \{-1\}$. Choose a strip-like end $\epsilon_{-}: (-\infty, 0] \times [0, 1] \to D^{2} \setminus \{-1\}$ near the puncture. Choose a family of Hamiltonians $H_{z}$ depending on $z \in D^{2} \setminus \{-1\}$, which agrees with $H$ over the strip-like end, and vanishes in a neighborhood of $1 \in \partial D^{2}$, and moreover away from this neighborhood is admissible for all $z$. Also, choose a family of admissible almost complex structures $J_{z}$ depending on $z \in D^{2} \setminus \{-1\}$ which agrees with $J_{t}$ over the strip-like end. We will see that admissibility condition is in fact irrelevant. Given an $H$-chord $\gamma$, and a locally finite chain $P \subset L$ (we fix a smooth triangulation of $L$ at the beginning) which can be taken to be represented by an oriented piecewise linear submanifold, we consider the parametrized inhomogeneous Cauchy-Riemann equation:
\begin{equation} \label{PSS moduli space}
\begin{cases}
(du - X_{H_{z}} \otimes dt)^{0,1} = 0\\
u(\partial (D^{2} \setminus \{-1\})) \subset L\\
\lim\limits_{s \to -\infty} u \circ \epsilon_{-}(s, \cdot) = \gamma(\cdot)\\
u(1) \in P
\end{cases}
\end{equation}
for smooth maps $u: D^{2} \setminus \{-1\} \to M$. Note that although $dt$ does not extend across the point $1 \in \partial D^{2}$, the Hamiltonian $H_{z}$ vanishes in a neighborhood of $1$ in $D^{2} \setminus \{-1\}$, so the above equation still makes sense. Alternatively, one may use a different one-form on $D^{2}$, which is only sub-closed, and vanishes on the boundary. This is a minor issue which will not affect the definition of the cochain map, up to homotopy. Let $\mathcal{M}_{PSS}(\gamma; P)$ be the moduli space of solutions to the equation \eqref{PSS moduli space}. The following is a transversality result for this moduli space, whose proof is similar to that of Lemma \ref{transversality for strips}. \par

\begin{lemma}
For a generic family of almost complex structures $J_{z}$, the moduli space $\mathcal{M}_{PSS}(\gamma; P)$ is a piecewise linear manifold of dimension $\deg(\gamma) - n + \dim{P}$.
\end{lemma}
\begin{proof}
	The only difference from the conclusion by standard transversality argument is that our moduli space is only a piecewise linear manifold. This is because without the condition $u(1) \in P$, the resulting moduli space is a smooth manifold, and we perturb the almost complex structures further such that the evaluation map from that moduli space at the boundary marked point $1$ is transverse to the piecewise linear submanifold $P$. Of course, if $P$ were represented by a smooth submanifold, the moduli space $\mathcal{M}_{PSS}(\gamma; P)$ would be a smooth manifold. But there are obstructions to representing integral homology classes by smooth submanifolds.
\end{proof}

\begin{remark}
	If we are only concerned with the cohomological unit, we only have to consider the case where $P$ is the (locally finite) fundamental chain $[L]$ of $L$, which is represented by the smooth manifold $L$ itself. In that case, the resulting moduli space is a smooth manifold, after generic perturbation. \par
	Using a Morse complex for certain Morse function on $L$, or pseudocycles to compute the cohomology $H^{*}(L)$ would resolve the problem of the moduli space not a priori being a smooth manifold. But piecewise linear manifolds are fine for our purpose, because we will only use zero-dimensional and one-dimensional moduli spaces.
\end{remark}

	There is a natural compactification $\bar{\mathcal{M}}_{PSS}(\gamma; P)$ of the moduli space $\mathcal{M}_{PSS}(\gamma; P)$ whose codimension one stratum is covered by:
\begin{equation}
\mathcal{M}(\gamma, \gamma') \times \mathcal{M}_{PSS}(\gamma'; P) \coprod \mathcal{M}_{PSS}(\gamma; \partial P)
\end{equation}
It can be proved that $\bar{\mathcal{M}}_{PSS}(\gamma; P)$ is a compact piecewise linear manifold with boundary and corners, but the proof is somewhat cumbersome and we do not quite need it. We only need the result for zero-dimensional and one-dimensional moduli spaces, which is straightforward. In particular, the zero-dimensional moduli space $\mathcal{M}_{PSS}(\gamma; P)$ is compact, where $\deg(\gamma) = n - \dim{P}$. \par
	We also have the following finiteness result. \par

\begin{lemma} \label{compactness for PSS moduli space}
For any locally finite chain $P$ of $L$, the compactified moduli space $\bar{\mathcal{M}}_{PSS}(\gamma; P)$ is empty for all but finitely many $H$-chords $\gamma$.
\end{lemma}
\begin{proof}
	Note that for any $P$, $\mathcal{M}_{PSS}(\gamma; P)$ is empty unless the relative homotopy class of the map $\gamma: ([0,1], \{0\}, \{1\}) \to (M, L, L)$ is trivial, i.e. the same as a constant map $[0, 1] \to L$, in the space $\mathcal{P}(M, L, L)$ of paths in $M$ with the two ends in $L$. Reeb chords on the level hypersurfaces do not have trivial relative homotopy class. Therefore the possible outputs $\gamma$ can only be Hamiltonians chords contained in the interior of $M$, and there are finitely many of them.
\end{proof}

	We consider the zero-dimensional moduli space $\mathcal{M}_{PSS}(\gamma; P)$, namely the case $\deg(\gamma) = n - \dim{P}$. Each $u \in \mathcal{M}_{PSS}(\gamma; P)$ induces a map of orientation lines:
\begin{equation}
PSS_{u}: |o_{P}| \to |o_{\gamma}|,
\end{equation}
Then we define:
\begin{equation} \label{PSS cochain map}
PSS: C_{n-*}^{lf}(L) \to CW^{*}(L, L)
\end{equation}
\begin{equation}
PSS(P) = \sum_{\substack{u \in \mathcal{M}_{PSS}(\gamma; P)\\ \deg(\gamma) = n - \dim{P}}} PSS_{u}(P).
\end{equation}
By Lemma \ref{compactness for PSS moduli space}, the sum is finite, and therefore $PSS$ is well-defined. Studying one-dimensional moduli space, plus a simple case of the much more involved gluing argument dealing with pseudoholomorphic curves and ambient chains together presented in \cite{FOOO3}, shows that $PSS$ is a cochain map, and thus induces the desired map on cohomology \eqref{PSS homomorphism}, via the non-compact Poincar\'{e} duality
\begin{equation*}
H_{n-*}^{lf}(L) \cong H^{*}(L).
\end{equation*}
The construction of PSS homomorphism is quite classical, and there are other versions using different chain models on $L$ in which this kind of mixed-type gluing argument as in \cite{FOOO3} is not required. But we find it pictorially clear to use simplicial chains on $L$ to describe the behavior of this map. Of course, all versions are equivalent, on the level of cohomology. \par
	We then define $e_{L} = PSS([L]) \in HW^{0}(L, L)$. By modifying $H$ slightly if necessary, we may assume that $e_{L}$ is represented by a unique $H$-chord $\gamma_{min}$ from $L$ to itself that has index zero. Another interpretation is that this chord corresponds to the unique minimum of $H|_{L}$, which on the other hand corresponds to the unique degree zero intersection point of $L \cap \phi_{H}L$ in the interior of $M$. Moreover, we can arrange that its action is the biggest among all chords (recall all Reeb chords have sufficiently negative action). On the other hand, for any $\epsilon > 0$, we may choose an admissible $H$ that is $C^{2}$-small in the interior such that the absolute value of the action of $\gamma_{min}$ is less than $\epsilon$. Note that we cannot ensure that this chord has positive action, because of the presence of the terms involving the primitive $f_{L}$, unlike the case of symplectic cohomology whose underlying cochain complex is generated by closed periodic orbits. But most of the estimates still go through. In particular, it is easy to see that $e_{L}$ is a unit for the multiplication $m^{2}$ on the level of cohomology. \par

\begin{lemma}
For any $[\gamma] \in HW^{*}(L, L)$, we have that
\begin{equation}
m^{2}(e_{L}, [\gamma]) = [\gamma] \in HW^{*}(L, L).
\end{equation}
\end{lemma}
\begin{proof}
	Consider the three punctured disk used to define multiplication structure on wrapped Floer cohomology. At the first positive puncture we put the asymptotic condition $\gamma_{min}$ which represents $e_{L}$, and at the second positive puncture we put the asymptotic condition $\gamma$. Recall that near the negative puncture, we have applied the time-$(\ln{2})$ Liouville flow to the Lagrangian submanifolds and the asymptotic Hamiltonian chord. Requiring that the inhomogeneous pseudoholomorphic triangle converges to $\phi^{2}\gamma'$ as a chord from $\phi^{2}L$ to itself, we obtain the moduli space $\mathcal{M}(\gamma'; \gamma_{min}, \gamma)$ of inhomogeneous pseudoholomorphic triangles. In order for this moduli space to be of dimension zero, we must require that $\deg(\gamma) = \deg(\gamma')$. \par
	Then the key point is that in our construnction of $e_{L}$, we have an inhomogeneous pseudoholomorphic map $D^{2} \setminus \{-1\} \to M$ which converges to $\gamma_{min}$ over the negative strip-like end near the puncture $-1$. We may glue this to the triangle at the first positive puncture, and obtain a disk with one positive puncture and one negative puncture. The fact that the marked point $1$ on $D^{2} \setminus \{-1\}$ is mapped to the fundamental chain of $L$ shows that there is no constraint on that marked point, which means after gluing we obtain an inhomogeneous pseudoholomorphic strip with moving boundary and domain-dependent Floer datum, connecting $\gamma$ and $\phi^{2}\gamma'$. Therefore, we see that the moduli space $\mathcal{M}(\gamma'; \gamma_{min}, \gamma) \times \mathcal{M}_{PSS}(\gamma_{0}, [L])$ is cobordant to the moduli space defining the continuation map with a perturbed family of Hamiltonians converging to the same $H$ over both strip-like ends. \par
	If $\gamma, \gamma'$ correspond Reeb chords on level hypersurfaces, then they must be on the same level because they have the same Conley-Zehnder index while being connected by a continuation strip, as all such chords are isolated. This forces them to have the same action, because the Hamiltonian $H$ depends only on the radial coordinate $r$ and the primitive $f_{L}$ for $L$ is locally constant over the cylindrical end. By the action-energy equality \eqref{action-energy equality for multiplication structure} applied to the inhomogeneous pseudoholomorphic triangle, we see that there is not enough energy for $\gamma$ to be different from $\gamma'$ for there to exist a non-constant continuation strip. If $\gamma, \gamma'$ are Hamiltonian chords in the interior, then the standard homotopy method for continuation map also shows that $[\gamma] = [\gamma'] \in HW^{*}(L, L)$.
\end{proof}

	A completely parallel argument shows that $e_{L}$ is also a right cohomological unit. In fact, the same proof implies that this element $e_{L}$ is a cohomological unit for the triangle products
\begin{equation}
m^{2}: HW^{*}(L, L) \otimes HW^{*}(L, L') \to HW^{*}(L, L')
\end{equation}
and
\begin{equation}
m^{2}: HW^{*}(L', L) \otimes HW^{*}(L, L) \to HW^{*}(L', L)
\end{equation}
for any admissible $L'$. \par

\begin{remark}
	The description of the cohomological unit $e_{L}$ does not say that it is nonzero, and the PSS homomorphism does not have to be injective or surjective. There are examples where $HW^{*}(L, L)$ vanishes, so in particular $e_{L} = 0$. There could be differentials coming from degree $-1$ generators in $CW^{*}(L, L)$ killing the representative of $e_{L}$, while in simplicial homology there are no simplicial chains of negative dimension.
\end{remark}

\subsection{Admissible Lagrangian submanifolds in the product}\label{section: wrapped Floer theory in the product}
	In some sense, the main concern of this paper is with the wrapped Floer cohomology in the product Liouville manifold $M \times N$. The fact that the Lagrangian submanifolds in consideration are in general non-compact brings up a lot of complication and difficulty. For this reason we restrict to specific classes of Lagrangian submanifolds in the product $M \times N$. \par
	
\begin{definition}\label{def: admissible Lagrangian submanifolds in the product manifold}
	A Lagrangian submanifold $\mathcal{L} \subset M \times N$ is called admissible for wrapped Floer thoery, if it is one of the following kind:
\begin{enumerate}[label=(\roman*)]

\item either a conical Lagrangian submanifold $\mathcal{L}$ with respect to the natural cylindrical end $\Sigma \times [1, +\infty)$;

\item or a product $L \times L'$ of conical Lagrangian submanifolds $L \subset M, L' \subset N$.

\end{enumerate}
\end{definition}

	We will later allow wider classes of Lagrangian submanifolds when studying geometric compositions of Lagrangian correspondences, but focus on the above classes throughout this section. \par
	In this subsection, we discuss the first class of Lagrangian submanifolds. Obviously, the wrapped Floer cohomology $HW^{*}(\mathcal{L}_{0}, \mathcal{L}_{1}; K, J)$ for a pair of conical Lagrangian submanifolds with respect to the quadratic Hamiltonian $K$ and the contact-type almost complex structure is well-defined. Thus the only thing to be discussed is the wrapped Floer cohomology of $(\mathcal{L}_{0}, \mathcal{L}_{1})$ with respect to the split Hamiltonian $H_{M, N}$ and the product almost complex structure $J_{M, N}$. This would require the study of the structure of the moduli space of inhomogeneous pseudoholomorphic strips as usual. While transversality is not a problem for us because of the introduction of time-dependent perturbations (in other words one-parameter families) of the almost complex structures, compactness is not a priori guaranteed. We need to prove:
\begin{enumerate}[label=(\roman*)]

\item For fixed $H_{M, N}$-chords $\gamma_{0}, \gamma_{1}$ from $\mathcal{L}_{0}$ to $\mathcal{L}_{1}$, the bordified moduli space $\bar{\mathcal{M}}(\gamma_{0}, \gamma_{1})$ of inhomogeneous pseudoholomorphic strips connecting $\gamma_{0}$ to $\gamma_{1}$ is compact;

\item For fixed $\gamma_{1}$, $\bar{\mathcal{M}}(\gamma_{0}, \gamma_{1})$ is non-empty for all but finitely many $\gamma_{0}$'s.

\end{enumerate} \par

	The first follows from the standard argument based upon the action-energy equality \eqref{action-energy equality for strips} since the Lagrangians $\mathcal{L}_{0}, \mathcal{L}_{1}$ are exact. \par
	The second requires further detailed estimate on the energy of inhomogeneous pseudoholomorphic strips. Consider any $H_{M, N}$-chord $\gamma$ from $\mathcal{L}_{0}$ to $\mathcal{L}_{1}$. Writing $\gamma = (x, y)$ its components in $M$ and $N$, we deduce that $x$ satisfies the Hamilton's equation for the Hamiltonian vector field $X_{H_{M}}$, while $y$ satisfies the Hamilton's equation for the Hamiltonian vector field $X_{H_{N}}$. We may then compute its action to be
\begin{equation}
\mathcal{A}_{H_{M, N}, \mathcal{L}_{0}, \mathcal{L}_{1}}(\gamma) = \int_{0}^{1} - x^{*}\lambda_{M} - y^{*}\lambda_{N} + H_{M}(x(t))dt + H_{N}(y(t))dt + f_{\mathcal{L}_{1}}(\gamma(1)) - f_{\mathcal{L}_{0}}(\gamma(0)).
\end{equation}
Let us estimate the action in the following four cases:
\begin{enumerate}[label=(\roman*)]

\item $x$ lies in the interior part (the compact domain) of $M$, and $y$ lies in the interior part of $N$. By choosing the Hamiltonian $H_{M}$ (and $H_{N}$ respectively) to be $C^{2}$-small in the interior part of $M$ (and respectively $N$), we can ensure that
\begin{equation}
|\mathcal{A}_{H_{M, N}, \mathcal{L}_{0}, \mathcal{L}_{1}}(\gamma)| \le C,
\end{equation}
for some small constant $C$ depending only on the geometry of the Liouville manifolds $M, N$, the Lagrangian submanifolds, and the chosen Hamiltonians $H_{M}, H_{N}$.

\item $x$ lies in the cylindrical end $\partial M \times [1, +\infty)$ of $M$, while $y$ lies in the interior part of $N$. Since $x$ satisfies the Hamilton's equation for $X_{H_{M}}$, it must lie on some hypersurface $\partial M \times \{R_{1}\}$ for some $R_{1}$. The two integral terms involving $x$ contribute $-R_{1}^{2}$, while the rest two terms satisfy
\begin{equation}
|\int_{0}^{1}-y^{*}\lambda_{N} + H_{N}(y(t))dt| \le C',
\end{equation}
for some universal constant $C'$. Suppose $R_{1}$ is large enough (slightly bigger than $1$) so that $\partial M \times \{R_{1}\} \times N$ is contained in $\Sigma \times [1, +\infty)$. Since the Lagrangian submanifolds $\mathcal{L}_{0}, \mathcal{L}_{1}$ are conical with respect to $\Sigma \times [1, +\infty)$, the primitives $f_{\mathcal{L}_{0}}, f_{\mathcal{L}_{1}}$ are locally constant there. Thus
\begin{equation}
f_{\mathcal{L}_{1}}(\gamma(1)) - f_{\mathcal{L}_{0}}(\gamma(0)) = D
\end{equation}
for some constant $D$ independent of $\gamma$ and depending only on the chosen primitives.

\item $x$ lies in the interior part of $M$, while $y$ lies in the cylindrical end $\partial N \times [1, +\infty)$ of $N$. For the same reason, $y$ lies on some hypersurface $\partial N \times \{R_{2}\}$. This time, the two integral terms involving $y$ contribute $-R_{2}^{2}$, while the rest two terms satisfy
\begin{equation}
|\int_{0}^{1}-x^{*}\lambda_{M} + H_{M}(x(t))dt| \le C'.
\end{equation}
Again, the primitives contribute a universal constant $D$.

\item $x$ lies in $\partial M \times [1, +\infty)$, and $y$ lies in $\partial N \times [1, +\infty)$. Thus $x$ is located on $\partial M \times \{R_{1}\}$ for some $R_{1}$, and $y$ is located on $\partial N \times \{R_{2}\}$ for some $R_{2}$. The four integral terms contribute $-R_{1}^{2} - R_{2}^{2}$, while the primitives contribute the same constant $D$.

\end{enumerate}

	The crucial observation is that, as either $R_{1}$ or $R_{2}$ (or both) increases, the action eventually decreases, and goes to $-\infty$ as either one of $R_{1}, R_{2} \to +\infty$. \par
	Now we are ready to prove the second compactness result. Suppose on the contrary for fixed $\gamma_{1}$ there are infinitely many $\gamma_{0}$'s for which $\bar{\mathcal{M}}(\gamma_{0}, \gamma_{1})$ is non-empty. Since the Hamiltonians $H_{M}, H_{N}$ are generic and non-degenerate, and the Reeb dynamics are assumed to be non-degenerate on both contact manifolds $\partial M$ and $\partial N$, there must be one (and in fact infinitely many) such $\gamma_{0} = (x_{0}, y_{0})$ so that either $x_{0}$ or $y_{0}$ lies in the corresponding cylindrical end. Let us assume that is the case with $x_{0}$, which lies on $\partial M \times \{R_{1}\}$ for a large $R_{1}$. Then the action $\mathcal{A}_{H_{M, N}, \mathcal{L}_{0}, \mathcal{L}_{1}}(\gamma_{0})$ satisfies:
\begin{equation}
-R_{1}^{2} - C' + D \le \mathcal{A}_{H_{M, N}, \mathcal{L}_{0}, \mathcal{L}_{1}}(\gamma_{0}) \le -R_{1}^{2} + C' + D,
\end{equation}
which is sufficiently negative as long as $R_{1}$ is sufficiently large. Then the action-energy equality \eqref{action-energy equality for strips} applied to this case implies any $u \in \bar{\mathcal{M}}(\gamma_{0}, \gamma_{1})$ has negative energy, which is of course impossible. This contradiction implies there can be only finitely many such $\gamma_{0}$'s. \par

\subsection{Products of Lagrangian submanifolds} \label{Wrapped Floer cohomology for product Lagrangian submanifolds}
	We now consider the wrapped Floer cohomology of product Lagrangian submanifolds $L \times L'$ in $M \times N$. Wrapped Floer cohomology for product Lagrangian submanifolds $L \times L'$ in $M \times N$ with respect to the split Hamiltonian $H_{M, N}$ and the product almost complex structure $J_{M, N}$ is easily defined as it amounts to wrapped Floer cohomology of the factors. \par
	Thus the major task is to define the wrapped Floer cohomology of product Lagrangian submanifolds with respect to the quadratic Hamiltonian $K$ and a contact-type almost complex structure $J$. Given conical Lagrangian submanifolds $L \subset M$ and $L' \subset N$, their product $L \times L'$ is still an exact Lagrangian submanifold of $M \times N$. However, regarding the more refined structure, it is not conical with respect to the contact hypersurface $\Sigma$ on the nose. \par
	In order to prove that the wrapped Floer cohomology is well-defined in this setting, we dig back into the usual setup, and provide more refined argument. While transversality is not a problem, the compactness results for the relevant moduli spaces do not follow from the previous argument, as the product Lagrangian submanifold $L \times L'$ is not conical. This is the main difficulty we have to overcome. \par
	There are two possible approaches to resolve this issue. The first one is to observe the existence of a canonical Hamiltonian deformation making the product Lagrangian $L \times L'$ conical. However the behavior at infinity would require further analysis which we will not attempt. The second one is to directly prove compactness results for the moduli space of inhomogeneous pseudoholomorphic strips/disks as before, without perturbing the Lagrangian boundary conditions in a Hamiltonian way. This is the approach we shall take in this paper. \par
	For simplicity, consider the moduli space $\mathcal{M}(\gamma_{0}, \gamma_{1})$ of inhomogeneous pseudoholomorphic strips connecting time-one $K$-chords $\gamma_{0}$ and $\gamma_{1}$ from $L \times L'$ to itself, and $\bar{\mathcal{M}}(\gamma_{0}, \gamma_{1})$ its bordification. Assume the Hamiltonian $K$ on the product $M \times N$ is generic. The desired compactness results necessary for the wrapped Floer cohomology to be defined are the following. \par

\begin{proposition}
\begin{enumerate}[label=(\roman*)]

\item For fixed $\gamma_{0}$ and $\gamma_{1}$, the moduli space $\bar{\mathcal{M}}(\gamma_{0}, \gamma_{1})$ is compact;

\item For fixed $\gamma_{1}$, then $\bar{\mathcal{M}}(\gamma_{0}, \gamma_{1})$ is empty for all but finitely many $\gamma_{0}$'s.

\end{enumerate}
\end{proposition}
\begin{proof}
	To prove these, we can in fact use the same argument as before. That is, we appeal to the action-energy equality \ref{action-energy equality for strips} that expresses the energy $E(u)$ of a strip in terms of the action of the asymptotic Hamiltonian chords. The proof of (i) can be immediately finished: any inhomogeneous pseudoholomorphic strip connecting fixed $x_{0}, x_{1}$ has fixed amount of energy $E(u)$, and that continues to hold for broken strips. Thus Gromov compactness theorem applies. \par
	 The proof of (ii) requires more detailed a priori estimate on the energy. First recall the formula for computing the action of a Hamiltonian chord. Writing $\gamma = (x, y)$ in components of the map $[0, 1] \to M \times N$, that formula can be rewritten as
\begin{equation}
\mathcal{A}_{K, L \times L', L \times L'}(\gamma) = \int_{0}^{1}-\gamma^{*}(\lambda_{M} \times \lambda_{N}) + K(\gamma(t))dt + f_{L}(x(1)) - f_{L}(x(0)) + f_{L'}(y(1)) - f_{L'}(y(0)).
\end{equation}
Suppose $\gamma$ is an $K$-chord that lies in the cylindrical end $\Sigma \times [1, +\infty)$. Then the Hamilton's equation implies that it must lie on some hypersurface $\Sigma \times \{R\}$ for some $R \ge 1$. Thus the first two terms contribute $-R^{2}$. To estimate the last four terms, we recall that the choices of primitives $f_{L}, f_{L'}$ are made so that they are locally constant on the conical ends $\partial L \times [1, +\infty)$ and $\partial L' \times [1, +\infty)$ respectively. In addition, we have chosen extensions of these functions to $M$ and respectively $N$ so that they are uniformly bounded, say by a constant $C$ that depends only on the Liouville structures on $M, N$ and the Lagrangian submanifolds $L, L'$. Thus, we have that
\begin{equation}
|f_{L}(x(1)) - f_{L}(x(0)) + f_{L'}(y(1)) - f_{L'}(y(0))| \le 4C.
\end{equation}
Now if the $K$-chord $\gamma$ lies far away from the compact domain, i.e. if $R$ is sufficiently large, the action $\mathcal{A}_{K, L \times L'}(\gamma)$ is sufficiently negative. \par
	Now we can use this estimate to prove the second statement. Because of the genericity assumption on $K$, if there were infinitely many $\gamma_{0}$'s for which $\bar{\mathcal{M}}(\gamma_{0}, \gamma_{1})$ is non-empty, then there must be such a $\gamma_{0}$ which lie on some hypersurface $\Sigma \times \{R\}$ for a sufficiently large $R$. Since we have fixed $\gamma_{1}$, its action is fixed, say $D$. Thus for any broken inhomogeneous pseudoholomorphic strip $u$ connecting $\gamma_{0}$ to $\gamma_{1}$, its energy satisfies
\begin{equation}
E(u) \le -R^{2} + 4C + D,
\end{equation}
which is negative if $R \gg 0$. This is a contradiction and thus the proof is complete. \par
\end{proof}

	With the desired compactness results proved, we may now conclude that the wrapped Floer differential of $L \times L'$ with respect to $K$ (and an admissible almost complex structure) is well-defined in the usual sense. The same argument can be applied to show that the wrapped Floer cohomology of a pair $(L_{1} \times L'_{1}, L_{2} \times L'_{2})$ is well-defined, with respect to the data $(K, J_{t})$. \par
	Finally, the remaining part of the story in the product manifold is to define wrapped Floer cohomology for a pair $(L \times L', \mathcal{L})$ where $\mathcal{L}$ is conical with respect to the cylindrical end $\Sigma \times [1, +\infty)$. As in the previous cases, we shall prove this by providing a priori estimate on the energy of an inhomogeneous pseudoholomorphic strip $u$ connecting a pair of $K$-chords $\gamma_{0}, \gamma_{1}$ from $L \times L'$ to $\mathcal{L}$. Note the action of any such a $K$-chord $\gamma = (x, y)$ is
\begin{equation}
\mathcal{A}_{K, L \times L', \mathcal{L}}(\gamma) = \int_{0}^{1}-\gamma^{*}(\lambda_{M} \times \lambda_{N}) + K(\gamma(t))dt + f_{\mathcal{L}}(\gamma(1)) - f_{L}(x(0)) - f_{L'}(y(0)).
\end{equation}
Again, the last three terms are uniformly bounded. For a $K$-chord lying on some hypersurface $\Sigma \times [1, +\infty)$, the first two terms contribute $-R^{2}$. This is good enough for the energy estimate in order to prove the desired compactness statements. This proves $HW^{*}(L \times L', \mathcal{L}; K, J_{t})$ is well-defined. The same argument also works for $HW^{*}(\mathcal{L}, L \times L'; K, J_{t})$. \par

\subsection{Invariance of wrapped Floer cohomology in the product}
	Having said a lot about the definitions of various Lagrangian submanifolds in the product Liouville manifold $M \times N$, in particular defined two versions of wrapped Floer cohomology, one with respect to $(H_{M, N}, J_{M, N})$ and the other with respect to $(K, J)$,  we question are isomorphic. The main result of this section is to remove this anbiguity and give affirmative answer on the level of cohomology: the split Hamiltonian defines isomorphic wrapped Floer cohomology to the one defined by admissible Hamiltonian. \par

	The proof of this theorem is rather technical and is deferred to section \ref{section: technical proofs} We call the above cochain map the action-restriction map. We defer the construction and the proof to the last section of the paper in which we provide necessary estimates involved. Benefiting from this theorem, we can freely talk about wrapped Floer cohomology of Lagrangian submanifolds in the product Liouville manifolds, thinking of the Hamiltonian function being either a split one or an admissible one. In other words, it makes the definition of wrapped Floer cohomology unambiguous when considering products of Liouville manifolds. This point is extremely helpful in the study of the quilted version of wrapped Floer cohomology, the construction of functors associated to Lagrangian correspondences, and the isomorphism of wrapped Floer cohomology under geometric compositions of Lagrangian correspondences, as we will discuss in upcoming sections. \par
	According to the discussion, as an immediate corollary of Theorem \ref{invariance for wrapped Floer cohomology in the product} we have the following K\"{u}nneth-type formula for product Lagrangian submanifolds. \par

\begin{corollary}
Let $L_{0}, L_{1} \subset M$ and $L'_{0}, L'_{1} \subset N$ be admissible Lagrangian submanifolds. Then there is a canonical isomorphism
\begin{equation}
HW^{*}(L_{0} \times L_{1}, L'_{0} \times L'_{1}; K, J) \cong HW^{*}(L_{0}, L_{1}; H_{M}, J_{M}) \otimes HW^{*}(L'_{0}, L'_{1}; H_{N}, J_{N})
\end{equation}
\end{corollary}
\begin{proof}
	Note that a time-one chord from $L_{0} \times L_{1}$ to $L'_{0} \times L'_{1}$ corresponds to a pair of chords $(\gamma, \gamma')$, where $\gamma$ is an $H_{M}$ chord from $L_{0}$ to $L_{1}$, and $\gamma'$ is an $H_{N}$-chord from $L'_{0}$ to $L'_{1}$, we get an obvious identification of graded modules
\begin{equation}
CW^{*}(L_{0} \times L_{1}, L'_{0} \times L'_{1}; H_{M, N}, J_{M, N}) = CW^{*}(L_{0}, L_{1}; H_{M}, J_{M}) \otimes CW^{*}(L'_{0}, L'_{1}; H_{N}, J_{N}).
\end{equation} \par
	We claim this is in fact an isomorphism of cochain complexes, taking the differential on $CW^{*}(L_{0} \times L_{1}, L'_{0} \times L'_{1}; H_{M, N}, J_{M, N})$ to the tensor product differential $m^{1}_{L_{0}, L_{1}} \otimes 1 + 1 \otimes m^{1}_{L'_{0}, L'_{1}}$ on the right hand side. To see this, we examine the Floer's equation on $M \times N$, and find it is equivalent to a pair of Floer's equations on $M$ and $N$ respectively, precisely because the Hamiltonian and the almost complex structure are of split type. Therefore we get a canonical one-to-one correspondence of solutions to Floer's equations, thus a canonical identification of parametrized moduli spaces of inhomogeneous pseudoholomorphic strips:
\begin{equation}
\tilde{\mathcal{M}}((\gamma_{0}, \gamma'_{0}), (\gamma_{1}, \gamma'_{1}) = \tilde{\mathcal{M}}(\gamma_{0}, \gamma_{1}) \times \tilde{\mathcal{M}}(\gamma'_{0}, \gamma'_{1}).
\end{equation}
A degree computation shows that one-dimensional component of $\tilde{\mathcal{M}}((\gamma_{0}, \gamma'_{0}), (\gamma_{1}, \gamma'_{1})$ corresponds to the union of the following two kinds of product moduli spaces:
\begin{enumerate}[label = (\roman*)]
\item the one-dimensional component of $\tilde{\mathcal{M}}(\gamma_{0}, \gamma_{1})$ times the zero-dimensional component of $\tilde{\mathcal{M}}(\gamma'_{0}, \gamma'_{1})$;

\item the zero-dimensional component of $\tilde{\mathcal{M}}(\gamma_{0}, \gamma_{1})$ times the one-dimensional component of $\tilde{\mathcal{M}}(\gamma'_{0}, \gamma'_{1})$.
\end{enumerate}
But zero-dimensional moduli spaces of solutions to Floer's equation consist of constant maps, thereby giving rise to identity maps. This proves the claim, and therefore proves the isomorphism
\begin{equation*}
HW^{*}(L_{0} \times L_{1}, L'_{0} \times L'_{1}; H_{M, N}, J_{M, N}) \cong HW^{*}(L_{0}, L_{1}; H_{M}, J_{M}) \otimes HW^{*}(L'_{0}, L'_{1}; H_{N}, J_{N}).
\end{equation*}
The rest follows from Theorem \ref{invariance for wrapped Floer cohomology in the product}.
\end{proof}

	In addition to the cohomology groups, we also the multiplicative structure on wrapped Floer cohomology. The main result is Theorem \ref{preserving multiplicative structure}, which in words states that the multiplicative structures defined using the split Hamiltonian and the admissible Hamiltonian coincide on the level of cohomology, after we identify the two wrapped Floer cohomology groups as in Theorem \ref{invariance for wrapped Floer cohomology in the product}. \par
	The proof involves more careful arrangement of the argument for Theorem \ref{invariance for wrapped Floer cohomology in the product}, which is also given in the last section. See particularly section \ref{section: proof of preserving multiplicative structure}. \par

\section{Wrapped Floer cohomology for Lagrangian correspondences}\label{wrapped Floer theory for Lagrangian correspondences}

\subsection{Quilted surfaces}
	Quilted surfaces naturally appear as basic geometric objects for generalizing Floer cohomology groups to the case of Lagrangian correspondences, as studied extensively in \cite{Wehrheim-Woodward1}, \cite{Wehrheim-Woodward2}, \cite{Wehrheim-Woodward3}, \cite{Wehrheim-Woodward4} in the case of compact monotone Lagrangian correspondences in compact monotone symplectic manifolds, or compact exact Lagrangian correspondences in the interior of compact exact symplectic manifolds with boundary and corners. Roughly speaking, a quilted surface is a collection of surfaces (called patches) seamed together at certain pairs of boundary components of the patches. We review the formal definition here. \par

\begin{definition}
	A quilted surface $\underline{S}$ of genus zero is a collection of bordered Riemann surfaces with genus zero $\{S_i\}_{i=1, \cdots, m}$ with boundary punctures, satisfying the following conditions:

\begin{enumerate}[label = (\alph*)]

\item Every surface $S_{i}$ in the collection, which we call a patch of $\underline{S}$, comes with chosen strip-like ends $\epsilon_{i, j}$ near the boundary punctures $z_{i, j}$, where $j \in P(S_i)$ labels the boundary punctures of $S_i$ in cyclic order. Let $I_{i, j}$ be the non-compact boundary component of $S_i$ between $z_{i, j-1}$ and $z_{i, j}$.

\item There is a collection $\{((i_{\sigma}, I_{\sigma}), ((i'_{\sigma}, I'_{\sigma}))\}_{\sigma \in \mathcal{S}}$ of seams (which we will also denote by $\mathcal{S}$ by abuse of notation), whose elements are pairs of elements in $\cup_{i=1}^{m} \{i\} \times \pi_{0}(\partial S_{i})$ such that for each $\sigma \in \mathcal{S}$, there is a diffeomorphism of boundary components:
$$\phi_{\sigma}: \partial S_{i_{\sigma}} \supset I_{\sigma} \to I'_{\sigma} \subset \partial S_{i'_{\sigma}},$$
which is real analytic and compatible with the chosen strip-like ends. Here real analyticity means that at each point $z \in I_{\sigma}$, there exists a neighborhood $U$ of $z$ in $S_{\sigma}$ such that $\phi_{\sigma}|_{U \cap I_{\sigma}}$ extends to an embedding of $U$ into $S_{i'\sigma}$ which reverses the complex structures. Compatibility with the strip-like ends means that if $I_{\sigma}$ lies in between two punctures (hence does $I'_{\sigma}$), namely $I_{\sigma} = I_{i_{\sigma}, j_{\sigma}}, I'_{\sigma} = I_{i'_{\sigma}, j'_{\sigma}}$, then $\phi_{\sigma}$ has to match up the end $\epsilon_{i_{\sigma}, j_{\sigma}}$ with the end $\epsilon_{i'_{\sigma}, j'_{\sigma}-1}$ and the end $\epsilon_{i_{\sigma}, j_{\sigma}-1}$ with the end $\epsilon_{i'_{\sigma}, j'_{\sigma}}$.

\end{enumerate}

\end{definition}

	We can label by $\mathcal{B}$ the remaining boundary components, which are not identified in the above process and are called true boundary components, as below:
$$\{(i_{b}, I_{b})\}_{b \in \mathcal{B}} = \cup_{i=1}^{m} \{i\} \times \pi_{0}(\partial S_{i}) \setminus \mathcal{S}.$$
We will again denote this set by $\mathcal{B}$ by abuse of notation. The quilted surface $\underline{S}$ has two kinds of ends: strip-like ends near punctures between true boundary components, and quilted ends. By definition, a quilted end is a maximal sequence $\underline{\epsilon} = (\epsilon_{i_{k}, j_{k}})_{k=1, \cdots, n_{\underline{\epsilon}}}$ of ends of patches with boundaries identified $\epsilon_{i_{k}, j_{k}}(\cdot, \delta_{i_{k}, j_{k}}) \cong \epsilon_{i_{k+1}, j_{k+1}}(\cdot, 0)$ via some seam. We call this quilted end a positive quilted end if all strip-like ends in this sequence are positive ends, otherwise a negative one (all strip-like ends in a quilted end have to be either all postive or all negative, by definition of seaming). \par
	There are also two different kinds of quilted ends. One is cyclic, meaning that the upper boundary $I_{i_{1}, j_{1}-1}$ of the first strip-like end $\epsilon_{i_{1}, j_{1}-1}$ is seamed together with the lower boundary $I_{i_{m}, j_{m}}$ of the last strip-like end $\epsilon_{i_{m}, j_{m}}$. The other is non-cyclic, meaning that these two boundaries are not seamed. \par
	Wehrheim and Woodward proved in \cite{Wehrheim-Woodward3} that any two quilted surfaces of the same combinatorial type are smoothly homotopic. Namely,

\begin{lemma}
Let $\underline{S}_{0}$ and $\underline{S}_{1}$ be two quilted surfaces of the same combinatorial type. Then there exists a smooth family $(\underline{S}_{t})_{t \in [0, 1]}$ of quilted surfaces connecting the two. Here smoothness means that the complex structures on patches, the strip-like ends, and the seaming maps depend smoothly on $t$.
\end{lemma}

	This lemma is then used to prove that Floer-theoretic invariants defined using quilted surfaces are independent of these auxiliary data involved in the definition of quilted surfaces, up to homotopy. \par

\subsection{Inhomogeneous pseudoholomorphic quilts}
	Let us begin by briefly recalling the local theory of moduli spaces of inhomogeneous pseudoholomorphic quilts, basically taken from \cite{Wehrheim-Woodward3}. Suppose that we have a collection of Liouville manifolds $\underline{M} = (M_{0}, M_{1}, \cdots, M_{m})$. Let $\underline{S}$ be a quilted surface. A Lagrangian label $\underline{\mathcal{L}}$ for $\underline{S}$ in $\underline{M}$ consists of the following:

\begin{enumerate}[label = (\alph*)]

\item a collection of Lagrangian correspondences $\mathcal{L}_{\sigma} \subset M_{i_{\sigma}}^{-} \times M_{i'_{\sigma}}$ for $\sigma \in \mathcal{S}$ associated to the seams;

\item a collection of Lagrangian submanifolds $L_{b} \subset M_{i_{b}}$ for $b \in \mathcal{B}$ associated to true boundary components. 

\end{enumerate}

	We will assume that all Lagrangian submanifolds/correspondences are closed exact submanifolds in the interior of the relevant Liouville manifolds, or conical over the natural cylindrical ends. For each quilted end $\underline{\epsilon} = (\epsilon_{i_{k}, j_{k}})_{k=1, \cdots, n_{\underline{\epsilon}}}$, we can associate to it the sequence of Lagrangian correspondences that label the seams and true boundary components that run into the strip-like ends $(\epsilon_{i_{k}, j_{k}})_{k=1, \cdots, n_{\underline{\epsilon}}}$. There are two slightly different cases:

\begin{enumerate}[label = (\roman*)]

\item If $\underline{\epsilon}$ is cyclic, then associated to it is a sequence of Lagrangian correspondences $\mathcal{L}_{\sigma_{1}} \subset M_{i_{1}}^{-} \times M_{i_{2}}, \cdots, \mathcal{L}_{\sigma_{n_{\underline{\epsilon}}-1}} \subset M_{i_{n_{\underline{\epsilon}}-1}}^{-} \times M_{i_{n_{\underline{\epsilon}}}}, \mathcal{L}_{\sigma_{n_{\underline{\epsilon}}}} \subset M_{i_{n_{\underline{\epsilon}}}}^{-} \times M_{i_{1}}^{-}$;

\item If $\underline{\epsilon}$ is non-cyclic, then associated to it is two Lagrangian submanifolds $L_{b_{0}} \subset M_{i_{1}}, L_{b_{n_{\underline{\epsilon}}}} \subset M_{i_{n_{\underline{\epsilon}}}}$, as well as a sequence of Lagrangian correspondences $\mathcal{L}_{\sigma_{1}} \subset M_{i_{1}}^{-} \times M_{i_{2}}, \cdots, \mathcal{L}_{\sigma_{n_{\underline{\epsilon}}-1}} \subset M_{i_{n_{\underline{\epsilon}}-1}}^{-} \times M_{i_{n_{\underline{\epsilon}}}}$. If is helpful to think of $L_{b_{n_{\underline{\epsilon}}}} \subset M_{i_{n_{\underline{\epsilon}}}}^{-}$ so that $L_{b_{n_{\underline{\epsilon}}}} \times L_{b_{0}}$ is also regarded as a Lagrangian correspondence.

\end{enumerate}
We will call a sequence of Lagrangian correspondences a generalized Lagrangian correspondence, and one of the type in $(a)$ a cyclic generalized Lagrangian correspondence, and one of the type in $(b)$ a non-cyclic generalized Lagrangian correspondence. \par
	Now suppose admissible Hamiltonians $H_{M_{i}}$ and one-parameter families of contact-type almost complex structures $J_{M_{i}}$ on $M_{i}$ have been chosen for all Liouville manifolds in the collection. Given a quilted surface $\underline{S}$ with Lagrangian labels, choose a quilted Floer datum $(\underline{\alpha}_{\underline{S}}, \underline{H}_{\underline{S}}, \underline{J}_{\underline{S}})$ for $\underline{S}$, which by definition consists of Floer data on all patches. \par
	
\begin{definition}
	An inhomogeneous pseudoholomorphic quilt with Lagrangian seam and boundary conditions is a tuple of maps $\underline{u}: \underline{S} \to \underline{M}$ such that
\begin{enumerate}[label = (\alph*)]

\item each patch $u_{i}: S_{i} \to M_{i}$ satisfies the inhomogeneous Cauchy-Riemann equation:
\begin{equation}
(du_{i} - \alpha_{S_{i}} \otimes X_{H_{S_{i}}})^{0, 1} = 0;
\end{equation}

\item $\underline{S}$ maps seams to the given Lagrangian correspondences, and true boundary components to the given Lagrangian submanifolds.
\end{enumerate}
\end{definition}

	The energy of an inhomogeneous pseudoholomorphic quilt is defined to be the sum of the energy of every patch. The asymptotic condition for an inhomogeneous pseudoholomorphic quilt at a strip-like end is just a chord. The asymptotic condition for a pseudoholomorphic quilt at a quilted end is a generalized chord, which we define below. \par

\begin{definition}
Let $\underline{S}$ be a quilted surface and let $\underline{\epsilon} = (\epsilon_{i_{k}, j_{k}})_{k=1, \cdots, n_{\underline{\epsilon}}}$ be a quilted end. There are two cases:
\begin{enumerate}[label = (\roman*)]

\item If $\underline{\epsilon}$ is cyclic, then we have a cyclic generalized Lagrangian correspondence as before. A generalized chord for this generalized Lagrangian correspondence is a tuple $\underline{\gamma} = (\gamma_{1}, \cdots, \gamma_{n_{\underline{\epsilon}}})$ where $\gamma_{k}$ is a $H_{M_{i_{k}}}$-chord in $M_{i_{k}}$ such that
\begin{equation}
(\gamma_{k}(1), \gamma_{k+1}(0)) \in \mathcal{L}_{\sigma_{k}}.
\end{equation}

\item If $\underline{\epsilon}$ is non-cyclic, then we have a non-cyclic generalized Lagrangian correspondence as before. A generalized chord for this generalized Lagrangian correspondence is a tuple $\underline{\gamma} = (\gamma_{0}, \gamma_{1}, \cdots, \gamma_{n_{\underline{\epsilon}}})$ where $\gamma_{k}$ is a $H_{M_{i_{k}}}$-chord in $M_{i_{k}}$ such that
\begin{equation}
\gamma_{0}(0) \in L_{b_{0}}, \gamma_{n_{\underline{\epsilon}}}(1) \in L_{b_{n_{\underline{\epsilon}}}},
\end{equation}
\begin{equation}
(\gamma_{k}(1), \gamma_{k+1}(0)) \in \mathcal{L}_{\sigma_{k}}, k=1, \cdots, n_{\underline{\epsilon}} - 1.
\end{equation}

\end{enumerate}

\end{definition}

	Suppose that Lagrangian labels for $\underline{S}$ have been chosen, and that associated to each positive (resp. negative) end $\underline{\epsilon}^{\pm}$ of $\underline{S}$ the asymptotic condition (a generalized chord $\underline{\gamma}_{\underline{\epsilon}}^{\pm}$) has been fixed. We consider an inhomogeneous pseudoholomorphic quilt $\underline{u}: \underline{S} \to \underline{M}$ with Lagrangian seam and boundary conditions, satisfying the prescribed asymptotic conditions over the quilted ends. Moreover, we assume that the energy of $\underline{u}$ is finite. The linearized Cauchy-Riemann operator at $\underline{u}$, which we denote by $D_{\underline{u}}\partial_{\underline{H}_{\underline{S}}, \underline{J}_{\underline{S}}}$, is defined to be the direct sum of the usual linearized Cauchy-Riemann operators on patches. $D_{\underline{u}}\partial_{\underline{H}_{\underline{S}}, \underline{J}_{\underline{S}}}$ is a Fredholm operator, basically because the linearized Cauchy-Riemann operator on each patch is Fredholm. For more details, see Remark 3.8 of \cite{Wehrheim-Woodward3}. \par
	Let $\underline{\Gamma}^{-}, \underline{\Gamma}^{+}$ be the collection of all generalized chords at negative ends and positive ends respectively. Let $\mathcal{M}_{\underline{S}}(\underline{\Gamma}^{-}; \underline{\Gamma}^{+})$ be the moduli space of inhomogeneous pseudoholomorphic quilts with Lagrangian seam and boundary conditions. Transversality can be proved in the same way as before. It is also not difficult to find a compactification of this moduli space which is a smooth manifold with boundary and corners, since the symplectic manifolds are Liouville manifolds and the Lagrangian submanifolds are all exact. We choose not to enter the full details here for general quilted surface $\underline{S}$ because we do not quite need the general result, and the notations could become very complicated. Also, in fact the above construction has a flaw for the purpose of defining invariants or operations of wrapped Floer cohomology - we have not taken time-shifting functions into account, which, unlike the case of closed manifolds, is absolutely necessary because Hamiltonians add up for example when we try to multiply inputs, but continuation maps do not generally induce cochain homotopy equivalence in the case of non-compact manifolds. Doing so for all kinds of quilted surfaces would require significantly amount of additional notations which we do not bother writing down. Instead, we will deal with only the quilted surfaces that are relevant to our constructions later, about which we will be more precise. \par
	Finally, we present a general action-energy equality for inhomogeneous pseudoholomorphic quilts, which will be extremely useful for proving compactness results. The action of a generalized chord $\underline{\gamma}$ is defined to the the sum of the action of the Hamiltonian chords as its components. The action-energy equality for inhomogeneous pseudoholomorphic quilts has the following form
\begin{equation} \label{action-energy equality for quilts}
E(\underline{u}) = \sum_{\underline{\gamma}^{-} \in \underline{\epsilon}^{-}}\mathcal{A}(\underline{\gamma}^{-}) - \sum_{\underline{\gamma}^{+} \in \underline{\epsilon}^{+}}\mathcal{A}(\underline{\gamma}^{+}),
\end{equation}
where the first sum is over all generalized chords as asymptotic values of $\underline{u}$ at negative quilted ends, and the second sum is over all generalized chords as asymptotic values of $\underline{u}$ at positive quilted ends. The proof is simply an application of Stokes' theorem, applied to each patch of the quilted surface. \par

\subsection{Quilted wrapped Floer cohomology} \label{section: definition of quilted wrapped Floer cohomology}
	Using quilted surfaces, we may extend the definition of the wrapped Floer cohomology to a quilted version. The basic case is that we have conical Lagrangian submanifolds $L \subset M, L' \subset N$, as well as an admissible Lagrangian correspondence $\mathcal{L} \subset M^{-} \times N$. Although having introduced general quilts in the previous subsection, proving well-definedness of wrapped Floer cohomology in general encounters a lot of complication and difficulty, and therefore we will essentially only focus on this basic case. The generators for the quilted wrapped Floer cochain complex $CW^{*}(L, \mathcal{L}, L')$ are time-one generalized chords, defined as follows. \par

\begin{definition}
	A time-one generalized $(H_{M}, H_{N})$-chord for the triple $(L, \mathcal{L}, L')$ is a pair $(\gamma^{0}, \gamma^{1})$ where $\gamma^{0}$ is a $H_{M}$-chord and $\gamma^{1}$ is a $H_{N}$-chord such that $\gamma^{0}(0) \in L, (\gamma^{0}(1), \gamma^{1}(0)) \in \mathcal{L}, \gamma^{1}(1) \in L'$. We denote by $\mathcal{X}_{(H_{M}, H_{N})}(L, \mathcal{L}, L')$ the set of all generalized $(H_{M}, H_{N})$-chords for the triple $(L, \mathcal{L}, L')$.
\end{definition}

	The notion of absolute Maslov index can be easily extended to generalized chords, which will provide gradings on the quilted version of  wrapped Floer cochain complexes. The way of defining this index is to regard a generalized chord as a $H_{M, N}$-chord in the product manifold $M^{-} \times N$ going from $L \times L'$ to $\mathcal{L}$. For each generalized chord $(\gamma^{0}, \gamma^{1})$, there is also a canonical orientation line associated to it, denoted by $|o_{(\gamma^{0}, \gamma^{1})}|$, defined in a similar way by regarding this generalized chord as a chord in the product manifold. We will not repeat these kinds of definitions later. \par
	Now consider the (parametrized) moduli space of quilted pseudoholomorphic strips $\tilde{\mathcal{M}}((\gamma^{0}_{0}, \gamma^{1}_{0}), (\gamma^{0}_{1}, \gamma^{1}_{1}))$ which converge exponentially to $(\gamma^{0}_{0}, \gamma^{1}_{0})$ at the negative quilted end, and to $(\gamma^{0}_{1}, \gamma^{1}_{1})$ at the positive quilted end, with Lagrangian boundary conditions $(L, \mathcal{L}, L')$. The standard transversality argument shows that for generic one-parameter families of almost complex structures $J_{M}$ and $J_{N}$ of contact-type, $\tilde{\mathcal{M}}((\gamma^{0}_{0}, \gamma^{1}_{0}), (\gamma^{0}_{1}, \gamma^{1}_{1}))$ is a smooth manifold of dimension $\deg((\gamma^{0}_{0}, \gamma^{1}_{0})) - \deg((\gamma^{0}_{1}, \gamma^{1}_{1}))$. There is a natural $\mathbb{R}$-action on $\tilde{\mathcal{M}}((\gamma^{0}_{0}, \gamma^{1}_{0}), (\gamma^{0}_{1}, \gamma^{1}_{1}))$, by simultaneous translation on the two patches (which are strips) of the quilted strip (due to the seam condition, only simultaneous translations are allowed). This action is free whenever $\deg((\gamma^{0}_{0}, \gamma^{1}_{0})) \ne \deg((\gamma^{0}_{1}, \gamma^{1}_{1}))$. Call the quotient $\mathcal{M}((\gamma^{0}_{0}, \gamma^{1}_{0}), (\gamma^{0}_{1}, \gamma^{1}_{1}))$ the moduli space of (unparametrized) quilted pseudoholomorphic strips, which is a smooth manifold of dimension $\deg((\gamma^{0}_{0}, \gamma^{1}_{0})) - \deg((\gamma^{0}_{1}, \gamma^{1}_{1})) - 1$. \par
	There is a natural compactification $\bar{\mathcal{M}}((\gamma^{0}_{0}, \gamma^{1}_{0}), (\gamma^{0}_{1}, \gamma^{1}_{1}))$, similar to the one in \eqref{compactification of moduli space of strips}:
\begin{equation}
\begin{split}
\bar{\mathcal{M}}((\gamma^{0}_{0}, \gamma^{1}_{0}), (\gamma^{0}_{1}, \gamma^{1}_{1})) &= \coprod \mathcal{M}((\gamma^{0}_{0}, \gamma^{1}_{0}), (\theta^{0}_{1}, \theta^{1}_{1})) \times \mathcal{M}((\theta^{0}_{1}, \theta^{1}_{1}), (\theta^{0}_{2}, \theta^{1}_{2}))\\
& \times \cdots \times \mathcal{M}((\theta^{0}_{k-1}, \theta^{1}_{k-1}), (\theta^{0}_{k}, \theta^{1}_{k})) \times \mathcal{M}((\theta^{0}_{k}, \theta^{1}_{k}), (\gamma^{0}_{1}, \gamma^{1}_{1})).
\end{split}
\end{equation}

	Similar to Lemma \ref{regularity + compactness} and Lemma \ref{compactness 2}, we have the following two lemmata whose proofs utilize the action-energy equality on $M^{-} \times N$, and are essentially the same as the previous ones. \par

\begin{lemma}
	For generic one-parameter families of almost complex structures $J_{M}$ and $J_{N}$ of contact-type, the moduli space $\bar{\mathcal{M}}((\gamma^{0}_{0}, \gamma^{1}_{0}), (\gamma^{0}_{1}, \gamma^{1}_{1}))$ is a compact smooth manifold with boundary and corners of dimension $\deg((\gamma^{0}_{0}, \gamma^{1}_{0})) - \deg((\gamma^{0}_{1}, \gamma^{1}_{1})) - 1$.
\end{lemma}

\begin{lemma}\label{compactness for quilted strips}
	For each $(\gamma^{0}_{1}, \gamma^{1}_{1})$, $\mathcal{M}((\gamma^{0}_{0}, \gamma^{1}_{0}), (\gamma^{0}_{1}, \gamma^{1}_{1}))$ is empty for all but finitely many $(\gamma^{0}_{0}, \gamma^{1}_{0})$.
\end{lemma}

	We then define a graded $\mathbb{Z}$-module
\begin{equation}	
CW^{*}(L, \mathcal{L}, L'; H_{M}, H_{N}, J_{M}, J_{N}) = \bigoplus_{(\gamma^{0}, \gamma^{1}) \in \mathcal{X}_{(H_{M}, H_{N})}(L, \mathcal{L}, L')} |o_{(\gamma^{0}, \gamma^{1})}|,
\end{equation}
which we call quilted wrapped Floer cochain complex. The differential $m^1 = m^1_{H_{M}, H_{N}, J_{M}, J_{N}}$ is defined to be
\begin{equation}
m^{1}: CW^{i}(L, \mathcal{L}, L'; H_{M}, H_{N}, J_{M}, J_{N}) \to CW^{i+1}(L, \mathcal{L}, L'; H_{M}, H_{N}, J_{M}, J_{N})
\end{equation}
\begin{equation}
m^{1}([(\gamma^{0}_{1}, \gamma^{1}_{1})]) = (-1)^{i} \sum_{\substack{\deg((\gamma^{0}_{0}, \gamma^{1}_{0})) = \deg((\gamma^{0}_{1}, \gamma^{1}_{1})) + 1\\ \underline{u} \in \mathcal{M}((\gamma^{0}_{0}, \gamma^{1}_{0}), (\gamma^{0}_{1}, \gamma^{1}_{1}))}} m^{1}_{\underline{u}}([(\gamma^{0}_{1}, \gamma^{1}_{1})]).
\end{equation}
By Lemma \ref{compactness for quilted strips}, the sum is indeed finite. For the matter of orientation, we define the orientation line of a generalized chord $(\gamma^{0}, \gamma^{1})$ to be the orientation line of the $H_{M, N}$-chord in $M^{-} \times N$ by setting $\Gamma(t) = (\gamma^{0}(1-t), \gamma^{1}(t))$, and define the orientations on the moduli spaces of quilted strips to be those on the moduli space of inhomogeneous pseudoholomorphic strips in the product $M^{-} \times N$. This point of view shall be made more explicit in Lemma \ref{folding quilt lemma}. For more general quilted surfaces, see for example \cite{Wehrheim-Woodward5}, but we shall not try to make use of the results in this paper. \par
	We call the resulting cohomology the quilted wrapped Floer cohomology of $(L, \mathcal{L}, L')$ with respect to $H_{M}, H_{N}, J_{M}, J_{N})$, and denote it by $HW^{*}(L, \mathcal{L}, L'; H_{M}, H_{N}, J_{M}, J_{N})$. As before, standard continuation argument shows that this is independent of choices of generic one-parameter families of admissible almost complex structures, and independent of admissible Hamiltonians $H_{M}$ on $M$ and $H_{N}$ on $N$. \par
	The advantage of using split Hamiltonians on the product Liouville manifold $M^{-} \times N$ is that we can fold the quilted inhomogeneous pseudoholomorphic curve to obtain an inhomogeneous pseudoholomorphic curve with Lagrangian boundary conditions in the product Liouville manifold $M^{-} \times N$, with respect to the data $(H_{M, N}, -J_{M} \times J_{N})$. Thus quilted wrapped Floer cohomology is usually the wrapped Floer cohomology in the product manifold: \par

\begin{lemma} \label{folding quilt lemma}
\begin{equation}
\begin{split}
HW^{*}(L, \mathcal{L}, L'; H_{M}, H_{N}, J_{M}, J_{N}) &\cong HW^{*}(\mathcal{L}, L \times L'; H_{M, N}, -J_{M} \times J_{N})\\
& \cong HW^{*}(L \times L', \mathcal{L}; K, J),
\end{split}
\end{equation}
where $K$ is some admissible Hamiltonian on $M^{-} \times N$, and $J$ is an almost complex structure of contact type with respect to the contact hypersurface $\Sigma$.
\end{lemma}
\begin{proof}
For the first isomorphism, the corresponding Floer cochain complexes are obviously isomorphic as graded modules. Folding a quilted pseudoholomorphic strip defining the differential on $CW^{*}(L, \mathcal{L}, L'; H_{M}, H_{N}, J_{M}, J_{N})$, we get an inhomogeneous pseudoholomorphic strip defining the differential on $CW^{*}(L \times L', \mathcal{L}; H_{M, N}, -J_{M} \times J_{N})$. Putting this "folding quilt" in a more explicit form, we define an inhomogeneous pseudoholomorphic map $w: \mathbb{R} \times [0, 1] \to M^{-} \times N$ by
\begin{equation}
w(s, t) = (u(s, 1-t), v(s, t)),
\end{equation}
from a quilted inhomogeneous pseudoholomorphic map $(u, v)$
This proves that the isomorphism intertwines the differentials. The second isomorphism is established in Theorem \ref{invariance for wrapped Floer cohomology in the product}.
\end{proof}

\begin{remark}
We should have written $L^{t}$ instead of $L$, where $L^{t}$ is the transpose of $L$ (indeed identical to $L$ as a smooth submanifold) and is regarded as a Lagrangian submanifold of $M^{-}$. But we will not bother doing this unless there might be confusion.
\end{remark}

	The above proof provides canonical identification between the moduli spaces, and hence we may take the wrapped Floer cohomology of $(\mathcal{L}, L \times L')$ in the product Liouville manifold $M^{-} \times N$ with respect to the split Hamiltonian $H_{M, N}$ and product almost complex structure $-J_{M} \times J_{N}$, as a definition for the quilted wrapped Floer cohomology $HW^{*}(L, \mathcal{L}, L')$. Also, we can use Theorem \ref{invariance for wrapped Floer cohomology in the product} to give an alternative definition of the quilted wrapped Floer cohomology by wrapped Floer cohomology on product Liouville manifolds with appropriate Lagrangian boundary conditions, using admissible Hamiltonians on the product manifolds. Since using either split Hamiltonians or admissible ones does not matter at least on the level of cohomology, we will drop the symbols for Hamiltonians and almost complex structures in the notation of wrapped Floer cohomology, until section \ref{section: technical proofs}, where we prove Theorem \ref{invariance for wrapped Floer cohomology in the product}. \par

\section{Functoriality in Lagrangian correspondences} \label{section: module-valued functors associated to Lagrangian correspondences}

\subsection{Wrapped Fukaya category on the cohomology level}
	The purpose of this section is to prove Theorem \ref{functors associated to Lagrangian correspondences}, which constructs functors between wrapped Fukaya categories associated to admissible Lagrangian correspondences $\mathcal{L} \subset M^{-} \times N$. As mentioned in the introduction, we work mostly on cohomology level. The main reason is that it is not proved in this paper that the version of the wrapped Fukaya category of the product manifold defined using split Hamiltonians is quasi-isomorphic to the standard one, although it seems quite straightforward to use split Hamiltonians at every place to perform the constructions on the chain level. \par
	Let us denote by $H(\mathcal{W}(M))$ the cohomology category of the wrapped Fukaya category of a Liouville manifold $M$. Its objects are closed exact Lagrangian submanifolds in the interior of $M$, as well as exact Lagrangian submanifolds $L$ that intersection $\partial M$ transversely with boundary being Legendrian submanifolds $l$ of $\partial M$, and over the cylindrical ends are of the form $l \times [1, +\infty)$. All Lagrangian submanifolds are assumed to be equipped with gradings and spin structures. The morphism space between two such Lagrangian submanifolds is the wrapped Floer cohomology group $HW(L_{0}, L_{1})$, defined using an admissible Hamiltonian and an admissible almost complex structure. The composition of morphisms is defined by cohomology level multiplication on wrapped Floer cohomology induced from the cochain level multiplication $m^{2}$. \par
	As for the product Liouville manifold $M^{-} \times N$, the Lagrangian submanifolds to be included in the wrapped Fukaya category are those considered in section \ref{section: wrapped Floer theory in the product}. We have shown that the wrapped Floer cohomology $HW^{*}(\mathcal{L}_{0}, \mathcal{L}_{1})$ of a pair of such Lagrangian submanifolds is well-defined, with respect to either the split Hamiltonian or the quadratic Hamiltonian. These will be morphism spaces in the cohomology category $H(\mathcal{W}(M^{-} \times N))$, and the composition is again the triangle product, which is well-defined by a similar argument. \par

\subsection{Bimodule-valued functor}
	The way we are going to prove Theorem \ref{functors associated to Lagrangian correspondences} is to first construct a bimodule-valued functor:
\begin{equation}\label{bimodule-valued functor}
\Phi: H(\mathcal{W}(M^{-} \times N)) \to (H(\mathcal{W}(M)), H(\mathcal{W}(N)))^{bimod}
\end{equation}
which sends each admissible Lagrangian correspondence $\mathcal{L} \subset M^{-} \times N$ to the bimodule $P_{\mathcal{L}}$ over $(H(\mathcal{W}(M)), H(\mathcal{W}(N)))$ which, to inputs $L \subset M, L' \subset N$, gives output:
\begin{equation}
P_{\mathcal{L}}(L, L') = HW^{*}(L, \mathcal{L}, L').
\end{equation}
Here by a bimodule over $(H(\mathcal{W}(M)), H(\mathcal{W}(N)))$, we mean a right-$H(\mathcal{W}(M))$ and left-$H(\mathcal{W}(N)))$ bimodule, despite the fact that we write $H(\mathcal{W}(M))$ on the left and $H(\mathcal{W}(N))$ on the right. This distinction is important in wrapped Floer theory, as in general there is no natural duality identifying left modules with right modules. \par
	One of the bimodule structure maps is
\begin{equation} \label{bimodule structure}
m^{1|0|1}: HW^{*}(L_{0}, L_{1}) \otimes HW^{*}(L_{1}, \mathcal{L}, L'_{0}) \otimes HW^{*}(L'_{0}, L'_{1}) \to HW^{*}(L_{0}, \mathcal{L}, L'_{1}),
\end{equation}
defined as follows. Let $\underline{S}^{bm}$ be the quilted surface which consists of two patches $S^{bm}_{0}, S^{bm}_{1}$ both of which are disks with one negative boundary puncture $z_{i, 0}$ and two positive boundary punctures $z_{i, 1}, z_{i, 2}$, near which the strip-like ends $\epsilon_{i, j}$ are chosen. The two patches are seamed along $I_{0, 1}$ and $I_{1, 1}$. The quilted surface $\underline{S}^{bm}$ has four ends - one negative quilted end, one positive quilted end and two positive strip-like ends. \par	
	To define inhomogeneous pseudoholomorphic quilts that are suitable for the purpose of defining $m^{1|0|1}$, we will have to use a time-shifting function that controls the moving boundary conditions to take care of the issue that Hamiltonians add up over the ends of the quilted surfaces, similar to the situation involved in defining the multiplication structure. So we choose a time-shifting function $\rho_{\underline{S}^{bm}}$ for the quilted surface, which by definition consists of two time-shifting functions $\rho_{S^{bm}_{i}}, i = 1, 2$, defined on boundary components of $S^{bm}_{i}$, such that they agree over the seam. We require that $\rho_{S^{bm}_{i}}$ takes the value $1$ near the two postive boundary punctures of $S^{bm}_{i}$ and the value $2$ at the negative boundary puncture of $S^{bm}_{i}$.  \par
	Also, we need to choose a Floer datum on the quilted surface $\underline{S}^{bm}$, which consists of Floer data $(\alpha_{S^{bm}_{i}}, H_{S^{bm}_{i}}, J_{S^{bm}_{i}})$ on patches $S^{bm}_{i}$ in the usual sense. Note that $S^{bm}_{0}$ is mapped to $M$ and $S^{bm}_{1}$ to $N$. So these Floer data for the two patches should be understood as families of Hamiltonians and almost complex structures on $M$ and $N$ respectively. \par
	Given $H_{M}$-chord $\gamma$ from $L_{0}$ to $L_{1}$, $H_{N}$-chord $\theta$ from $L'_{0}$ to $L'_{1}$, as well as generalized chords $(\gamma^{0}, \theta^{0})$ for the triple $(L_{0}, \mathcal{L}, L'_{1})$ and $(\gamma^{1}, \theta^{1})$ for $(L_{1}, \mathcal{L}, L'_{0})$, the inhomogeneous Cauchy-Riemann equation for quilted maps $\underline{u}: \underline{S}^{bm} \to (M, N)$ with the prescribed asymptotic behavior has the following form:
\begin{equation}
\begin{cases}
(du_{i} - \alpha_{S^{bm}_{i}} \otimes X_{H_{S^{bm}_{0}}}(u_{0}))^{0, 1} = 0, & i = 1, 2\\
u_{0}(z) \in \phi_{M}^{\rho_{S^{bm}_{0}}(z)}L_{0}, & z \in I_{0, 2}\\
u_{0}(z) \in \phi_{M}^{\rho_{S^{bm}_{0}}(z)}L_{1}, & z \in I_{0, 0}\\
(u_{0}(z), u_{1}(z)) \in (\phi_{M}^{\rho_{S^{bm}_{0}}(z)} \times \phi_{N}^{\rho_{S^{bm}_{1}}(z)})\mathcal{L}, & z \text{ lies on the seam } (I_{0, 1}, I_{1, 1})\\
u_{1}(z) \in \phi_{N}^{\rho_{S^{bm}_{1}}(z)}L'_{1}, & z \in I_{1, 2}\\
u_{1}(z) \in \phi_{N}^{\rho_{S^{bm}_{1}}(z)}L'_{0}, & z \in I_{1, 0}\\
\lim\limits_{s \to -\infty} u_{0} \circ \epsilon_{0, 0}(s, \cdot) = \phi_{M}^{2}\gamma^{0}(\cdot), & \lim\limits_{s \to -\infty} u_{1} \circ \epsilon_{1, 0}(s, \cdot) = \phi_{N}^{2}\theta^{0}(\cdot)\\
\lim\limits_{s \to +\infty} u_{0} \circ \epsilon_{0, 1}(s, \cdot) = \gamma^{1}(\cdot), & \lim\limits_{s \to +\infty} u_{1} \circ \epsilon_{1, 1}(s, \cdot) = \theta^{1}(\cdot)\\
\lim\limits_{s \to +\infty} u_{0} \circ \epsilon_{0, 2}(s, \cdot) = \gamma(\cdot), & \lim\limits_{s \to +\infty} u_{1} \circ \epsilon_{1, 2}(s, \cdot) = \theta(\cdot)\\
\end{cases}
\end{equation}
The picture for such a quilted map is shown in Figure 1. \par

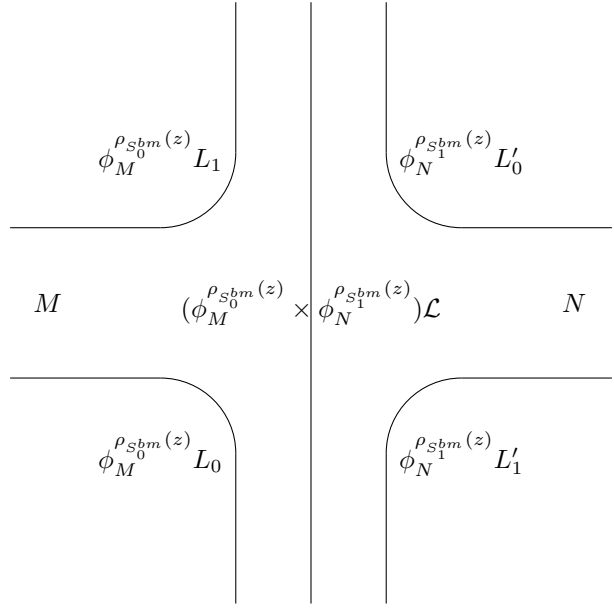
\begin{figure}\label{fig: the quilted map defining bimodule structure}
\centering
\begin{tikzpicture}
	\draw (-4, 1) -- (-2, 1);
	\draw (-2, 1) arc (270:360:1cm);
	\draw (-1, 2) -- (-1, 4);
	\draw (-4, -1) -- (-2, -1);
	\draw (-2, -1) arc (90:0:1cm);
	\draw (-1, -2) -- (-1, -4);
	\draw (0, 4) -- (0, -4);
	\draw (1, 4) -- (1, 2);
	\draw (1, 2) arc (180: 270:1cm);
	\draw (2, 1) -- (4, 1);
	\draw (1, -4) -- (1, -2);
	\draw (1, -2) arc (180:90:1cm);
	\draw (2, -1) -- (4, -1);
	\draw (-2, 2) node {$\phi_{M}^{\rho_{S^{bm}_{0}}(z)}L_{1}$};
	\draw (-2, -2) node {$\phi_{M}^{\rho_{S^{bm}_{0}}(z)}L_{0}$};
	\draw (2, 2) node {$\phi_{N}^{\rho_{S^{bm}_{1}}(z)}L'_{0}$};
	\draw (2, -2) node {$\phi_{N}^{\rho_{S^{bm}_{1}}(z)}L'_{1}$};
	\draw (0, 0) node {$(\phi_{M}^{\rho_{S^{bm}_{0}}(z)} \times \phi_{N}^{\rho_{S^{bm}_{1}}(z)})\mathcal{L}$};
	\draw (-3.5, 0) node {$M$};
	\draw (3.5, 0) node {$N$};
\end{tikzpicture}
\caption{the quilted map defining the bimodule structure map}
\end{figure}

	Let $\mathcal{M}^{bm}((\gamma^{0}, \theta^{0}); \gamma, (\gamma^{1}, \theta^{1}), \theta)$ be the moduli space of inhomogeneous pseudoholomorphic quilts $\underline{u}: \underline{S}^{bm} \to (M, N)$ satisfying the above conditions. As before, we have the following standard transversality result: \par

\begin{lemma}
	For a generic choice of Floer datum for the quilted surface $\underline{S}^{bm}$, the moduli space $\mathcal{M}^{bm}((\gamma^{0}, \theta^{0}); \gamma, (\gamma^{1}, \theta^{1}), \theta)$ is a smooth manifold of dimension
\begin{equation*}
\deg((\gamma^{1}, \theta^{1})) - \deg(\gamma) - \deg((\gamma^{0}, \theta^{0})) - \deg(\theta).
\end{equation*}
\end{lemma}

	There is a natural compactification $\bar{\mathcal{M}}^{bm}((\gamma^{0}, \theta^{0}); \gamma, (\gamma^{1}, \theta^{1}), \theta)$, whose codimension one stratum consists of the following:
\begin{equation}\label{boundary of moduli space of bimodule maps}
\begin{split}
&\coprod \mathcal{M}((\gamma^{0}, \theta^{0}); (\gamma^{0}_{1}, \theta^{0}_{1})) \times \mathcal{M}^{bm}((\gamma^{0}_{1}, \theta^{0}_{1}); \gamma, (\gamma^{1}, \theta^{1}), \theta)\\
&\cup \coprod \mathcal{M}^{bm}((\gamma^{0}, \theta^{0}); \gamma_{1}, (\gamma^{1}, \theta^{1}), \theta) \times \mathcal{M}(\gamma_{1}; \gamma)\\
&\cup \coprod \mathcal{M}^{bm}((\gamma^{0}, \theta^{0}); \gamma, (\gamma^{1}_{1}, \theta^{1}_{1}), \theta) \times \mathcal{M}((\gamma^{1}_{1}, \theta^{1}_{1}); (\gamma^{1}, \theta^{1}))\\
&\cup \coprod \mathcal{M}^{bm}((\gamma^{0}, \theta^{0}); \gamma, (\gamma^{1}, \theta^{1}), \theta_{1}) \times \mathcal{M}(\theta_{1}; \theta).
\end{split}
\end{equation} \par

	Similar to Lemma \ref{regularity + compactness}, we can prove the following:
\begin{lemma}
	$\bar{\mathcal{M}}^{bm}((\gamma^{0}, \theta^{0}); \gamma, (\gamma^{1}, \theta^{1}), \theta)$ is a compact smooth manifold with boundary and corners of dimension $\deg((\gamma^{0}, \theta^{0})) - \deg(\gamma) - \deg((\gamma^{1}, \theta^{1})) - \deg(\theta)$.
\end{lemma}
	However, we will not give a complete proof which is standard but somewhat cumbersome, since we will only need the $0$-dimensional and $1$-dimensional compactified moduli spaces (i.e. when $\deg((\gamma^{0}, \theta^{0})) - \deg(\gamma) - \deg((\gamma^{1}, \theta^{1})) - \deg(\theta)$ is either zero or one), which are then easily seen to be compact smooth manifolds and compact smooth manifolds with boundary respectively. Also, we have the following analogue of Lemma \ref{compactness 2}: \par

\begin{lemma}
	For a fixed triple $(\gamma, (\gamma^{1}, \theta^{1}), \theta)$, the moduli space
\begin{equation*}
\mathcal{M}^{bm}((\gamma^{0}, \theta^{0}); \gamma, (\gamma^{1}, \theta^{1}), \theta)
\end{equation*}
is empty for all but finitely many $(\gamma^{0}, \theta^{0})$'s.
\end{lemma}

	In particular, $0$-dimensional moduli spaces are compact manifolds, hence consist of finitely many elements. "Counting" these elements then defines a homomorphism of graded $\mathbb{Z}$-modules
\begin{equation}
\begin{split}
m^{1|0|1}&: CW^{*}(L_{0}, L_{1}; H_{M}, J_{M}) \otimes CW^{*}(L_{1}, \mathcal{L}, L'_{0}; H_{M}, H_{N}, J_{M}, J_{N}) \otimes CW^{*}(L'_{0}, L'_{1}; H_{N}, J_{N})\\
&\to CW^{*}(\phi_{M}^{2}L_{0}, (\phi_{M}^{2} \times \phi_{N}^{2})\mathcal{L}, \phi_{N}^{2}L'_{1}; \frac{H_{M}}{2} \circ \phi_{M}^{2}, \frac{H_{N}}{2} \circ \phi_{N}^{2}, (\phi_{M}^{2})^{*}J_{M}, (\phi_{N}^{2})^{*}J_{N})
\end{split}
\end{equation}
\begin{equation}
\begin{split}
&m^{1|0|1}([\gamma], [(\gamma^{1}, \theta^{1})], [\theta])\\
&= \sum_{\substack{\deg((\gamma^{0}, \theta^{0})) = \deg(\gamma) + \deg((\gamma^{1}, \theta^{1})) + \deg(\theta)\\ \underline{u} \in \mathcal{M}^{bm}((\gamma^{0}, \theta^{0}); \gamma, (\gamma^{1}, \theta^{1}), \theta)}} m^{1|0|1}_{\underline{u}}([\gamma], [(\gamma^{1}, \theta^{1})], [\theta]).
\end{split}
\end{equation}
Here $m^{1|0|1}_{\underline{u}}$ is the induced map on orientation lines. By analyzing the boundary of one-dimensional moduli spaces, we can prove that $m^{1|0|1}$ is a cochain map. Also, we identify the output
\begin{equation*}
CW^{*}(\phi_{M}^{2}L_{0}, (\phi_{M}^{2} \times \phi_{N}^{2})\mathcal{L}, \phi_{N}^{2}L'_{1}; \frac{H_{M}}{2} \circ \phi_{M}^{2}, \frac{H_{N}}{2} \circ \phi_{N}^{2}, (\phi_{M}^{2})^{*}J_{M}, (\phi_{N}^{2})^{*}J_{N})
\end{equation*}
with $CW^{*}(L_{0}, \mathcal{L}, L'_{1}; H_{M}, H_{N}, J_{M}, J_{N})$ canonically by conjugating by the Liouville flow. Thus we get a cochain map
\begin{equation}
\begin{split}
&m^{1|0|1}: CW^{*}(L_{0}, L_{1}; H_{M}, J_{M}) \otimes CW^{*}(L_{1}, \mathcal{L}, L'_{0}; H_{M}, H_{N}, J_{M}, J_{N}) \otimes CW^{*}(L'_{0}, L'_{1}; H_{N}, J_{N})\\
&\to CW^{*}(L_{0}, \mathcal{L}, L'_{1}; H_{M}, H_{N}, J_{M}, J_{N}),
\end{split}
\end{equation}
which induces the desired homomorphism \eqref{bimodule structure} on cohomology. \par
	There are other two maps as part of the bimodule structure maps, which on the chain level take the following forms
\begin{equation}
m^{1|0|0}: CW^{*}(L_{0}, L_{1}; H_{M}, J_{M}) \otimes CW^{*}(L_{1}, \mathcal{L}, L'_{0}) \to CW^{*}(L_{0}, \mathcal{L}, L'_{0}),
\end{equation}
\begin{equation}
m^{0|0|1}: CW^{*}(L_{0}, \mathcal{L}, L'_{0}) \otimes CW^{*}(L'_{0}, L'_{1}; H_{N}, J_{N}) \to CW^{*}(L_{0}, \mathcal{L}, L'_{1}).
\end{equation}
These are defined using moduli spaces of inhomogeneous pseudoholomorphic quilted maps from similar quilted surfaces. For example, $m^{1|0|0}$ is defined using quilted surfaces without the puncture $z_{1, 2}$ on the patch $S^{bm}_{1}$, while $m^{0|0|1}$ is defined using quilted surfaces without the puncture $z_{0, 2}$ on the patch $S^{bm}_{0}$. Because the Lagrangian submanifolds of $M$ and of $N$ we consider are conical, and the Lagrangian correspondence $\mathcal{L}$ is admissible, by the same argument we can show that the moduli spaces of inhomogeneous pseudoholomorphic quilted maps from these quilted surfaces have good structures, i.e. they have natural compactifications which are compact smooth manifolds with boundary and corners. Moreover, both $m^{1|0|0}$ and $m^{0|0|1}$ are cochain maps, inducing the desired cohomological level maps. This completes the construction of the first functor on the level of its effect on objects. \par
	Let us proceed to investigate the effect of the functor $\Phi$ on morphisms. Suppose we are given admissible Lagrangian correspondences $\mathcal{L}_{0}, \mathcal{L}_{1} \subset M^{-} \times N$, then the wrapped Floer cohomology $HW^{*}(\mathcal{L}_{0}, \mathcal{L}_{1})$ (defined using an admissible Hamiltonian on $M^{-} \times N$) is the morphism space from $\mathcal{L}_{0}$ to $\mathcal{L}_{1}$ in the cohomology category $H(\mathcal{W}(M^{-} \times N))$. The goal is therefore to construct a bimodule homomorphism
\begin{equation}\label{bimodule hom}
\Phi_{\Gamma}: P_{\mathcal{L}_{0}} \to P_{\mathcal{L}_{1}}
\end{equation}
for each element $[\Gamma] \in HW^{*}(\mathcal{L}_{0}, \mathcal{L}_{1})$. This is in fact straightforward, using the multiplicative structure on wrapped Floer cohomology of Lagrangians in the product manifold $M^{-} \times N$. Namely, for each $L \subset M, L' \subset N$, we have the following isomorphisms
\begin{equation}
HW^{*}(L, \mathcal{L}_{i}, L') \cong HW^{*}(\mathcal{L}_{i}, L \times L'), i = 0, 1.
\end{equation}
Then multiplication with $[\Gamma] \in HW^{*}(\mathcal{L}_{0}, \mathcal{L}_{1})$ under the triangle product
\begin{equation}
m^{2}: HW^{*}(\mathcal{L}_{1}, L \times L') \otimes HW^{*}(\mathcal{L}_{0}, \mathcal{L}_{1}) \to HW^{*}(\mathcal{L}_{0}, L \times L')
\end{equation}
induces a map
\begin{equation}
\Phi_{\Gamma}: HW^{*}(\mathcal{L}_{1}, L \times L') \to HW^{*}(\mathcal{L}_{0}, L \times L'),
\end{equation}
of degree $\deg(\Gamma)$. We have to prove it is indeed a bimodule homomorphism. Intuitively this is quite clear; however to rigorously implement such an idea, we work with the following alternative construction of the bimodule homomorphism. \par
	The alternative construction involves a kind of quilted surface $\underline{S}^{bh}$ consisting of two patches, $S^{bh}_{0}, S^{bh}_{1}$. $S^{bh}_{i}$ is a disk with one negative boundary puncture $z_{i, 0}$ and two positive boundary punctures $z_{i, 1}, z_{i, 2}$, cyclically ordered on the boundary of $D^{2}$. The boundary component of $S^{bh}_{i}$ going from $z_{i, j-1}$ to $z_{i, j}$ is denoted by $I_{i, j}$. We suppose that strip-like ends $\epsilon_{i, j}$ near all punctures have been chosen, in each individual patch. The seaming condition is the following: $I_{0, 2}$ with $I_{1, 2}$, and $I_{0, 0}$ with $I_{1, 0}$. After we seam the two patches together, the resulting quilted surface has three quilted ends: one negative quilted strip-like end $(\epsilon_{0, 0}, \epsilon_{1, 0})$, and one positive quilted strip-like end $(\epsilon_{0, 1}, \epsilon_{1, 1})$, and another positive end $(\epsilon_{0, 2}, \epsilon_{1, 2})$. The third one, $(\epsilon_{0, 2}, \epsilon_{1, 2})$, is regarded as a quilted cylindrical end, as the two pairs of near by boundary components are both seamed together. \par
	To write down the inhomogeneous Cauchy-Riemann equation for quilted maps $\underline{u}: \underline{S}^{bh} \to (M, N)$, we also need a time-shifting function $\rho_{\underline{S}^{bh}}$, which by definition consists of two time-shifting functions $\rho_{S^{bh}_{0}}$ and $\rho_{S^{bh}_{1}}$ defined on the boundary components of $S^{bh}_{0}$ and $S^{bh}_{1}$ respectively, such that they agree over the seams. We require that $\rho_{S^{bh}_{i}} = 1$ near $z_{i, 1}, z_{i, 2}$, and that $\rho_{S^{bh}_{i}} = 2$ near $z_{i, 0}$, for $i = 1, 2$. \par
	Then we choose Floer datum for the quilted surface $\underline{S}^{bh}$, which consists of Floer data $(\alpha_{S^{bh}_{0}}, H_{S^{bh}_{0}}, J_{S^{bh}_{0}})$ and $(\alpha_{S^{bh}_{1}}, H_{S^{bh}_{1}}, J_{S^{bh}_{1}})$ on the two patches, satisfying the following conditions. $(\alpha_{S^{bh}_{0}}, H_{S^{bh}_{1}}, J_{S^{bh}_{1}})$ agrees with $(dt, H_{M}, J_{M})$ over the two positive strip-like ends $\epsilon_{0, 1}, \epsilon_{0, 2}$ of $S^{bh}_{0}$, and with $(2dt, \frac{H_{M}}{4} \circ \phi_{M}^{2}, (\phi_{M}^{2})^{*}J_{M})$ over the negative strip-like end $\epsilon_{0, 0}$ of $S^{bh}_{0}$. And there are similar conditions for $(\alpha_{S^{bh}_{1}}, H_{S^{bh}_{1}}, J_{S^{bh}_{1}})$. \par
	Given generalized Hamiltonian chords $(\gamma^{0}, \theta^{0})$ for $(L, \mathcal{L}_{0}, L')$, $(\gamma^{1}, \theta^{1})$ for $(L, \mathcal{L}_{1}, L')$, as well as an $H_{M, N}$-chord $\Gamma$ in $M^{-} \times N$ from $\mathcal{L}_{0}$ to $\mathcal{L}_{1}$, the inhomogeneous Cauchy-Riemann equation for quilted maps $\underline{u}: \underline{S}^{bh} \to (M, N)$ with these asymptotic conditions takes the following form:
\begin{equation}
\begin{cases}
(du_{i} - \alpha_{S^{bh}_{i}} \otimes X_{H_{S^{bh}_{i}}}(u_{i}))^{0, 1} = 0, & i = 1, 2\\
u_{0}(z) \in \phi_{M}^{\rho_{S^{bh}_{0}}(z)}L, & z \in I_{0, 1}\\
(u_{0}(z), u_{1}(z)) \in (\phi_{M}^{\rho_{S^{bh}_{0}}(z)} \times \phi_{N}^{\rho_{S^{bh}_{1}}(z)})\mathcal{L}_{1}, & z \text{ lies on the seam } (I_{0,2}, I_{1, 2})\\
(u_{0}(z), u_{1}(z)) \in (\phi_{M}^{\rho_{S^{bh}_{0}}(z)} \times \phi_{N}^{\rho_{S^{bh}_{1}}(z)})\mathcal{L}_{0}, & z \text{ lies on the seam } (I_{0,0}, I_{1, 0})\\
u_{1}(z) \in \phi_{N}^{\rho_{S^{bh}_{1}}}(z)L', & z \in I_{1, 1}\\
\lim\limits_{s \to -\infty} u_{0} \circ \epsilon_{0, 0}(s, \cdot) = \phi_{M}^{2}\gamma^{0}(\cdot)\\
\lim\limits_{s \to +\infty} u_{0} \circ \epsilon_{0, 1}(s, \cdot) = \gamma^{i}(\cdot)\\
\lim\limits_{s \to -\infty} u_{1} \circ \epsilon_{1, 0}(s, \cdot) = \phi_{N}^{2}\theta^{0}(\cdot)\\
\lim\limits_{s \to +\infty} u_{1} \circ \epsilon_{1, 1}(s, \cdot) = \theta^{i}(\cdot)\\
\lim\limits_{s \to +\infty} (u_{0}, u_{1}) \circ (\epsilon_{0, 2}(s, 1-t), \epsilon_{1, 2}(s, t)) = \Gamma(t)
\end{cases}
\end{equation}
The picture of such a quilted map is shown in Figure 2.
\par

\begin{figure}\label{fig: the quilted map defining bimodule homomorphism}
\centering
\begin{tikzpicture}
	\draw (-6, 1.5) -- (6, 1.5);
	\draw (-6, 0) -- (-0.3, 0);
	\draw (0, 0) circle (0.3cm);
	\draw (0.3, 0) -- (6, 0);
	\draw (-6, -1.5) -- (6, -1.5);
	\draw (0, 2) node {$\phi_{M}^{\rho_{S^{bh}_{0}}(z)}L$};
	\draw (0, -2) node {$\phi_{N}^{\rho_{S^{bh}_{1}}}(z)L'$};
	\draw (-3, 0) node {$(\phi_{M}^{\rho_{S^{bh}_{0}}(z)} \times \phi_{N}^{\rho_{S^{bh}_{1}}(z)})\mathcal{L}_{0}$};
	\draw (3, 0) node {$(\phi_{M}^{\rho_{S^{bh}_{0}}(z)} \times \phi_{N}^{\rho_{S^{bh}_{1}}(z)})\mathcal{L}_{1}$};
	\draw (0, 1) node {$M$};
	\draw (0, -1) node {$N$};
\end{tikzpicture}
\caption{the quilted map defining bimodule homomorphisms}
\end{figure}
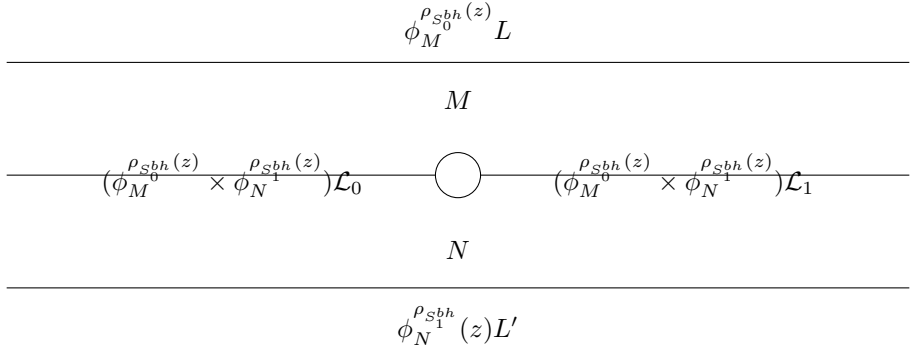

	Let $\mathcal{M}^{bh}((\gamma^{0}, \theta^{0}); (\gamma^{1}, \theta^{1}), \Gamma)$ be the (uncompactified) moduli space of inhomogeneous pseudoholomorphic quilts of the above kind. A standard transversality argument implies that for a generic choice of Floer datum for the quilted surface, the moduli space $\mathcal{M}^{bh}((\gamma^{0}, \theta^{0}); (\gamma^{1}, \theta^{1}), \Gamma)$ is a smooth manifold of dimension
\begin{equation*}
\deg((\gamma^{0}, \theta^{0})) - \deg((\gamma^{1}, \theta^{1})) - \deg(\Gamma).
\end{equation*} \par
	There is a natural compactification $\bar{\mathcal{M}}^{bh}((\gamma^{0}, \theta^{0}); (\gamma^{1}, \theta^{1}), \Gamma)$ whose codimension one boundary is covered by
\begin{equation}
\begin{split}
&\coprod \mathcal{M}((\gamma^{0}, \theta^{0}), (\gamma^{0}_{1}, \theta^{0}_{1})) \times \mathcal{M}^{bh}((\gamma^{0}_{1}, \theta^{0}_{1}); (\gamma_{1}, \theta_{1}), \Gamma)\\
&\cup \coprod \mathcal{M}^{bh}((\gamma^{0}, \theta^{0}); (\gamma^{1}_{1}, \theta^{1}_{1}), \Gamma) \times \mathcal{M}((\gamma^{1}_{1}, \theta^{1}_{1}), (\gamma_{1}, \theta_{1}))\\
&\cup \coprod \mathcal{M}^{bh}((\gamma^{0}, \theta^{0}); (\gamma_{1}, \theta_{1}), \Gamma_{1}) \times \mathcal{M}(\Gamma_{1}, \Gamma),
\end{split}
\end{equation}
where the last term $\mathcal{M}(\Gamma_{1}, \Gamma)$ stands for the moduli space of inhomogeneous pseudoholomorphic strips in $M^{-} \times N$ with boundary on the pair $(\mathcal{L}_{0}, \mathcal{L}_{1})$, with respect to the split Hamiltonian $H_{M, N}$ and product almost complex structure $-J_{M} \times J_{N}$, with asymptotic convergence conditions being $\Gamma_{1}$ at $-\infty$ and $\Gamma$ at $+\infty$. \par
	There are regularity and compactness results for this moduli space similar to Lemma \ref{regularity + compactness} and Lemma \ref{compactness 2}. They allow us to count rigid elements in zero-dimensional moduli spaces to obtain a map:
\begin{equation} \label{bimodule homomorphism from quilted surface}
\begin{split}
&\Phi: CW^{*}(L, \mathcal{L}_{1}, L'; H_{M}, H_{N}, J_{M}, J_{N}) \otimes CW^{*}(\mathcal{L}_{0}, \mathcal{L}_{1}; H_{M, N}, -J_{M} \times J_{N})\\
&\to CW^{*}(\phi_{M}^{2}L, (\phi_{M}^{2} \times \phi_{N}^{2})\mathcal{L}_{0}, \phi_{N}^{2} L'; \frac{H_{M}}{2} \circ \phi_{M}^{2}, \frac{H_{N}}{2} \circ \phi_{N}^{2}, (\phi_{M}^{2})^{*}J_{M}, (\phi_{N}^{2})^{*}J_{N})
\end{split}
\end{equation}
\begin{equation}
\Phi([\Gamma], [(\gamma^{1}, \theta^{1})]) = \sum_{\substack{\underline{u} \in \mathcal{M}^{bh}((\gamma^{0}, \theta^{0}); (\gamma^{1}, \theta^{1}), \Gamma)\\ \deg((\gamma^{0}, \theta^{0})) = \deg((\gamma^{1}, \theta^{1})) + \deg(\Gamma)}} \Phi_{\underline{u}}([\Gamma], [(\gamma^{1}, \theta^{1})]),
\end{equation}
where $\Phi_{\underline{u}}$ is the induced map on orientation lines. A standard gluing argument implies that this is a cochain map. Also, we identify the output
\begin{equation*}
CW^{*}(\phi_{M}^{2}L, (\phi_{M}^{2} \times \phi_{N}^{2})\mathcal{L}_{0}, \phi_{N}^{2} L'; \frac{H_{M}}{2} \circ \phi_{M}^{2}, \frac{H_{N}}{2} \circ \phi_{N}^{2}, (\phi_{M}^{2})^{*}J_{M}, (\phi_{N}^{2})^{*}J_{N})
\end{equation*}
with $CW^{*}(L, \mathcal{L}_{0}, L'; H_{M}, H_{N}, J_{M}, J_{N})$ canonically by conjugating by the Liouville flow, and thus regard the map \eqref{bimodule homomorphism from quilted surface} as
\begin{equation}
\begin{split}
\Phi: &CW^{*}(L, \mathcal{L}_{1}, L'; H_{M}, H_{N}, J_{M}, J_{N}) \otimes CW^{*}(\mathcal{L}_{0}, \mathcal{L}_{1}; H_{M, N}, -J_{M} \times J_{N})\\
&\to CW^{*}(L, \mathcal{L}_{0}, L'; H_{M}, H_{N}, J_{M}, J_{N}).
\end{split}
\end{equation}
Hence we get an induced map on cohomology:
\begin{equation*}
HW^{*}(L, \mathcal{L}_{1}, L') \otimes HW^{*}(\mathcal{L}_{0}, \mathcal{L}_{1}) \to HW^{*}(L, \mathcal{L}_{0}, L').
\end{equation*}
Fixing the input $[\Gamma] \in HW^{*}(\mathcal{L}_{0}, \mathcal{L}_{1})$, we get the desired map $\Phi_{\Gamma}$. Although its form of presence is not quite transparent, this will be the $(0, 0)$-th order map of the bimodule homomorphism \eqref{bimodule hom}. \par
	To extend this map to higher orders, at least to orders $(1, 0), (0, 1)$ and $(1, 1)$ so that we can verify that it defines a bimodule homomorphism on the cohomological category, we shall introduce a few more quilted surfaces and study inhomogeneous pseudoholomorphic maps from them. The relevant quilted surface $\underline{S}^{bh, k, l}$ is obtained from $\underline{S}^{bh}$ by adding $k$ positive punctures on the $I_{0, 1}$ boundary component of $S^{bh}_{0}$ and $l$ positive punctures on the $I_{1, 1}$ boundary component of $S^{bh}_{1}$, where $k = 0, 1$ and $l = 0, 1$ for our interest in this paper. Convergence conditions for the inhomogeneous pseudoholomorphic quilted maps from $\underline{S}^{bh, k, l}$ at these additional punctures are Floer cochains between Lagrangians in $M$ and respectively $N$. Since we only consider conical Lagrangian submanifolds of $M$ and of $N$, and admissible Lagrangian correspondences from $M$ to $N$, the moduli spaces of such quilted maps all have good structures, i.e. they have natural compactifications which are compact smooth manifolds with boundary and corners. The description of the compactification is similar, and counting rigid elements therein yields the following maps
\begin{equation}
\begin{split}
\Phi^{1, 0}: &CW^{*}(\mathcal{L}_{0}, \mathcal{L}_{1}; H_{M, N}, -J_{M} \times J_{N})\\
&\to \hom(CW^{*}(L_{0}, L_{1}; H_{M}, J_{M}) \otimes CW^{*}(L_{1}, \mathcal{L}_{0}, L'_{0}), CW^{*}(L_{0}, \mathcal{L}_{1}, L'_{0})),
\end{split}
\end{equation}

\begin{equation}
\begin{split}
\Phi^{0, 1}: &CW^{*}(\mathcal{L}_{0}, \mathcal{L}_{1}; H_{M, N}, -J_{M} \times J_{N})\\
&\to \hom(CW^{*}(L_{0}, \mathcal{L}_{0}, L'_{0}) \otimes CW^{*}(L'_{0}, L'_{1}; H_{N}, J_{N}), CW^{*}(L_{0}, \mathcal{L}_{1}, L'_{1})),
\end{split}
\end{equation}

\begin{equation}
\begin{split}
\Phi^{1, 1}: &CW^{*}(\mathcal{L}_{0}, \mathcal{L}_{1}; H_{M, N}, -J_{M} \times J_{N}) \to \hom(CW^{*}(L_{0}, L_{1}; H_{M}, J_{M})\\
&\otimes CW^{*}(L_{1}, \mathcal{L}_{0}, L'_{0}) \otimes CW^{*}(L'_{0}, L'_{1}; H_{N}, J_{N}), CW^{*}(L_{0}, \mathcal{L}_{1}, L'_{1})),
\end{split}
\end{equation}
which are all cochain maps with respect to the corresponding differentials, where the differential on the $\hom$-complex is induced by (quilted) Floer differentials on various (quilted) wrapped Floer cochain complexes. In particular, these include the case $k = l = 0$. \par

\begin{remark}
	Allowing general numbers $k, l$, we can define the $A_{\infty}$-bimodule structures and homomorphisms in this way, which will appear in \cite{Gao}. As already mentioned, there is still a technical issue regarding full chain-level $A_{\infty}$-structures for wrapped Floer theory for the product manifold, which is the reason why we stay in the case $k, l \le 1$.
\end{remark}

\subsection{Module-valued functor}
	To convert the above construction into the form as stated in Theorem \ref{functors associated to Lagrangian correspondences}, it suffices to play with some simple homological algebra, which mostly has to do with a Yoneda-type functor. \par
	There is a general way of turning bimodules to functors to the category of modules. To be more precise, let $A, B$ be (locally small) linear categories, and $P$ be a bimodule over $(A, B)$, i.e. a right-$A$ and left-$B$ bimodule. Then we can construct in a canonical way a functor from $A$ to the category of left $B$-modules.
\begin{equation}
F_{P}: A \to l-Mod(B),
\end{equation}
which is defined as follows. For any object $X \in Ob(A)$, $F_{P}(X)$ is the left $B$-module:
\begin{equation}
Y \mapsto P(X, Y).
\end{equation}
For any morphism $a \in \hom_{A}(X_{0}, X_{1})$, $F_{P}(a)$ is the $B$-module homomorphism:
\begin{equation}
F_{P}(X_{1}) \to F_{P}(X_{0})
\end{equation}
which is induced by multiplication by $a$ with respect to the right-$A$ structure:
\begin{equation}
\hom_{A}(X_{0}, X_{1}) \otimes P(X_{1}, Y) \to P(X_{0}, Y).
\end{equation}
With additional little amount of argument, one can show: \par

\begin{lemma}
	Suppose both $A$ and $B$ are unital. Then the above construction yields a fully faithful functor
\begin{equation}
F: (A, B)^{bimod} \to Func(A, l-Mod(B)).
\end{equation}
\end{lemma}

	We apply this lemma to the bimodule-valued functor \eqref{bimodule-valued functor} to obtain the desired functor \eqref{functor in a categorical level}, and evaluate at any object $\mathcal{L}$ to obtain the functor \eqref{functor associated to a single Lagrangian correspondence}. \par
	Finally, the statement that both functors are cohomologically unital follows from the fact that the Yoneda functor is cohomologically fully faithful and therefore cohomologically unital. While the cohomological fullness and failthfulness follow from the fact that the wrapped Fukaya category is cohomologically unital. \par
	As a concluding remark, we mention that the above process of turning bimodules to module-valued functors can be extended to $A_{\infty}$-categories, which will explained in \cite{Gao} to carry out chain-level constructions of $A_{\infty}$-functors from Lagrangian correspondences. \par

\section{Geometric compositions and representability}

\subsection{Geometric composition of Lagrangian correspondences}
	The main reason for us to study geometric compositions of Lagrangian correspondences is that we attempt to prove representability of the above-constructed module-valued functors $\Phi_{\mathcal{L}}$ (and $A_{\infty}$-version to appear in \cite{Gao}). The result necessary for establishing representability is Theorem \ref{geometric composition isomorphism}, which we shall prove in this section. \par
	In the case of wrapped Floer cohomology, the definition of the geometric composition of two Lagrangian correspondences is slightly different from that of compact Lagrangian correspondences between closed/tame symplectic manifolds when the usual compact Fukaya categories are concerned, in the sense that the Lagrangian correspondences should be composed in the way governed by the flow of the Hamiltonian vector field, which encodes information about wrapping. We formalize the description as follows. \par
	
\begin{definition}\label{geometric composition}
	Given Lagrangian correspondences $\mathcal{L}_{01} \subset M_{0}^{-} \times M_{1}, \mathcal{L}_{12} \subset M_{1}^{-} \times M_{2}$, the geometric composition of them is defined to be the map from the fiber product of $\mathcal{L}_{01}$ and $\mathcal{L}_{12}$ over the graph of $\psi_{H_{M_{1}}} \subset M_{1} \times M_{1}^{-}$ to $M_{0}^{-} \times M_{2}$
\begin{equation}
\iota_{02}: \mathcal{L}_{02} = \mathcal{L}_{01} \circ_{H_{M_{1}}} \mathcal{L}_{12} = \mathcal{L}_{01} \times_{\Gamma(\psi_{H_{M_{1}}})} \mathcal{L}_{12} \to M_{0}^{-} \times M_{2},
\end{equation}
where the map $\iota_{02}$ is induced by the projection $M_{0} \times M_{1} \times M_{1} \times M_{2} \to M_{0} \times M_{2}$.
\end{definition}

	In a generic situation, the geometric composition is a Lagrangian immersion, which might not even be proper. In this paper, we will restrict to the case where the geometric composition is a proper embedding. This assumption of course impose very strong restriction on the geometry of these Lagrangian correspondences. \par
	This definition of geometric composition is well suited to the wrapped Lagrangian Floer theory. However, it causes one potential trouble: in general the geometric composition of arbitrary two conical Lagrangian submanifolds might no longer be conical, because of the Hamiltonian perturbation. In the next subsection, we will impose a further condition on such geometric compositions, and explain how the wrapped Floer cohomology can still be defined for these particular kinds of Lagrangian submanifolds. \par
	In this section, we shall focus on the case of the geometric composition of a conical Lagrangian submanifold $L \subset M$ with an admissible Lagrangian correspondence $\mathcal{L} \subset M^{-} \times N$. We also call $\iota: L \circ_{H_{M}} \mathcal{L} \to N$ the geometric transform of $L$ by $\mathcal{L}$. Studying geometric compositions of general Lagrangian correspondences as in Definition \ref{geometric composition} is of importance to further purposes, for example, understanding composition of functors. However, currently there are still unsolved technical issues concerning wrapped Floer theory in multiple products of Liouville manifolds, so we will not discuss that point. \par

\subsection{Wrapped Floer cohomology of the geometric composition} \label{section: well-definedness of wrapped Floer cohomology of the geometric composition}
	Assume from now on that the geometric composition $\iota: L \circ_{H_{M}} \mathcal{L} \to N$ is a proper embedding. However, it may no longer be conical. Thus proving well-definedness of its wrapped Floer cohomology is a task before trying to prove the isomorphism of wrapped Floer cohomology under geometric composition \eqref{isomorphism under geometric composition}. This subsection is devoted to proving the first part of Theorem \ref{geometric composition isomorphism}. \par
	Let $f$ be the primitive for $L$ and $F$ be the primitive for $\mathcal{L}$. That is, $df = \lambda_{M}|_{L}$ and $dF = (-\lambda_{M} \times \lambda_{N})|_{\mathcal{L}}$. The following lemma provides a canonical choice of a primtive for the geometric composition:

\begin{lemma}
If $\iota: L \circ_{H_{M}} \mathcal{L} \to N$ is a proper embedding, then
\begin{equation} \label{primitive for the geometric composition}
g = (f + F \circ (\psi_{H_{M}} \times id_{N}) + i_{X_{H_{M}}}\lambda_{M}) \circ \iota^{-1}
\end{equation}
is a primitive for the Lagrangian submanifold $\iota(L \circ_{H_{M}} \mathcal{L})$ of $N$. Here $i_{X_{H_{M}}}\lambda_{M}$ means contracting the one-form $\lambda_{M}$ by the Hamiltonian vector field $X_{H_{M}}$. In particular, $L \circ_{H_{M}} \mathcal{L}$ is exact.
\end{lemma}
\begin{proof}
	Recall that (the preimage of) the geometric composition is the fiber product over the graph of $\psi_{H_{M}}$:
\begin{equation}
L \times_{\Gamma(\psi_{H_{M}})} \mathcal{L} = \{(p, \psi_{H_{M}}(p), q) \in M \times M \times N| p \in L, (\psi_{H_{M}}(p), q) \in \mathcal{L}\}
\end{equation}
which is canonically isomorphic to the fiber product over the diagonal of $L$ and the perturbed $\mathcal{L}$:
\begin{equation} \label{fiber product with the perturbed Lagrangian correspondence}
L \times_{\Delta_{M}} (\psi_{H_{M}}^{-1} \times id_{N})\mathcal{L} = \{(p, p, q) \in M \times M \times N| p \in L, (p, q) \in (\psi_{H_{M}}^{-1} \times id_{N})\mathcal{L}\}
\end{equation}
By tranversality assumption, $L \circ_{H_{M}} \mathcal{L}$ is a smooth submanifold of $M \times M \times N$ and comes with a natural smooth map $\iota$ to $N$ induced by the projection $M \times M \times N \to N$ sending $(p, \psi_{H_{M}}(p), q)$ to $q$. Moreover, the above isomorphism is canonical in the sense that the second fiber product \eqref{fiber product with the perturbed Lagrangian correspondence} also has a natural map to $N$ which agrees with $\iota$, sending $(p, p, q)$ to $q$. We will for convenience seek a function on \eqref{fiber product with the perturbed Lagrangian correspondence} whose differential is equal to the pullback of $\lambda_{N}$. \par
	Since $F$ is a primitive for $\mathcal{L}$ with respect to the one form $(-\lambda_{M} \times \lambda_{N})|_{\mathcal{L}}$, a canonical choice of a primitive for the perturbed Lagrangian correspondence $(\psi_{H_{M}}^{-1} \times id_{N})\mathcal{L}$ is given by:
\begin{equation}
\tilde{F} = F \circ (\psi_{H_{M}} \times id_{N}) + i_{X_{H_{M}}}\lambda_{M}.
\end{equation}
The above formula is obtained from the proof of the well-known fact that a Hamiltonian symplectomorphism is an exact symplectomorphism in an exact symplectic manifold, which writes down an explicit formula for the change of the primitive under a Hamiltonian isotopy. To clarify the above formula, the value of $\tilde{F}$ at a point $(p, q) \in (\psi_{H_{M}}^{-1} \times id_{N})\mathcal{L}$ is
\begin{equation}
F(\psi_{H_{M}}(p), q) + (i_{X_{H_{M}}}\lambda_{M})(p).
\end{equation}
By noting that on the diagonal $\Delta_{M}$, the restriction of $\lambda_{M} \times (-\lambda_{M})$ equals zero, we may add $f$ and $\tilde{F}$ together and restrict that to the fiber product $L \times_{\Delta_{M}} (\psi_{H_{M}}^{-1} \times id_{N})\mathcal{L}$ to obtain a fuction $h$ ($h \circ \iota^{-1} = g$) satisfying
\begin{equation}
dh = \lambda_{N}|_{L \times_{\Delta_{M}} (\psi_{H_{M}}^{-1} \times id_{N})\mathcal{L}}.
\end{equation}
Finally, recall that $\iota: L \times_{\Delta_{M}} (\psi_{H_{M}}^{-1} \times id_{N})\mathcal{L} \to N$ is induced by the projection $M \times M \times N \to N$, and therefore we may conclude that
\begin{equation}
dg = \lambda_{N}|_{\iota(L \times_{\Delta_{M}} (\psi_{H_{M}}^{-1} \times id_{N})\mathcal{L})},
\end{equation}
as desired.
\end{proof}

	Having extended $f$ to the whole $M$ and $F$ to the whole $M^{-} \times N$ which are locally constant in the corresponding cylindrical ends, we get a natural extension of $g$ to the whole $N$. Let us try to analyze the behavior of $g$ in the cylindrical end $\partial N \times [1, +\infty)$. Since we assume the geometric composition to be properly embedded in $N$, the Lagrangian correspondence $\mathcal{L}$ in consideration should be conical with respect to the natural choice of cylindrical end $\Sigma \times [1, +\infty)$ of $M^{-} \times N$, and the projection $\mathcal{L} \to N$ must be proper. As a consequence, the primitive $F$ is locally constant in $\Sigma \times [1, +\infty)$. Note that the projection of $\Sigma \times [1, +\infty)$ to the $N$-factor contains a cylindrical part $\partial N \times [A, +\infty)$ of $N$, for some constant $A$ possibly slightly bigger than $1$. Thus for a point $q \in \iota(L \times_{\Delta_{M}} (\psi_{H_{M}}^{-1} \times id_{N})\mathcal{L}) \cap \partial N \times [1, +\infty)$ which lies on a hypersurface $\partial N \times \{R\}$ for large $R$, $F(\psi_{H_{M}}(p), q)$ is uniformly bounded independent of $R$, for the unique point $(p, p, q) = \iota^{-1}(q) \in L \times_{\Delta_{M}} (\psi_{H_{M}}^{-1} \times id_{N})\mathcal{L}$. Also, for the same reason, $f(p)$ is uniformly bounded independent of the position of $q$. It remains to estimate the term $(i_{X_{H_{M}}}\lambda_{M})(p)$. Since $\mathcal{L} \to N$ is proper, for $R$ sufficiently large, the point $\psi_{H_{M}}(p)$ must lie in the cylindrical end $\partial M \times [1, +\infty)$, thus on some hypersurface $\partial M \times \{R_{1}\}$ for some $R_{1} = R_{1}(R)$ depending on $R$. Moreover, as $R$ increases, $R_{1}$ also increases. Since $H_{M}$ is quadratic in the radial coordinates on the cylindrical end $\partial M \times [1, +\infty)$, the Hamiltonian symplectomorphism $\psi_{H_{M}}$ preserves the cylindrical end $\partial M \times [1, +\infty)$ and in fact preserves every level hypersurface. Thus $p \in \partial M \times \{R_{1}\}$ also. As a consequence, we may calculate that
\begin{equation} \label{the extra term contributing to the primitive}
(i_{X_{H_{M}}}\lambda_{M})(p) = 2R_{1}^{2}.
\end{equation}
The crucial fact is that it is constant along the hypersurface $\partial M \times \{R_{1}\}$, and therefore independent of the position of the point $q \in \partial N \times \{R\}$ as long as $q$ various on the $R$-level hypersurface. \par
	It is with this choice of primitive $g$ for the geometric composition $L \circ_{H_{M}} \mathcal{L}$ that the wrapped Floer cohomology is to be defined. In a more categorical language, the geometric composition $L \circ_{H_{M}} \mathcal{L}$ together with its canonical grading and relative spin structure as well as this particular choice of primitive becomes an object of the wrapped Fukaya category. We will not investigate the $A_{\infty}$-structures in this paper, but rather focus on the wrapped Floer cohomology. Since we always assume that $\iota$ is a proper embedding, we will quite ofen write $L \circ_{H_{M}} \mathcal{L}$ instead of $\iota(L \circ_{H_{M}} \mathcal{L})$ for the Lagrangian submanifold of $N$ by abuse of notation. \par
	As before, we understand that the main difficulty in proving well-definedness of wrapped Floer cohomology for the geometric composition is with the compactness of the relevant moduli spaces. The general strategy is to make use of the action-energy equality \eqref{action-energy equality for strips} to provide a priori energy bound for the inhomogeneous pseudoholomorphic strips in terms of the action of the asymptotic Hamiltonian chords. Thus we shall always begin by estimating the action of Hamiltonian chords in consideration. \par
	First, we explain why the self wrapped Floer cohomology $HW^{*}(L \circ_{H_{M}} \mathcal{L}, L \circ_{H_{M}} \mathcal{L})$ is well-defined. Let $y$ be an $H_{N}$-chord from $L \circ_{H_{M}} \mathcal{L}$ to itself. Then its action is calculated by the following formula:
\begin{equation}
\mathcal{A}_{H_{N}, L \circ_{H_{M}} \mathcal{L}}(y) = -\int_{0}^{1}y^{*}\lambda_{N} + H_{N}(y(t))dt + g(y(1)) - g(y(0)).
\end{equation}
If $y$ lies in the interior part of $N$, we may ensure that $|\mathcal{A}_{H_{N}, L \circ_{H_{M}} \mathcal{L}}(y)|$ is small, say bounded by a small constant $c$, by choosing $H_{N}$ to be small in the interior part. If $y$ lies in the cylindrical end $\partial N \times [1, +\infty)$, it must lie on some hypersurface $\partial N \times \{R\}$ for some $R$, by the Hamilton's equation. Thus the integral term contributes $-R^{2}$ to the aciton. It remains to estimate $g(y(1)) - g(y(0))$. Since $y(1)$ and $y(0)$ lies on the same hypersurface $\partial N \times \{R\}$, we may deduce that $|g(y(1)) - g(y(0))|$ is uniformly bounded independent of $R$, because of the above analysis on the behavior of the primitive $g$, in particular the calculation \eqref{the extra term contributing to the primitive} showing that the values of $i_{X_{H_{M}}}\lambda_{M}$ at $p_{1}$ and $p_{0}$ agree, where the point $p_{i}$ is defined by $(p_{i}, p_{i}, y(i)) = \iota^{-1}(y(i))$, for $i=1, 2$. Therefore, we obtain the following estimate for the action of such a Hamiltonian chord:
\begin{equation}
-R^{2} - C \le \mathcal{A}_{H_{N}, L \circ_{H_{M}} \mathcal{L}}(y) \le -R^{2} + C,
\end{equation}
where $C$ is a universal constant independent of $R$, which depends only on $L, \mathcal{L}$ and the Hamiltonians $H_{M}, H_{N}$, all of which are given data. \par
	Now we are ready to prove the desired compactness results. First, for fixed Hamiltonian chords $y_{0}, y_{1}$ from $\iota(L \circ_{H_{M}} \mathcal{L})$ to itself, the action-energy equality \eqref{action-energy equality for strips} implies that for any $u \in \bar{\mathcal{M}}(y_{0}, y_{1})$, its energy $E(u) = \mathcal{A}_{H_{N}, L \circ_{H_{M}} \mathcal{L}}(y_{0}) - \mathcal{A}_{H_{N}, L \circ_{H_{M}} \mathcal{L}}(y_{1})$ is fixed. Also, the maximum principle prevents any component of the broken inhomogeneous pseudoholomorphic strip from escaping to infinity, and therefore Gromov compactness theorem can be applied and implies that $\bar{\mathcal{M}}(y_{0}, y_{1})$ is compact. This argument only uses the exactness condition and does not rely on the estimate of the action of Hamiltonian chords as above. Second, for fixed $y_{1}$, if there were infinitely many $y_{0}$'s for which $\bar{\mathcal{M}}(y_{0}, y_{1})$ is non-empty, there must be one which lies on $\partial N \times \{R\}$ for some large $R$. For $u \in \bar{\mathcal{M}}(y_{0}, y_{1})$, its energy then satisfies
\begin{equation}
-R^{2} - C - \mathcal{A}_{H_{N}, L \circ_{H_{M}} \mathcal{L}}(y_{1}) \le E(u) \le -R^{2} + C - \mathcal{A}_{H_{N}, L \circ_{H_{M}} \mathcal{L}}(y_{1}),
\end{equation}
which is negative for $R$ sufficiently large, as $\mathcal{A}_{H_{N}, L \circ_{H_{M}} \mathcal{L}}(y_{1})$ is fixed. This is a contradiction and thus proves that for fixed $y_{1}$, $\bar{\mathcal{M}}(y_{0}, y_{1})$ is empty for all but finitely many $y_{0}$'s. Therefore, the differential on $CW^{*}(L \circ_{H_{M}} \mathcal{L}, L \circ_{H_{M}} \mathcal{L})$ has a finite output, and the self wrapped Floer cohomology is well-defined in the usual sense. \par
	Second, we would like to show that the wrapped Floer cohomology $HW^{*}(L \circ_{H_{M}} \mathcal{L}, L')$ is well-defined, for some conical Lagrangian submanifold $L'$ of $N$, with a locally constant primitive $f'$. For an $H_{N}$-chord $y$ from $L \circ_{H_{M}} \mathcal{L}$ to $L'$, its action is
\begin{equation}
\mathcal{A}_{H_{N}, L \circ_{H_{M}} \mathcal{L}, L'}(y) = -\int_{0}^{1}y^{*}\lambda_{N} + H_{N}(y(t))dt + f'(y(1)) - g(y(0)).
\end{equation}
The first two terms are the same as those in the previous case, so it suffices to estimate $f'(y(1)) - g(y(0))$. If $y$ lies on $\partial N \times \{R\}$, then $f'(y(1))$ is some constant $C'$ because $f'$ is locally constant in the cylindrical end, while $2R_{1}^{2} - D \le g(y(0)) \le 2R_{1}^{2} + D$, where $D$ is a universal constant independent of $R$ and depends only on the bounds of $f, F$, both of which are uniformly bounded. Therefore,
\begin{equation}
f'(y(1)) - g(y(0)) \le C' - 2R_{1}^{2} + D.
\end{equation}
Recall $R_{1} = R_{1}(R)$ increases as $R$ increases, so $R_{1} \to +\infty$ as $R \to +\infty$, and correspondingly $f'(y(1)) - g(y(0)) \to -\infty$. Based on the previous argument using action-energy equality, this is enough to establish compactness results for the relevant moduli space $\bar{\mathcal{M}}(y_{0}, y_{1})$, hence confirming well-definedness of wrapped Floer cohomology in this case. \par
	Third, we attempt to prove that the wrapped Floer cohomology $HW^{*}(L', L \circ_{H_{M}} \mathcal{L})$ is well-defined. Note this is not symmetric to the second case in the presence of wrapping, so an independent proof is necessary. A Hamiltonian chord $y$ from $L'$ to $L \circ_{H_{M}} \mathcal{L}$ has action
\begin{equation}
\mathcal{A}_{H_{N}, L', L \circ_{H_{M}} \mathcal{L}}(y) = -\int_{0}^{1}y^{*}\lambda_{N} + H_{N}(y(t))dt + g(y(1)) - f'(y(0)).
\end{equation}
This situation is somewhat opposite to the previous one, as $g(y(1)) - f'(y(0))$ is approximately $2R_{1}^{2}$ up to some bounded constant, which does not go to $-\infty$ as $R \to +\infty$, but instead goes to $+\infty$. This appears to be an essential obstruction to making wrapped Floer cohomology well-defined, which we cannot resolve temporarily. In this regard, we shall restrict ourselves to some special classes of Lagrangian correspondences $\mathcal{L}$. Either of the following condition is sufficient for the wrapped Floer cohomology $HW^{*}(L', L \circ_{H_{M}} \mathcal{L})$ to be well-defined:
\begin{enumerate}[label=(\roman*)] \label{condition for the wrapped Floer cohomology for the geometric composition as a left-module to be well-defined}

\item There are only finitely many time-one $H_{N}$-chords from $L'$ to $L \circ_{H_{M}} \mathcal{L}$. The typical example is $\mathcal{L} = \Delta_{M} \subset M^{-} \times M$ the diagonal, and $L, L'$ are conical Lagrangian submanifolds of $M$ whose Legendrian boundary $\partial L, \partial L'$ do not intersect (this is a generic condition).

\item For geometric reason from the geometry of $\mathcal{L}$, we may deduce that $R_{1} = R_{1}(R) \le cR$ for some constant $0 < c < \frac{\sqrt{2}}{2}$.

\end{enumerate}
The above conditions are temporary, as currently there are no better or more natural conditions to be formulated. Fortunately, we will not frequently consider $HW^{*}(L' L \circ_{H_{M}} \mathcal{L})$ as it does not appear in the isomorphism of wrapped Floer cohomology under geometric composition. \par
	Finally, it remains to understand the well-definedness of the wrapped Floer cohomology $HW^{*}(L_{1} \circ_{H} \mathcal{L}_{1}, L_{2} \circ_{H} \mathcal{L}_{2})$, for a pair of geometric compositions. This time, the action of a Hamiltonian chord $y$ from $L_{1} \circ_{H} \mathcal{L}_{1}$ to $L_{2} \circ_{H} \mathcal{L}_{2}$ is:
\begin{equation}
\mathcal{A}_{H_{N}, L_{1} \circ_{H} \mathcal{L}_{1}, L_{2} \circ_{H} \mathcal{L}_{2}}(y) = -\int_{0}^{1}y^{*}\lambda_{N} + H_{N}(y(t))dt + g_{2}(y(1)) - g_{1}(y(0)),
\end{equation}
where $g_{i}$ is the canonical choice of the primitive for $L_{i} \circ_{H} \mathcal{L}_{i}$, as in \eqref{primitive for the geometric composition}. As before, in order to prove the desired compactness results, it suffices to prove that if $y$ lies on $\partial N \times \{R\}$, its action is sufficiently negative as long as $R$ is sufficiently large, and decreases as $R$ increases, at least for large values of $R$. Again, the integral terms contribute $-R^{2}$. To estimate $g_{2}(y(1)) - g_{1}(y(0))$, we recall from the definion of these primitives $g_{i}$ in \eqref{primitive for the geometric composition} that the term $i_{X_{H_{M}}}\lambda_{M}$ appears in both $g_{2}$ and $g_{1}$, and takes the same at $p_{1}$ and $p_{0}$ since $y(1)$ and $y(0)$ lies on the same hypersurface $\partial N \times \{R\}$. Thus when substracting $g_{1}(y(0))$ from $g_{2}(y(1))$, the contribution from this term cancel. The remaining terms are
\begin{equation}
f_{2}(p_{1}) + F_{2}(\psi_{H_{M}}(p_{1}), y(1)) - f_{1}(p_{0}) - F_{1}(\psi_{H_{M}}(p_{0}), y(0)),
\end{equation}
which is uniformly bounded independent of $R$ since all these functions are locally constant in the corresponding cylindrical ends. Thus, the action $\mathcal{A}_{H_{N}, L_{1} \circ_{H} \mathcal{L}_{1}, L_{2} \circ_{H} \mathcal{L}_{2}}(y)$ satisfies
\begin{equation}
-R^{2} - C \le \mathcal{A}_{H_{N}, L_{1} \circ_{H} \mathcal{L}_{1}, L_{2} \circ_{H} \mathcal{L}_{2}}(y) \le -R^{2} + C,
\end{equation}
for some universal constant $C$ depending only on $L_{1}, L_{2}, \mathcal{L}_{1}, \mathcal{L}_{2}$ (and their primitives), as well as the Hamiltonians $H_{M}, H_{N}$, but not on $R$. For sufficiently large $R$, this estimate shows the desired result. \par

\subsection{Representability by the geometric composition}\label{section:representability}
	The geometric composition $L \circ_{H_{M}} \mathcal{L}$ provides a natural candidate to represent the module-valued functor constructed in the previous section. The first and essential step toward proving this fact is the isomorphism on wrapped Floer cohomologies under this geometric composition. The full $A_{\infty}$-structures will be discussed in \cite{Gao}. \par
	Recall that in the context of wrapped Fukaya categories, the geometric composition  of $L$ and $\mathcal{L}$ is defined to be the Lagrangian immersion:
\begin{equation}
\iota: L \circ_{H_{M}} \mathcal{L} = L \times_{\Gamma(\psi_{H_{M}})} \mathcal{L} \to N,
\end{equation}
where $L \times_{\Gamma(\psi_{H_{M}})} \mathcal{L}$ is the fiber product of $L$ and $\mathcal{L}$ over the graph of $\psi_{H_{M}} \subset M^{-} \times M$. Here the map $\iota$ is induced by the projection $M \times M \times N \to N$. As discussed in the preceding subsection, there is a well-defined wrapped Floer cohomology for the geometric composition. In this subsection, we prove the second part of Theorem \ref{geometric composition isomorphism}. \par
	Before we discuss the proof, let us first observe that there is a natural identification between generalized chords for $(L, \mathcal{L}, L')$ and time-one $H_{N}$-chords from $L'$ to itself, which yields a map of abelian groups
\begin{equation}\label{identification of generalized chords by energy-zero solutions}
CW^{*}(L, \mathcal{L}, L') \cong CW^{*}(L \circ_{H_{M}} \mathcal{L}, L').
\end{equation}
However, this naive identification is not a cochain map in general, so we want to "correct" this to make it a cochain map. \par
	The proof again uses a quilted Floer-theoretic construction. We consider a quilted surface $\underline{S}^{gc}$ consisting of two patches $S^{gc}_{0}, S^{gc}_{1}$, where $S^{gc}_{0}$ is a disk with two positve boundary punctures near which strip-like ends $\epsilon_{0, j}, j = 0, 1$ are chosen, and $S^{gc}_{1}$ is a disk with one negative boundary puncture and two positive boundary punctures, near which strip-like ends $\epsilon_{1, j}, j = 0, 1, 2$ are chosen. We seam together the boundary component $I_{0, 1}$ with $I_{1, 2}$. The quilted surface $\underline{S}^{gc}$ then has one negative strip-like end and two positive quilted ends, each one consisting of two strip-like end components. \par

\begin{figure}
\centering
\begin{tikzpicture}
	\draw (-4, -1) -- (4, -1);
	\draw (-4, 1) -- (-2, 1);
	\draw (-2, 1) arc (270:360:1cm);
	\draw (-1, 2) -- (-1, 3);
	\draw (0, 1) -- (0, 3);
	\draw (1, 0) arc (270:180:1cm);
	\draw (1, 0) -- (4, 0);
	\draw (1, 3) -- (1, 2);
	\draw (1, 2) arc (180:270:1cm);
	\draw (2, 1) -- (4, 1);
	\draw (0, -1.5) node {$\phi_{N}^{\rho_{S^{gc}_{1}}(z)}L'$};
	\draw (-3, 1.5) node {$\phi_{N}^{\rho_{S^{gc}_{1}}(z)}(L \circ_{H_{M}} \mathcal{L})$};
	\draw (2, 0) node {$(\phi_{M}^{\rho_{S^{gc}_{0}}(z)} \times \phi_{N}^{\rho_{S^{gc}_{1}}(z)}) \mathcal{L}$};
	\draw (3, 1.5) node {$\phi_{M}^{\rho_{S^{gc}_{0}}(z)}L$};
	\draw (3, 0.75) node {$M$};
	\draw (-3, 0) node {$N$};
\end{tikzpicture}
\caption{the quilted map defining the isomorphism under geometric composition}
\end{figure}
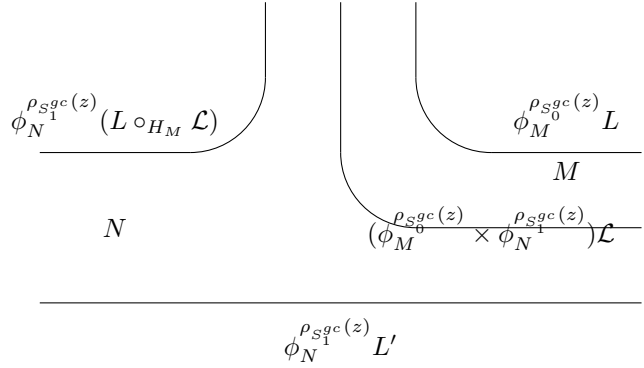

	We consider the moduli space $\mathcal{N}((\gamma^{0}, \gamma^{1}); e, \theta)$ of inhomogeneous pseudoholomorphic quilts $\underline{u}: \underline{S}^{gc'} \to (M, N)$, with Lagrangian seam and boundary conditions: $(u_{0}(x), u_{1}(x) \in \mathcal{L}$ for $x$ along the seam, and $u_{0}(I_{0, 0}) \subset L, u_{1}(I_{1, 0}) \subset L \circ_{H_{M}} \mathcal{L}$. The asymptotic conditions we put are the following: over the negative quilted end, $\underline{u}$ should converge to a $H_{N}$-chord $\theta$ in $N$ from $L \circ_{H_{M}} \mathcal{L}$ to $L'$; over the positive quilted end consisting of $(\epsilon_{0, 0}, \epsilon_{1, 1})$, $\underline{u}$ should converge to a generalized chord $(\gamma^{0}, \gamma^{1})$ for $(L, \mathcal{L}, L')$; over the other positive quilted end, $\underline{u}$ should converge to the generalized chord $e$ for $(L, \mathcal{L}, L \circ_{H_{M}} \mathcal{L})$ which corresponds to the unique Hamiltonian chord in $M_{2}$ from $L \circ_{H_{M}} \mathcal{L}$ to itself, representing the cohomological unit in $HW^{*}(L \circ_{H_{M}} \mathcal{L}, L \circ_{H_{M}} \mathcal{L})$. More precisely, we should also include a time-shifting function and use moving boundary conditions and Floer data that are compatible with the time-shifting function. Since these kinds of discussions have been carried out a lot of times, we will omit them here. \par
	The outcome is a cochain map
\begin{equation}
gc: CW^{*}(L, \mathcal{L}, L') \to CW^{*}(L \circ_{H_{M}} \mathcal{L}, L'),
\end{equation}
which induces the desired map on cohomology. \par
	To prove this cochain map induces an isomorphism on cohomology, we proceed in the following three steps. First recall from section \ref{cohomological unit} that the cohomological unit in $HW^{*}(L \circ_{H_{M}} \mathcal{L}, L \circ_{H_{M}} \mathcal{L})$ is represented by a unique Hamiltonian chord $\gamma_{min}$ from $L \circ_{H_{M}} \mathcal{L}$ to itself that is contained in the the interior of $N$, which has the biggest action compared to all other chords. Moreover, by the argument in section \ref{cohomological unit}, we have the quantitative estimate: for any $\epsilon > 0$, we may choose the admissible Hamiltonian $H_{N}$ on $N$ with its $C^{2}$-norm small enough in the interior, such that the action of $\gamma_{min}$ satisfies
\begin{equation*}
|\mathcal{A}_{H_{N}, L \circ_{H_{M}} \mathcal{L}, L \circ_{H_{M}} \mathcal{L}}(\gamma_{min})| \le \epsilon.
\end{equation*}
Moreover, we may arrange that the action of $\gamma_{min}$ is the biggest among all the $H_{N}$-chords from $L \circ_{H_{M}} \mathcal{L}$ to itself. \par
	Second, we need to provide a lower bound for the energy of the inhomogeneous pseudoholomorphic quilt $\underline{u}: \underline{S}^{gc} \to \underline{M}$. This is a general property of non-trivial pseudoholomorphic curves. We first fold the quilt to get an inhomogeneous pseudoholomorphic curve in the product manifold $M^{-} \times N$ with Lagrangian boundary conditions. Then we use a traditional trick due to Gromov to convert this inhomogeneous pseudoholomorphic curve to a pseudoholomorphic section of some locally trivial Hamiltonian fibration whose projection is holomorphic with respect to a unique almost complex structure on the total space determined by the almost complex structures and Hamiltonian functions chosen on the Liouville manifolds $M^{-}$ and $N$. If the original quilted map is not a trivial solution, then this pseudoholomorphic section is not contained in the horizontal distribution, and therefore attains some energy in the vertical direction. Gromov's monotonicity lemma then ensures minimal amount of energy, say $\delta$, that every non-trivial pseudoholomorphic curve has to attain. We may assume $\delta > \epsilon$, by choosing $\epsilon$ smaller if necessary. \par
	Finally, we consider the action filtrations (from below) on these two wrapped Floer cochain complexes. By a general action-energy equality \ref{action-energy equality for quilts}, combined with the previous two steps, it follows that the map $gc$ increases the action and therefore preserves the filtrations on these filtered cochain complexes. Inhomogeneous pseudoholomorphic quilts with minimal energy among all those with the same boundary conditions and asymptotic conditions at the quilted cylindrical end only increase the action by the smallest possible amount, and they must necessarily be trivial solutions, mapping the entire quilted surface into generalized chords, which identify the input and output generalized chords, and pass through the constraint given by the fundamental class of $L \circ_{H_{M}} \mathcal{L}$. This defines a map $gc_{0}$, which is precisely the natural identification \eqref{identification of generalized chords by energy-zero solutions}. Other inhomogeneous pseudoholomorphic maps of this kind have positive energy, and define another map $gc_{+}$ which strictly increase the action. Now $gc$ can be decomposed as $gc = gc_{0} + gc_{+}$, where $gc_{0}$ preserves the action and is an isomorphism of modules, and $gc_{+}$ strictly increases the action. From this viewpoint, $gc$ can be regarded as an upper triangular matrix whose diagonal entries are invertible, so by elementary algebra it is an isomorphism of modules. Therefore $gc$ is an isomorphism of cochain complexes, and in particular induces an isomorphism on cohomology. \par

\begin{remark}
	More generally, one would consider the geometric composition of two Lagrangian correspondences $\mathcal{L}_{01} \subset M_{0}^{-} \times M_{1}$ and $\mathcal{L}_{12} \subset M_{1}^{-} \times M_{2}$, and hope to prove a similar isomorphism:
\begin{equation*}
HW^{*}(L, \mathcal{L}_{01}, \mathcal{L}_{12}, L') \cong HW^{*}(L, \mathcal{L}_{01} \circ_{H_{M_{1}}} \mathcal{L}_{12}, L').
\end{equation*}
However, as mentioned in the section \ref{section: definition of quilted wrapped Floer cohomology} defining quilted wrapped Floer cohomology, there is certain essential difficulty in proving well-definedness of wrapped Floer cohomology for the pair $(\mathcal{L}_{01} \times L', L \times \mathcal{L}_{12})$ in the triple product $M_{0}^{-} \times M_{1} \times M_{2}^{-}$. This is why we did not discuss such an isomorphism in this paper. But since it is important for the purpose of understanding compositions of these functors, we plan to discuss this point in future research.
\end{remark}

\section{From split Hamiltonians to admissible Hamiltonians} \label{section: technical proofs}

\subsection{Outline of strategy}
	The entire section is devoted to the proof of Theorem \ref{invariance for wrapped Floer cohomology in the product} and Theorem \ref{preserving multiplicative structure}. We first give an outline of the strategy taken here. Note that the symplectic action functional on the space of paths from one Lagrangian submanifold to the other, associated to each Hamiltonian as well as the Lagrangian submanifolds together with their primitives, defines a filtration on the Floer cochain complexes. We can arrange that the truncated Floer cochain complexes $CW^{*}_{(-\infty, a]}(L_{0}, L_{1}; H, J)$ for a positive number $a$ captures all essential information, because we can set up wrapped Floer theory using Hamiltonians that are $C^2$-small in the interior part of the Liouville manifold. Such Hamiltonians are practically convenient for setting up wrapped Floer theory of a Liouville manifold, which turns out to be independent of the behavior of Hamiltonians in the interior part, as long as they are generic. In particular, the action of any time-one chord for such a Hamiltonian either negative or bounded above from a given positive number, which depends only on the Lagrangian submanifolds which the chord starts from and lands on. The main strategy is then to change the split Hamiltonian to an admissible one in several steps, while keeping track of the action of the new chords that possibly appear in each step. We will make sure that the potentially newly arising chords all have sufficiently positive action, so we can eventually rule them out in the truncated Floer cochain complexes as they do not contribute to Floer cohomology. In this section, by chords we always mean time-one Hamiltonian chords, unless we specify the time. \par
	Now consider the product Liouville manifold $M \times N$ with the product Liouville form $\lambda = \lambda_{M} \bigoplus \lambda_{N}$. Let $\mathcal{L}_{0}, \mathcal{L}_{1} \subset M \times N$ be Lagrangian submanifolds, which are
\begin{enumerate}[label=(\roman*)]

\item either conical with respect to the cylindrical end $\Sigma \times [1, +\infty)$ of $M \times N$ described in section \ref{product manifold}. We choose primitives $f_{\mathcal{L}_{0}}, f_{\mathcal{L}_{1}}$ such that they are locally constant over the cylindrical end $\Sigma \times [1, +\infty)$.

\item or product Lagrangian submanifolds $\mathcal{L}_{0} = L_{0} \times L'_{0}, \mathcal{L}_{1} = L_{1} \times L'_{1}$ of conical Lagrangian submanifolds $L_{0}, L_{1} \subset M$ and $L'_{0}, L'_{1} \subset N$. In this case the primitive $f_{\mathcal{L}_{i}}$ is the sum $\pi_{M}^{*}f_{L_{i}} + \pi_{N}^{*}f_{L'_{i}}$ of primitives for $L_{i}$ and $L'_{i}$ respectively, which are locally constant over the cylindrical ends of $M$ and $N$ respectively.

\item or one is conical with respect to $\Sigma \times [1, +\infty)$ and the other is a product.

\end{enumerate}
While those cases seem to be very different, it turns out that our analysis does not differ much in all cases. This is mainly due to the following (obvious) observation. \par

\begin{lemma}
	Let $\mathcal{L} = L \times L'$ be a product of conical Lagrangian submanifolds $L \subset M, L' \subset N$. Let $f_{L}$ be a primitive for $L$ which is locally constant over $\partial M \times [1, +\infty)$, and similarly for $f_{L'}$. Then $f_{\mathcal{L}} = \pi_{M}^{*}f_{L_{i}} + \pi_{N}^{*}f_{L'_{i}}$ is a primitive for $\mathcal{L}$ which is uniformly bounded in the cylindrical end $\Sigma \times [1, +\infty)$, with $\max|f_{\mathcal{L}}| \le \max(|f_{L}| + |f_{L'}|)$. In particular, it is locally constant in the region $(\partial M \times [1, +\infty)) \cup (\partial N \times [1, +\infty))$.
\end{lemma}

	For each time-dependent Hamiltonian function $H_{t}$, we can define the action functional:
\begin{equation}
\mathcal{A}_{H, \mathcal{L}_{0}, \mathcal{L}_{1}}(\Gamma) = - \int_{0}^{1} \Gamma^{*} \lambda + \int_{0}^{1} H(t, \Gamma(t))dt + f_{\mathcal{L}_{1}}(\Gamma(1)) - f_{\mathcal{L}_{0}}(\Gamma(0)),
\end{equation}
where $\Gamma$ is a time-one $H_{t}$-chord from $\mathcal{L}_{0}$ to $\mathcal{L}_{1}$. The action functional defines a filtration on the free graded module over the set of $X_{H_{t}}$-chords. And if $H_{t}$ has certain growth condition (not necessarily quadratic) such that the Floer complex $CW^{*}(\mathcal{L}_{0}, \mathcal{L}_{1}; H, J)$ is defined, it also defines a filtration on this complex. The same remark applies to the case where $H$ is time-independent which will be our primary concern in this section. \par
	We start from the split Hamiltonian $H_{M, N} = \pi_{M}^{*}H_{M} + \pi_{N}^{*}H_{N}$. Since $H_{M}$ and $H_{N}$ are chosen to be $C^2$-small in the interior, we may assume that there exists a $\epsilon > 0$ such that
$$| - \int_{0}^{1} \Gamma^{*} \lambda + \int_{0}^{1} H(\Gamma(t))dt | \le \epsilon$$
for every $(H_{M, N})$-chord $\Gamma$ that lies in the interior of $M \times N$ which can be identified with the product of the interior of $M$ with the interior of $N$. Fix a positive number $a>0$ which is bigger than $\max{(|f_{\mathcal{L}_{0}}| + |f_{\mathcal{L}_{1}}|)} + \epsilon$. Thus the truncated Floer cochain complex $CW^{*}_{(-b, a]}(\mathcal{L}_{0}, \mathcal{L}_{1}; H_{M, N}, J_{M, N})$ includes all Hamiltonian chords in the interior of $M \times N$. The goal is to find an admissible Hamiltonian $K$ as well as an almost complex structure $J$ of contact type such that for each $b$ we can construct the following sequence of cochain maps:
\begin{equation}\label{act-res}
\begin{split}
CW^{*}_{(-b, a]}(\mathcal{L}_{0}, \mathcal{L}_{1}; H_{M, N}, J_{M, N}) \to CW^{*}_{(-b, a]}(\mathcal{L}_{0}, \mathcal{L}_{1}; K, J) \\
\to CW^{*}_{(-5b, a]}(\mathcal{L}_{0}, \mathcal{L}_{1}; H_{M, N}, J_{M, N}) \to  CW^{*}_{(-5b, a]}(\mathcal{L}_{0}, \mathcal{L}_{1}; K, J)
\end{split}
\end{equation}
with the property that the following two compositions
\begin{align*}
CW^{*}_{(-b, a]}(\mathcal{L}_{0}, \mathcal{L}_{1}; H_{M, N}, J_{M, N}) &\to CW^{*}_{(-5b, a]}(\mathcal{L}_{0}, \mathcal{L}_{1}; H_{M, N}, J_{M, N}), \\
CW^{*}_{(-b, a]}(\mathcal{L}_{0}, \mathcal{L}_{1}; K, J) &\to CW^{*}_{(-5b, a]}(\mathcal{L}_{0}, \mathcal{L}_{1}; K, J)
\end{align*}
are inclusions with respect to the corresponding action filtrations. This is enough to establish the desired isomorphism between the two Floer cohomology groups. \par

\subsection{The first step}
	We first deformation the Hamiltonian $H_{M, N}$ so that it becomes constant outside a compact set. The key is to choose the constant and the compact set carefully so that the additional chords that could possibly appear after change of Hamiltonians have sufficiently postive action, which can then be excluded from the truncated Floer complexes and do not contribute to Floer cohomology. \par
	We review conditions on the two Hamiltonians $H_{M}$ and $H_{N}$. Recall that we have chosen $H_{M}$ such that $H_{M}(r_{1}) = r_{1}^2$ for $r_{1} \ge 1 + \epsilon$, and in the region $\{r_{1} \le {1}\}$, $H_{M}$ is $C^2$-small taking values in $[-\epsilon, 0]$, and has $C^1$-norm less than $\epsilon$. Moreover, we can assume the derivative $\frac{\partial H_{M}}{\partial r_{1}} \le \frac{4}{\epsilon}$, since the increasing amount of the function $H_{M}$ from $r_{1} = 1$ to $r_{1} = 1 + \epsilon$ is not bigger than $(1 + \epsilon)^{2} - \epsilon$. We impose the same conditions on $H_{N}$. \par
	We know that for Hamiltonians that increase over the cylindrical end having growth rate bigger than linear functions, the chords will have more negative action as the level they lie in increases. So given $b>0$, we can choose $A$ sufficiently large such that all chords (in the product, of the original split Hamiltonian $H_{M, N}$) initially having action greater than $-b$ are contained in the "square" region $\{r_{1} \le A, r_{2} \le A\}$. Then we deform it to a new Hamiltonian $H_1 = H_{M,1} + H_{N,1}$ which is still of split type (namely $H_{M,1}$ depends only on $M$-factor and $H_{N,1}$ depends only on $N$-factor), such that
\begin{equation}
\begin{split}
H_{M,1} =
\begin{cases}
H_{M}(r_1) & \text{if $r_1 \le A-\epsilon$}, \\
C & \text{if $r_1 \ge A$}.
\end{cases}
H_{N,1} = 
\begin{cases}
H_{N}(r_1) & \text{if $r_2 \le A-\epsilon$}, \\
C & \text{if $r_2 \ge A$}.
\end{cases}
\end{split}
\end{equation}
Here we are going to take 
\begin{equation}
C = (A-\frac{3\epsilon}{4})^2.
\end{equation}
Moreover, we require that
\begin{equation}
\begin{cases}
\frac{\partial H_{M,1}}{\partial r_1} \le \frac{2A}{3}, & r_1 \in [A-\epsilon, A];\\
\frac{\partial H_{N,1}}{\partial r_2} \le \frac{2A}{3}, & r_2 \in [A-\epsilon, A].
\end{cases}
\end{equation}
These requirements are possible since total variation of $H_{M,1}$ or $H_{N,1}$ is $\frac{\epsilon A}{2} - \frac{7\epsilon^2}{16}$ in the region $r_{1} \in [A - \epsilon, A]$ or $r_{2} \in [A - \epsilon, A]$. \par

	To illustrate what indeed happens, we should think of both the Hamiltonians $H_{M}$ and $H_{N}$ are being deformed to the constant function $C$ outside a compact set bounded by $A$ in exactly the same way (as functions on the half ray $[1, +\infty)$. We arrange such deformation carefully to make sure the resulting Hamiltonians are still increasing, but grow in a tempered way in the region $r_{1}, r_{2} \in [A - \epsilon, A]$. Concretely, we impose certain bounds on derivatives of them with respect to the redial coordinates $r_{1}$ and respectively $r_{2}$ within that range. \par
	Now the new Hamiltonian $H_{M, 1} + H_{N, 1}$ might have addtional time-one chords from $\mathcal{L}_{0}$ to $\mathcal{L}_{1}$ compared to $H_{M, N}$. There are four types of them:
\begin{enumerate}[label=(\roman*)]

\item Constant chords on levels $H_{1} = 2C$ (if they exist, they correspond to intersection points of $\mathcal{L}_{0}$ with $\mathcal{L}_{1}$). These have action
\begin{equation}
\mathcal{A}_{H_{1}, \mathcal{L}_{0}, \mathcal{L}_{1}}(\Gamma) = 2C + f_{\mathcal{L}_{1}}(\Gamma(1)) - f_{\mathcal{L}_{0}}(\Gamma(0)),
\end{equation}
which is sufficiently positive if we choose $A$ sufficiently large, as $C = (A-\frac{3\epsilon}{4})^2
$. In fact, these constant chords cannot exist generically, because the Lagrangians are disjoint over the cylindrical ends, except the case $\mathcal{L}_{0} = \mathcal{L}_{1}$. Even $\mathcal{L}_{0} = \mathcal{L}_{1}$ does not matter, because we may slightly perturb one to geometrically separate them at infinity, without affecting wrapped Floer cohomology.

\item Chords in $M \times N$ that project to Reeb chords in $M$ on levels $\partial M \times \{r_{1}\}$ for $r_1$ close to $A$ ($r_{1} \in [A-\epsilon, A]$), and project to constant chords in $N$ on levels $H_{N,1} = C$. These chords have action
\begin{equation}
\begin{split}
\mathcal{A}_{H_{1}, \mathcal{L}_{0}, \mathcal{L}_{1}}(\Gamma) &= -\frac{\partial H_{M,1}}{\partial r_1}r_1 + H_{M,1}(r_1) + C\\
&+ f_{\mathcal{L}_{1}}(\Gamma(1)) - f_{\mathcal{L}_{0}}(\Gamma(0))\\
&\ge -\frac{2A^2}{3} + (A - \epsilon)^2 + C\\
&+ f_{\mathcal{L}_{1}}(\Gamma(1)) - f_{\mathcal{L}_{0}}(\Gamma(0))\\
&\gg 0,
\end{split}
\end{equation}
if $A$ is chosen sufficiently large (recall the primitives $f_{\mathcal{L}_{0}}, f_{\mathcal{L}_{1}}$ are locally constant over the cylindrical end). We will not repeat this kind of sentence, unless special care is needed.

\item Chords in $M \times N$ that project to constant chords in $M$ on levels $H_{M,1}=C$, and Reeb chords in $N$ on levels $\partial N \times \{r_2\}$ for $r_2$ close to $A$ ($r_2 \in [A-\epsilon, A]$). This case is symmetric to the above one.

\item Chords in $M \times N$ that project to Reeb chords in $M$ on levels $\partial M \times \{r_1\}$ for $r_1$ close to $A$ ($r_1 \in [A-\epsilon, A]$), and project to Reeb chords in $N$ on levels $\partial N \times \{r_2\}$ for $r_2$ close to $A$ ($r_2 \in [A-\epsilon, A]$). These chords have action
\begin{equation}
\begin{split}
\mathcal{A}_{H_{1}, \mathcal{L}_{0}, \mathcal{L}_{1}}(\Gamma) &= -\frac{\partial H_{M,1}}{\partial r_1}r_1 + H_{M,1}(r_1) + -\frac{\partial H_{M,2}}{\partial r_2}r_2 + H_{M,2}(r_2)\\
&+ f_{\mathcal{L}_{1}}(\Gamma(1)) - f_{\mathcal{L}_{0}}(\Gamma(0))\\
&\ge -\frac{2A^2}{3} + (A - \epsilon)^2 -\frac{2A^2}{3} + (A - \epsilon)^2\\
&+ f_{\mathcal{L}_{1}}(\Gamma(1)) - f_{\mathcal{L}_{0}}(\Gamma(0))\\
&\gg 0
\end{split}
\end{equation}

\end{enumerate}\par

	To summarize, we have deformed the split Hamiltonian $H_{M, N}$ to a new one so that the undesired Hamiltonian chords all have sufficiently positive action, which therefore are excluded from the action filtration window $(-b, a]$. \par 

\subsection{The second step}
	Now the Hamiltonian $H_1 = H_{M,1} + H_{N,1}$ (split-type) on $M \times N$ is constant equal to $2C$ in the region $\{r_1 \ge A, r_2 \ge A\}$. We want to deform it to a Hamiltonian that is constant outside the compact set $\{r_1 \le B, r_2 \le B\}$ for appropriate choice of $B > A$ to be determined. \par
	By the first step, this is already the case in the region $\{r_1 \ge A, r_2 \ge A\}$. Consider the following two regions:

\begin{gather}
I = M \times (\partial N \times [A, +\infty)), \\
II = (\partial M \times [A, +\infty)) \times N.
\end{gather}
We first deal with the region $I$. The other case is symmetric. \par

	The idea is to find an appropriate way to deform $\pi_{M}^{*}H_{M, 1}$ to a new function $H_{I, 2}$ on the region $I$ by interpolating it with a suitable constant.

\begin{definition}
Define a smooth cut-off function $\rho: [A, +\infty) \to [0,1]$ as follows.

\begin{equation}
\rho = 
\begin{cases}
0 & \text{on $[A, A_{1}]$}, \\
1 & \text{on $[B - \epsilon, +\infty)$},
\end{cases}
\end{equation}
and is strictly increasing on $[A_1, B-\epsilon]$, where $A_1 > A$ is a big postive number to be chosen later. Moreover, we require that for $r_2 \in [A_{1} + \epsilon, B-2\epsilon]$, the derivative of $\rho$ is constant (namely $\rho$ is linear), and this constant satisfies

\begin{equation}
\rho'(r_2) \equiv \text{constant} \in [\frac{1}{B-A_{1}-\epsilon}, \frac{1}{B-A_{1}-3\epsilon}].
\end{equation}
\end{definition}

	These assumptions on the cut-off function $\rho$ will imply the following consequences:

\begin{lemma}
\begin{enumerate}[label=(\roman*)]
\item The amount that $\rho$ increases on the interval $[A_{1} + \epsilon, B - 2\epsilon]$ is between $1-\frac{2\epsilon}{B - A_{1}-\epsilon}$ and $1$;

\item The sum of the total variation of $\rho$ on $[A_{1}, A_{1} + \epsilon]$ and on $[B - 2\epsilon, B - \epsilon]$ is between $0$ and $\frac{2\epsilon}{B - A_{1} - \epsilon}$;

\item For $r_2 \in [A_1, A_{1} + \epsilon]$, we have that $0 \le \rho(r_2) \le \frac{2\epsilon}{B-A_{1}-\epsilon}$. And we can choose $\rho$ growing in a tempered way such that the derivative $\rho'(r_2) \le \frac{3}{B-A_{1}-\epsilon}$ for all $r_2 \in [A_1, A_{1} + \epsilon]$;

\item For $r_2 \in [B-2\epsilon, B-\epsilon]$, we have that $1-\frac{2\epsilon}{B-A_{1}-\epsilon} \le \rho(r_2) \le 1$. And we can choose $\rho$ such that the derivative $\rho'(r_2) \le \frac{3}{B-A_{1}-\epsilon}$ for all $r_2 \in [B-2\epsilon, B-\epsilon]$.

\end{enumerate}
\end{lemma}
\begin{proof}
	The proofs of all the four statements are elementary calculations.
\end{proof}

	Now we define the function $H_{I, 2}$ as follows:
\begin{gather}
H_{I, 2}: M \times \partial N \times [A, +\infty) \to \mathbb{R},\\
H_{I, 2}(x, y, r_2) = (1-\rho(r_2))H_{M,1}(x) + \rho(r_2)C.
\end{gather} \par

	To see what Hamiltonian chords can possibly arise and to estimate their action, first of all we need to find the Hamiltonian vector field for $H_{I, 2}$. The symplectic form on $M \times \partial N \times [A, +\infty)$ is of the form $\omega_{M} \bigoplus d(r_2 \lambda_{N}|_{\partial N})$ where $\lambda_{N}|_{\partial N}$ is the contact form on $\partial N$. Thus we find that the Hamiltonian vector field of $H_{I, 2}$ has the form:

\begin{equation}
X_{H_{I, 2}}(x, y, r_2) = (1-\rho(r_2)) X_{H_{M,1}}(x) - (C-H_{M,1}(x)) \rho'(r_2)Y_{\partial N}(y),
\end{equation}
where $Y_{\partial N}(y)$ is the Reeb vector field for $\lambda_{N}|_{\partial N}$ on $\partial N \times \{1\}$. \par
	Again, the new Hamiltonian $H_{I,2} + H_{N,1}$ might have additional time-one chords. Consider such a time-one chord $\Gamma$ of $X_{H_{I,2}}$ from $\mathcal{L}_{0}$ to $\mathcal{L}_{1}$. Its projection to $M$ is a $X_{H_{M,1}}$-chord $\gamma$ of time-$(1-\rho(r_2))$ (which does not necessarily starts from and ends in Lagrangian submanifolds), along which $H_{M,1}$ is constant equal to $H_{M,1}(\gamma(0))$, by Hamilton's equation (since $H_{M,1}$ is time-independent). Since $X_{H_{I,2}}$ does not have $\frac{\partial}{\partial r_{2}}$-component, the chord $\Gamma$ lies on some level $\{r_2 = \text{constant}\}$. Hence its projection to $\partial N \times [A, +\infty) \subset N$ is a Reeb chord of $Y_{\partial N}$ of time $(C-H_{M,1}(\gamma(0)))\rho'(r_2)$ on level $\partial N \times \{r_2\}$ for some $r_2 \ge A$, but with opposite direction to the one determined by $Y_{\partial N}$. To summarize, a time-one $X_{H_{I,2}}$-chord $\Gamma$ corresponds to a pair $(\gamma, \theta)$ where $\gamma$ is a time-$(1-\rho(r_{2}))$ $X_{H_{M,1}}$-chord in $M$ and $\theta$ is a time-$(C-H_{M,1}(\gamma(0)))\rho'(r_{2})$ Reeb chord for $-Y_{\partial N}$, located on the level $\partial N \times \{r_{2}\}$. \par
	Now we compute the action of $\Gamma$. Note that we have only modified $H_{M,1} + H_{N,1}$ on the region $I=M \times \partial N \times [A, +\infty)$ to the new one $H_{I,2}+H_{N,1}$, so we should compute the action with respect to $H_{I,2}+H_{N,1}$. Furthermore, in the region $I$ the function $H_{N,1}$ is constant, so we can replace it by its value $C$ (the same $C$ as in the first step). A straightforward calculation by the definition of the action gives:

\begin{equation}
\begin{split}
\mathcal{A}_{H_{I,2}+H_{N,1}, \mathcal{L}_{0}, \mathcal{L}_{1}}(\Gamma) &= -\int_{0}^{1-\rho(r_2)} \gamma^{*}\lambda_{M} - \int_{0}^{(C-H_{M,1}(\gamma(0)))\rho'(r_2)} - r_{2} \theta^{*}\lambda_{N}|_{\partial N}\\
&+ \int_{0}^{1} H_{I,2}(\Gamma(t))dt + C + f_{\mathcal{L}_{1}}(\Gamma(1)) - f_{\mathcal{L}_{0}}(\Gamma(0))\\
& = - \int_{0}^{1-\rho(r_2)} \gamma^{*}\lambda_{M} + r_{2}\rho'(r_2)(C-H_{M,1}(\gamma(0))) + \rho(r_2)(C-H_{M,1}(\gamma(0)))\\
&+ H_{M,1}(\gamma(0)) + C + f_{\mathcal{L}_{1}}(\Gamma(1)) - f_{\mathcal{L}_{0}}(\Gamma(0))\\
& = - \int_{0}^{1-\rho(r_2)} \gamma^{*}\lambda_{M} + (1- \rho(r_2) -r_{2}\rho'(r_2))H_{M,1}(\gamma(0))\\
&+ (1+ \rho(r_2) + r_{2}\rho'(r_2))C + f_{\mathcal{L}_{1}}(\Gamma(1)) - f_{\mathcal{L}_{0}}(\Gamma(0)).
\end{split}
\end{equation}

	There are five possible classes of $X_{H_{M,1}}$-chords, which are listed below. And according to the class of the projection of $\Gamma$ to the $X_{H_{M,1}}$-chord $\gamma$, we estimate the action of $\Gamma$. The goal is to show that the action is sufficiently positive.

\begin{enumerate}[label=(\roman*)]

\item Short chords in the interior of $M$. Here we say these chord are short because the Hamiltonian $H_{M,1}$ is $C^2$-small there. For these chords, we have
\begin{equation}
\int_{0}^{1-\rho(r_2)} \gamma^{*}\lambda_{M} \le \epsilon.
\end{equation}

	Now there are three sub-cases to consider, depending on the value of $r_2$.

\begin{enumerate}[label=\bfseries (i \alph*):]

\item $r_2 \in [A, A_1]$, where $\rho(r_2) \equiv 0, \rho'(r_2) \equiv 0$. In this case, $X_{H_{I,2}}$ does not have $Y_{\partial N}$-component. This implies that
\begin{equation}
\begin{split}
\mathcal{A}_{H_{I,2}+H_{N,1}, \mathcal{L}_{0}, \mathcal{L}_{1}}(\Gamma) &\ge -\epsilon + H_{M,1}(\gamma(0)) + C + f_{\mathcal{L}_{1}}(\Gamma(1)) - f_{\mathcal{L}_{0}}(\Gamma(0))\\
&\ge C - 2\epsilon + f_{\mathcal{L}_{1}}(\Gamma(1)) - f_{\mathcal{L}_{0}}(\Gamma(0))\\
& \gg 0.
\end{split}
\end{equation}

\item $r_{2} \in [B-\epsilon, +\infty)$, where $\rho(r_2) \equiv 1, \rho'(r_2) \equiv 0$. In this case, $X_{H_{I,2}}$ is zero so $\Gamma$ is constant. Thus we have
\begin{equation}
\begin{split}
\mathcal{A}_{H_{I,2}+H_{N,1}, \mathcal{L}_{0}, \mathcal{L}_{1}}(\Gamma) &= H_{M,1}(\gamma(0)) + C + f_{\mathcal{L}_{1}}(\Gamma(1)) - f_{\mathcal{L}_{0}}(\Gamma(0))\\
&\ge C - \epsilon + f_{\mathcal{L}_{1}}(\Gamma(1)) - f_{\mathcal{L}_{0}}(\Gamma(0))\\
&\gg 0.
\end{split}
\end{equation}

\item $r_2 \in [A_1, B-\epsilon]$. Recall that $H_{M,1}$ is $C^2$-small, taking values in $[-\epsilon, 0]$. Therefore we have
\begin{align*}
r_{2}(C-H_{M,1}(\gamma(0)))\rho'(r_2) &\ge 0, \\
\rho(r_2)(C-H_{M,1}(\gamma(0))) &\ge 0
\end{align*}
Thus we obtain the estimate
\begin{equation}
\begin{split}
\mathcal{A}_{H_{I,2}+H_{N,1}, \mathcal{L}_{0}, \mathcal{L}_{1}}(\Gamma) &\ge - 2\epsilon + C + f_{\mathcal{L}_{1}}(\Gamma(1)) - f_{\mathcal{L}_{0}}(\Gamma(0))\\
&\gg 0.
\end{split}
\end{equation}

\end{enumerate}

\item Reeb chords on level $\partial M \times \{r_1\}$ for some $r_1$ close to $1$ ($r_1 \in [1, 1+\epsilon]$) where $H_{M,1}(r_1) \in [-\epsilon, (1+\epsilon)^2]$. Recall we have assumed that $\frac{\partial H_{M,1}}{\partial r_1} \le \frac{4}{\epsilon}$. For these chords, we have
\begin{equation}
\int_{0}^{1-\rho(r_2)} \gamma^{*}\lambda_{M} \le \frac{4}{\epsilon} (1-\rho(r_2)).
\end{equation}

	Now there are three sub-cases to consider, depending on the value of $r_2$.

\begin{enumerate}[label=\bfseries (ii \alph*):]

\item $r_2 \in [A, A_1]$, where $\rho(r_2) \equiv 0, \rho'(r_2) \equiv 0$. We have
\begin{equation}
\begin{split}
\mathcal{A}_{H_{I,2}+H_{N,1}, \mathcal{L}_{0}, \mathcal{L}_{1}}(\Gamma) &\ge -\frac{4}{\epsilon} + H_{M,1}(\gamma(0)) + C + f_{\mathcal{L}_{1}}(\Gamma(1)) - f_{\mathcal{L}_{0}}(\Gamma(0))\\
&\ge -\frac{4}{\epsilon} - \epsilon + C + f_{\mathcal{L}_{1}}(\Gamma(1)) - f_{\mathcal{L}_{0}}(\Gamma(0))\\
&\gg 0.
\end{split}
\end{equation}
Reminder: here $\epsilon > 0$ is small but fixed. So we can choose $A$ large ($C=(A-\frac{3\epsilon}{4})^2$) such that the above estimate holds.

\item $r_2 \in [B-\epsilon, +\infty)$, where $\rho(r_2) \equiv 1, \rho'(r_2) \equiv 0$. We have
\begin{equation}
\begin{split}
\mathcal{A}_{H_{I,2}+H_{N,1}, \mathcal{L}_{0}, \mathcal{L}_{1}}(\Gamma) &\ge -\frac{4}{\epsilon} + 2C + f_{\mathcal{L}_{1}}(\Gamma(1)) - f_{\mathcal{L}_{0}}(\Gamma(0))\\
&\gg 0.
\end{split}
\end{equation}

\item $r_2 \in [A_1, B-\epsilon]$. Since $H_{M,1}$ takes values between $-\epsilon$ and $(1+\epsilon)^2$, we still have 
\begin{align*}
r_{2}(C-H_{M,1}(\gamma(0)))\rho'(r_2) &\ge 0, \\
\rho(r_2)(C-H_{M,1}(\gamma(0))) &\ge 0
\end{align*}
Thus we obtain the estimate
\begin{equation}
\begin{split}
\mathcal{A}_{H_{I,2}+H_{N,1}, \mathcal{L}_{0}, \mathcal{L}_{1}}(\Gamma) &\ge -\frac{4}{\epsilon} - \epsilon + C + f_{\mathcal{L}_{1}}(\Gamma(1)) - f_{\mathcal{L}_{0}}(\Gamma(0))\\
&\gg 0.
\end{split}
\end{equation}

\end{enumerate}

\item Reeb chords on level $\partial M \times \{r_1\}$ for some $r_1 \in [1+\epsilon, A - \epsilon]$. For these chords, we have
\begin{equation}
\int_{0}^{1-\rho(r_2)} \gamma^{*}\lambda_{M} = 2r_{1}^{2}(1-\rho(r_2)).
\end{equation}

	Now there are three sub-cases to consider, depending on the value of $r_2$.

\begin{enumerate}[label=\bfseries (iii \alph*):]

\item $r_2 \in [A, A_1]$, where $\rho(r_2) \equiv 0, \rho'(r_2) \equiv 0$. We have
\begin{equation}
\begin{split}
\mathcal{A}_{H_{I,2}+H_{N,1}, \mathcal{L}_{0}, \mathcal{L}_{1}}(\Gamma) &= -2r_{1}^2 + H_{M,1}(r_1) + C + f_{\mathcal{L}_{1}}(\Gamma(1)) - f_{\mathcal{L}_{0}}(\Gamma(0))\\
&= C - r_{1}^2 + f_{\mathcal{L}_{1}}(\Gamma(1)) - f_{\mathcal{L}_{0}}(\Gamma(0))\\
&\ge C - (A - \epsilon)^2 + f_{\mathcal{L}_{1}}(\Gamma(1)) - f_{\mathcal{L}_{0}}(\Gamma(0))\\
&= \frac{\epsilon A}{2} - \frac{7\epsilon^2}{16} + f_{\mathcal{L}_{1}}(\Gamma(1)) - f_{\mathcal{L}_{0}}(\Gamma(0))\\
&\gg 0.
\end{split}
\end{equation}

\item $r_2 \in [B-\epsilon, +\infty)$, where $\rho(r_2) \equiv 1, \rho'(r_2) \equiv 0$. We have
\begin{equation}
\begin{split}
\mathcal{A}_{H_{I,2}+H_{N,1}, \mathcal{L}_{0}, \mathcal{L}_{1}}(\Gamma) &= C - H_{M,1}(\gamma(0)) + H_{M,1}(\gamma(0)) + C + f_{\mathcal{L}_{1}}(\Gamma(1)) - f_{\mathcal{L}_{0}}(\Gamma(0))\\
&= 2C + f_{\mathcal{L}_{1}}(\Gamma(1)) - f_{\mathcal{L}_{0}}(\Gamma(0))\\
&\gg 0.
\end{split}
\end{equation}

\item $r_2 \in [A_1, B-\epsilon]$. Since $C=(A-\frac{3\epsilon}{4})^2 > (A-\epsilon)^2$ and $H_{M,1}(\gamma(0)) \le (A-\epsilon)^2$, we still have
\begin{align*}
r_{2}(C-H_{M,1}(\gamma(0)))\rho'(r_2) &\ge 0, \\
\rho(r_2)(C-H_{M,1}(\gamma(0))) &\ge 0
\end{align*}
Hence,
\begin{equation}
\begin{split}
\mathcal{A}_{H_{I,2}+H_{N,1}, \mathcal{L}_{0}, \mathcal{L}_{1}}(\Gamma) &\ge -2r_{1}^{2}(1-\rho(r_2)) + r_{1}^{2} + C + f_{\mathcal{L}_{1}}(\Gamma(1)) - f_{\mathcal{L}_{0}}(\Gamma(0))\\
&\ge C - (A-\epsilon)^2 + f_{\mathcal{L}_{1}}(\Gamma(1)) - f_{\mathcal{L}_{0}}(\Gamma(0))\\
&= \frac{\epsilon A}{2} - \frac{7\epsilon^2}{16} + f_{\mathcal{L}_{1}}(\Gamma(1)) - f_{\mathcal{L}_{0}}(\Gamma(0))\\
&\gg 0.
\end{split}
\end{equation}

\end{enumerate}

\item Reeb chords on level $\partial M \times \{r_1\}$ for some $r_1$ close to $A$ ($r_1 \in [A-\epsilon, A]$, where $\frac{\partial H_{M,1}}{\partial r_1} \le \frac{2A}{3}$. So for these chords, we have
\begin{equation}
\int_{0}^{1-\rho(r_2)} \gamma^{*}\lambda_{M} \le \frac{2A}{3} r_{1}(1-\rho(r_2)) \le \frac{2A^2}{3}(1-\rho(r_2)).
\end{equation}

	Now there are five sub-cases to consider, depending on the value of $r_2$.

\begin{enumerate}[label=\bfseries (iv \alph*):]

\item $r_2 \in [A, A_1]$, where $\rho(r_2) \equiv 0, \rho'(r_2) \equiv 0$. We have
\begin{equation}
\begin{split}
\mathcal{A}_{H_{I,2}+H_{N,1}, \mathcal{L}_{0}, \mathcal{L}_{1}}(\Gamma) &\ge -\frac{2A^2}{3} + r_{1}^2 + C + f_{\mathcal{L}_{1}}(\Gamma(1)) - f_{\mathcal{L}_{0}}(\Gamma(0))\\
&\ge -\frac{2A^2}{3} + (A - \epsilon)^2 + C + f_{\mathcal{L}_{1}}(\Gamma(1)) - f_{\mathcal{L}_{0}}(\Gamma(0))\\
&\gg 0.
\end{split}
\end{equation}

\item $r_2 \in [B-\epsilon, +\infty)$, where $\rho(r_2) \equiv 1, \rho'(r_2) \equiv 0$. We have
\begin{equation}
\begin{split}
\mathcal{A}_{H_{I,2}+H_{N,1}, \mathcal{L}_{0}, \mathcal{L}_{1}}(\Gamma) &= H_{M,1}(\gamma(0)) + C + f_{\mathcal{L}_{1}}(\Gamma(1)) - f_{\mathcal{L}_{0}}(\Gamma(0))\\
&\ge (A-\epsilon)^2 + C + f_{\mathcal{L}_{1}}(\Gamma(1)) - f_{\mathcal{L}_{0}}(\Gamma(0))\\
&\gg 0.
\end{split}
\end{equation}

\item $r_2 \in [A_1, \frac{A_1+B}{2}]$. In this case, we can choose appropriate $\rho$ such that
\begin{equation}
1 - r_{2}\rho'(r_2) - \rho(r_2) \ge 0.
\end{equation}
So the action satisfies
\begin{equation}
\begin{split}
\mathcal{A}_{H_{I,2}+H_{N,1}, \mathcal{L}_{0}, \mathcal{L}_{1}}(\Gamma) &= (1 - r_{2}\rho'(r_2) - \rho(r_2))r_{1}^2 + \frac{2A}{3}(1-\rho(r_2))r_{1}\\
&+ (1 + r_{2}\rho'(r_2) + \rho(r_2))C + f_{\mathcal{L}_{1}}(\Gamma(1)) - f_{\mathcal{L}_{0}}(\Gamma(0))\\
&\ge C + f_{\mathcal{L}_{1}}(\Gamma(1)) - f_{\mathcal{L}_{0}}(\Gamma(0))\\
&\gg 0.
\end{split}
\end{equation}

\item $r_2 \in [\frac{A_1+B}{2}, B-2\epsilon]$. Recall that in this region, the derivative $\rho'(r_2) \equiv constant \in [\frac{1}{B-A_{1}-\epsilon}, \frac{1}{B-A_{1}-3\epsilon}]$. So we have
\begin{equation}
\frac{r_{2} - A_{1} - \epsilon}{B - A_{1} - \epsilon} \le \rho(r_2) \le \frac{2\epsilon}{B - A_{1} - \epsilon} + \frac{r_{2} - A_{1} - \epsilon}{B - A_{1} - 3\epsilon} \le \frac{r_{2} - A_{1} + \epsilon}{B - A_{1} - 3\epsilon},
\end{equation}
\begin{equation}
\frac{r_2}{B - A_{1} - \epsilon} \le r_{2}\rho(r_2) \le \frac{r_2}{B - A_{1} - 3\epsilon}.
\end{equation}
Thus we obtain
\begin{equation}
\begin{split}
1 - r_{2}\rho'(r_2) - \rho(r_2) &\ge 1 - \frac{2r_{2} - A_{1} + \epsilon}{B - A_{1} - 3\epsilon}\\
&\ge 1 - \frac{2B - A_{1} - 3\epsilon}{B - A_{1} - 3\epsilon}\\
&= -1 - \frac{A_{1} + 3\epsilon}{B - A_{1} - 3\epsilon},
\end{split}
\end{equation}
and also
\begin{equation}
\begin{split}
1 + r_{2}\rho'(r_2) + \rho(r_2) &\ge 1 + \frac{2r_{2} - A_{1} - \epsilon}{B - A_{1} -\epsilon}\\
&\ge 1 + \frac{2\frac{A_{1} + B}{2} - A_{1} - \epsilon}{B - A_{1} - \epsilon}\\
&= 1 + \frac{B - \epsilon}{B - A_{1} - \epsilon}\\
&= 2 + \frac{A_1}{B - A_{1} - \epsilon}.
\end{split}
\end{equation}
Thus the action satisfies
\begin{equation}
\begin{split}
\mathcal{A}_{H_{I,2}+H_{N,1}, \mathcal{L}_{0}, \mathcal{L}_{1}}(\Gamma) &\ge (-1 - \frac{A_{1} + 3\epsilon}{B - A_{1} - 3\epsilon})r_{1}^2 + (2 + \frac{A_1}{B - A_{1} - \epsilon})C\\
&+ f_{\mathcal{L}_{1}}(\Gamma(1)) - f_{\mathcal{L}_{0}}(\Gamma(0))\\
&\ge (-1 - \frac{A_{1} + 3\epsilon}{B - A_{1} - 3\epsilon})A^2 + (2 + \frac{A_1}{B - A_{1} - \epsilon})(A-\frac{3\epsilon}{4})^2\\
&+ f_{\mathcal{L}_{1}}(\Gamma(1)) - f_{\mathcal{L}_{0}}(\Gamma(0))
\end{split}
\end{equation}
We may choose $A_{1}, B$ sufficiently large such that $2 + \frac{A_1}{B - A_{1} - \epsilon}$ is much bigger than $1 + \frac{A_{1} + 3\epsilon}{B - A_{1} - 3\epsilon}$, which will then ensure that $\mathcal{A}_{H_{I,2}+H_{N,1}, \mathcal{L}_{0}, \mathcal{L}_{1}}(\Gamma) \gg 0$.
\item $r_2 \in [B-2\epsilon, B-\epsilon]$. Here we have
\begin{equation}
1 - \frac{2\epsilon}{B - A_{1} - \epsilon} \le \rho(r_2) \le 1.
\end{equation}

	This case is special - we realize that we need the following additional assumption on the choice of the function $\rho$:
\begin{assumption}
$0 \le \rho'(r_2) \le \frac{1}{B - A_{1} - \epsilon}$, for $r_2 \in [B-2\epsilon, B-\epsilon]$.
\end{assumption}

	This assumption does not conflict any previous estimates in which we do not use many assumptions on $\rho$. There are of course a lot of functions that satisfy this assumption. With this, we can estimate the action:
\begin{equation} \label{est: d5}
\begin{split}
\mathcal{A}_{H_{I,2}+H_{N,1}, \mathcal{L}_{0}, \mathcal{L}_{1}}(\Gamma) &= (1 - r_{2}\rho'(r_2) -\rho(r_2))r_{1}^2 + (1 + r_{2}\rho'(r_2) + \rho(r_2))C\\
&+ \frac{2A}{3} r_{1}(1-\rho(r_2))r_1 + f_{\mathcal{L}_{1}}(\Gamma(1)) - f_{\mathcal{L}_{0}}(\Gamma(0))\\
&\ge (1 - \frac{r_2}{B - A_{1} -1})r_{1}^2 + (1 + 1 - \frac{2\epsilon}{B - A_{1} - \epsilon})C\\
&+ f_{\mathcal{L}_{1}}(\Gamma(1)) - f_{\mathcal{L}_{0}}(\Gamma(0))\\
&\ge (-\frac{B - \epsilon}{B - A_{1} - \epsilon} + (2 - \frac{2\epsilon}{B - A_{1} - \epsilon})C\\
&+ f_{\mathcal{L}_{1}}(\Gamma(1)) - f_{\mathcal{L}_{0}}(\Gamma(0))\\
&\ge (-1 - \frac{A_{1}}{B - A_{1} - \epsilon})A^2 + (2 - \frac{2\epsilon}{B - A_{1} - \epsilon})(A - \frac{3\epsilon}{4})^2\\
&+ f_{\mathcal{L}_{1}}(\Gamma(1)) - f_{\mathcal{L}_{0}}(\Gamma(0))\\
&\gg 0
\end{split}
\end{equation}
for suitable choice of $A_{1}, B$ such that $(2 - \frac{2\epsilon}{B - A_{1} - \epsilon}) + (-1 - \frac{A_{1}}{B - A_{1} - \epsilon})$ is strictly bigger than zero. An easy computation shows that it suffices to require that $B > 2A_{1}$.

\end{enumerate}

\item Constant chords on level $\partial M \times \{r_1\}$ for $r_1 \ge A$, where $H_{M,1} \equiv C$. For these chords, we have
\begin{equation}
\int_{0}^{1-\rho(r_2)} \gamma^{*}\lambda_{M} = 0.
\end{equation}

	Now there are three sub-cases to consider, depending on the value of $r_2$.

\begin{enumerate}[label=\bfseries (v \alph*):]

\item $r_2 \in [A, A_1]$, where $\rho(r_2) \equiv 0, \rho'(r_2) \equiv 0$. We have
\begin{equation}
\begin{split}
\mathcal{A}_{H_{I,2}+H_{N,1}, \mathcal{L}_{0}, \mathcal{L}_{1}}(\Gamma) &= H_{M,1}(\gamma(0)) + C + f_{\mathcal{L}_{1}}(\Gamma(1)) - f_{\mathcal{L}_{0}}(\Gamma(0))\\
&= 2C + f_{\mathcal{L}_{1}}(\Gamma(1)) - f_{\mathcal{L}_{0}}(\Gamma(0))\\
&\gg 0.
\end{split}
\end{equation}

\item $r_2 \in [B-\epsilon, +\infty)$, where $\rho(r_2) \equiv 1, \rho'(r_2) \equiv 0$. We have
\begin{equation}
\begin{split}
\mathcal{A}_{H_{I,2}+H_{N,1}, \mathcal{L}_{0}, \mathcal{L}_{1}}(\Gamma) &= 2C + f_{\mathcal{L}_{1}}(\Gamma(1)) - f_{\mathcal{L}_{0}}(\Gamma(0))\\
&\gg 0.
\end{split}
\end{equation}

\item $r_2 \in [A_1, B-\epsilon]$. Again, we have
\begin{equation}
\begin{split}
\mathcal{A}_{H_{I,2}+H_{N,1}, \mathcal{L}_{0}, \mathcal{L}_{1}}(\Gamma) &= 2C + f_{\mathcal{L}_{1}}(\Gamma(1)) - f_{\mathcal{L}_{0}}(\Gamma(0))\\
&\gg 0.
\end{split}
\end{equation}
\end{enumerate}

\end{enumerate} \par

	We have finished deforming $H_{M,1} + H_{N,1}$ to $H_{I,2} + H_{N,1}$ whose additional time-one chords involved all have sufficiently positive action. We then perform a symmetric construction in the region $II=\partial M \times [A, +\infty) \times N$ deforming $H_{N,1}$ to $H_{II,2}$, which is constant for $r_2 \ge B$, while the addtional time-one chords all have sufficiently positive action. \par
	The upshot of this second step is that we get a Hamiltonian $H_{I,2} + H_{II,2}$ on $M \times N$ which agrees with the original split Hamiltonian $H_{M, N}$, and is constant equal to $2C$ outside the compact subset $\{r_1 \le B, r_2 \le B\}$. Also, the additional chords compared to the original Hamiltonian $H_{M, N}$ all have sufficiently positive action. \par
	So far, the requirements on the choice of the parameters $A, A_{1}, B$ are that $A$ be sufficiently large, and $B > 2A_{1}$. There might be further constraint coming from the third step to be carried out below, so we wait until finishing that step before we specify the choices. But indeed, we will choose $A_{1}$ and $B$ both proportional to $A$ by suitable constant multiples, in order to carry out the homotopy procedure systematically in section \ref{uniform}. \par

\subsection{The third step}
	We then deform the Hamiltonian $H_{I,2} + H_{II,2}$ to a Hamiltonian $K$ which is admissible for the natural choice of the cylindrical end of $M \times N$ described in section \ref{product manifold}. Namely, we want $K$ to have quadratic growth in the radial coordinate outside a compact set. To help better visualize the picture, we observe that according to the description of the cylindrical end in section \ref{product manifold}, the radial coordinate $r$ is roughly speaking $r_{1} + r_{2}$ (precisely, with little deviation caused by the two constants $\alpha, \beta$ as in section \ref{product manifold}, but which can be chosen to be arbitrarily close to $1$). However we will not use this fact in our argument. \par
	Note the Hamiltonian $H_{I,2} + H_{II,2}$ is constant equal to $2C$ outside the compact set $\{r_1 \le B, r_2 \le B\}$. In particular, this is true on $\Sigma \times [B, +\infty)$. We then deform it to a Hamiltonian $K$ on $\Sigma \times [B, +\infty)$ by a smooth cut-off function such that the following holds:

\begin{enumerate}[label=(\roman*)]

\item $K$ agrees with $H_{I,2} + H_{II,2}$ in the region $\{r \le B\}$. In particular, it agrees with the split Hamiltonian $H_{M, N}$ in the region $\{r_1 \le A - \epsilon, r_2 \le A - \epsilon\}$.

\item $K$ is convex and strictly increasing with respect to the radial coordinate on $\Sigma \times [B, +\infty)$.

\item For $r \ge B+\epsilon$, we have
\begin{equation}
K(z, r) = \frac{1}{5}r^2 + 2C - \frac{1}{5}(B + \frac{\epsilon}{2})^2.
\end{equation}

\item $K$ does not grow too fast in the region $B \le r \le B + \epsilon$. Specifically, we require that for $r \in [B, B + \epsilon]$,
\begin{equation}
0 \le \frac{\partial K}{\partial r} \le \frac{2B}{5}.
\end{equation}
This is possible, since the amount that $K$ increases from $r=B$ to $r=B+\epsilon$ is $\frac{B\epsilon}{5} - \frac{3\epsilon^2}{20}$.

\end{enumerate}

	With such a Hamiltonian $K = K_{b}$ (depending on $b$), we can construct our desired map. \par
	
\begin{lemma}
	For such a $K_{b}$, there is a well-defined homomorphism of modules of truncated Floer cochain groups:
\begin{equation}
R_{b}: CW^{*}_{(-b, a]}(\mathcal{L}_{0}, \mathcal{L}_{1}; H_{M, N}) \to CW^{*}_{(-b, a]}(\mathcal{L}_{0}, \mathcal{L}_{1}; K_{b})
\end{equation}
\end{lemma}
\begin{proof}
	The proof is based on analyzing the action of the $K$-chords, so that the action of additional chords that do not agree with Hamiltonian chords for $H_{I, 2} + H_{II, 2}$ is sufficiently large and therefore does not fall in the action filtration window $(-b, a]$. Thus we are able to define the desired map on truncated Floer complexes, which is basically the identity map (by identifying generators). \par
	There might be additional $K$-chords of action not smaller than $-b$, which are on the level hypersurface $\Sigma \times \{r\}$ for $r$ close to $B$, say $r \in [B, B + 2\epsilon]$. Let $\Gamma$ be such a chord, we estimate its action in the following two cases.
\begin{enumerate}[label=(\roman*)]

\item $r \in [B, B + \epsilon]$. In this case we have that
\begin{equation}
\begin{split}
\mathcal{A}_{K, \mathcal{L}_{0}, \mathcal{L}_{1}}(\Gamma) &\ge -\frac{2B}{5}r + 2C\\
&\ge -\frac{2B}{5}(B + \epsilon) + 2C\\
&= 2(A - \frac{3\epsilon}{4})^2 - \frac{1}{5}(2B^2 + 2B\epsilon).
\end{split}
\end{equation}

\item $r \in [B + \epsilon, B + 2\epsilon]$. In this case we have
\begin{equation}
\begin{split}
\mathcal{A}_{K, \mathcal{L}_{0}, \mathcal{L}_{1}}(\Gamma) &\ge -\frac{2}{5}r^2 + \frac{1}{5}r^2 + 2C - \frac{1}{5}(B + \frac{\epsilon}{2})^2\\
&\ge 2C - \frac{1}{5}r^2 - \frac{1}{5}(B+\frac{\epsilon}{2})^2\\
&\ge 2C - \frac{1}{5}(B + 2\epsilon)^2 - \frac{1}{5}(B+\frac{\epsilon}{2})^2\\
&= 2(A - \frac{3\epsilon}{4})^2 - \frac{1}{5}(2B^2 + 5B\epsilon \frac{17\epsilon^2}{4}).
\end{split}
\end{equation}

\end{enumerate}

	Now we specify the choice of $B$ and $A_{1}$ to make sure the above two estimates are sufficiently positive. Note $2 < \frac{3}{\sqrt{2}} < \sqrt{5}$ and $1 < \frac{3}{4\sqrt{2}} + \frac{1}{2} < \frac{3}{2\sqrt{2}}$. Let us take 
\begin{equation} \label{choice of constants}
\begin{split}
&B = \frac{3}{\sqrt{2}}A,\\
&A_{1} = (\frac{3}{4\sqrt{2}} + \frac{1}{2})A.
\end{split}
\end{equation}
These choices ensure the action of the additional chords are sufficiently positive, and also make the previous estimate \eqref{est: d5} valid. This finishes all the steps in deforming the split Hamiltonian to an admissilble one, so we obtain the desired homomorphism of modules
\begin{equation}\label{act-res b/a}
CW^{*}_{(-b, a]}(\mathcal{L}_{0}, \mathcal{L}_{1}; \pi_{M}^{*}H_{M,1} + \pi_{N}^{*}H_{N,1}) \to CW^{*}_{(-b, a]}(\mathcal{L}_{0}, \mathcal{L}_{1}; K)
\end{equation}
\end{proof}

	Regarding the other homomorphism
\begin{equation}
CW^{*}_{(-b, a]}(\mathcal{L}_{0}, \mathcal{L}_{1}; K) \to CW^{*}_{(-5b, a]}(\mathcal{L}_{0}, \mathcal{L}_{1}; \pi_{M}^{*}H_{M,1} + \pi_{N}^{*}H_{N,1}),
\end{equation}
it is well-defined because of from the fact that $(r_1 + r_2)^2 \ge r_{1}^2 + r_{2}^2$. In more detail, this comes from the fact that within the region $\{r_1 \le A - \epsilon, r_2 \le A - \epsilon\}$, we have not changed the Hamiltonian $\pi_{M}^{*}H_{M,1} + \pi_{N}^{*}H_{N,1})$, and the chords that lie outside this region all have action out of the filtration range, by our choices of constants $A$ and $B = \frac{3}{\sqrt{2}}A, C=(A - \frac{3\epsilon}{4})^2$. \par

\subsection{Almost complex structures}\label{change almost complex structure}
	The matter with almost complex structures is less subtle than that with Hamiltonians. This is a general feature of Floer theory on non-compact manifolds: even though the almost complex structures are required to be of contact type near infinity, there is still plenty of flexibility of perturbing them (for example this is how we can achieve transversality). Suppose that we have chosen regular $J_{M}$ (resp. $J_{N}$) for $H_{M}$ (resp. $H_{N}$) of contact type. According to our construction, $K$ agrees with $H_{M, N}$ for $r_1 \le A - \epsilon, r_2 \le A - \epsilon$, and is still split ($= H_{M,1} + H_{N,1}$ when $r_1 \le A_1, r_2 \le A_1$. We may then choose a regular $J$ for $K$ which is of the form $J_{M, N}$ when $r_1 \le A_1, r_2 \le A_1$, and is of contact type on $\Sigma \times [B + \epsilon, +\infty)$. \par
	In fact, it does no harm even if we require $J$ to be of split type when $r_1 \le B, r_2 \le B$. Regularity results still hold if we have chosen $J_{M}$ and $J_{N}$ generically. But we do not quite need to make such an asumption here. In any case, we do not have to worry too much about almost complex structures, as long as they are compatible with the symplectic form, and of contact type outside a compact set. \par

\subsection{A homotopy argument} \label{uniform}
	There is a minor unsatisfactory point, which is that the function $K$ we obtained, more specifically the number $B$ where the function $K$ starts growing quadratically, depends on the filtration number $b$. Thus, starting from the inclusion $CW^{*}_{(-b, a]}(\mathcal{L}_{0}, \mathcal{L}_{1}; H_{M, N}) \subset CW^{*}_{(-2b, a]}(\mathcal{L}_{0}, \mathcal{L}_{1}; H_{M, N})$, we get different Hamiltonians with respect to which wrapped Floer cochain spaces are defined - one might not be directly seen as a subcomplex of the other. However, from the viewpoint of homotopy invariance of wrapped Floer cohomology, this defect is irrelevant. And that suggests that the following procedure can be done. \par
	To indicate this dependence, let us denote by $B_{b}$ the choice we make for $B$ according to $b$, and denote by $K_{b}$ the Hamiltonian function $K$ which is $\frac{1}{5}r^2 + 2C - \frac{1}{5}(B+\frac{\epsilon}{2})^2$ for $r \ge B + \epsilon$. Also recall in our choices, $A_{1}, B$ are appropriate constant multiples of $A$ as in \eqref{choice of constants}, and $C=(A-\frac{3\epsilon}{4})^2$. In a word, every parameter essentially depends only on $A$, and they have very explicit relations. We should choose $A$ according to $b$ such that $A_{wb} = \sqrt{2w}A_{b}$. To arrive at a situation where we get wrapped Floer complexes that have canonical inclusion relations, we need to deform the Hamiltonians $K_{b}$ in a uniform way to a single Hamiltonian $K$. \par
	The way we construct this homotopy is as follows. Note that the time-$(-\ln B)$ (the minus sign means backward) Liouville flow rescales the Hamiltonian $K_{b}$ to be one that is quadratic outside a small neighborhood of the compact set $\{r \le 1\}$. We define the homotopy $K_{b, t} = \frac{K_{b}}{t^{2} B^{2}} \circ \phi^{-tB}$. Define a homotopy of almost complex structures in a similar way. The continuation map associated to this homotopy gives a cochain homotopy equivalence:
\begin{equation} \label{rescaling the Hamiltonian}
CW^{*}(\mathcal{L}_{0}, \mathcal{L}_{1}; K_{b}, J_{b}) \to CW^{*}(\mathcal{L}_{0}, \mathcal{L}_{1}; K_{b, 1}, J_{b, 1}).
\end{equation}
Note that $K_{b, 1}$ is quadratic outside a neighborhood of $\{r \ge 1\}$, but for different values of $b$ this might behave slightly differently inside of this region. Nonetheless, for each $b$ we can take another compactly-supported homotopy of Hamiltonians to make all $K_{b, 1}$ the same, i.e. independent of $b$. We denote the resulting Hamiltonian by $K$. Similarly, we modify the almost complex structure to obtain $J$ independent of $b$. Since the continuation map associated to compactly-supported homotopy of Hamiltonians/almost complex structures is a cochain homotopy equivalence, its composition with the previous cochain homotopy equivalence \eqref{rescaling the Hamiltonian} yields a cochain homotopy equivalence:
\begin{equation}
h_{b}: CW^{*}(\mathcal{L}_{0}, \mathcal{L}_{1}; K_{b}, J_{b}) \to CW^{*}(\mathcal{L}_{0}, \mathcal{L}_{1}; K, J)
\end{equation} \par

	Moreover, this cochain map increases action of Hamiltonian chords, hence preserves the action filtrations of the form $(-b, a]$ for $a$ fixed at the beginning and large enough for the Floer complex to be independent of $a$. This is because of the action-energy identity (applied to the family of Hamiltonians). This we get cochain homotopy equivalence
\begin{equation}
h_{b}: CW^{*}_{(-b, a]}(\mathcal{L}_{0}, \mathcal{L}_{1}; K_{b}, J_{b}) \to CW^{*}_{(-b,a]}(\mathcal{L}_{0}, \mathcal{L}_{1}; K, J)
\end{equation}
on the truncated wrapped Floer complexes. \par

\subsection{Isomorphism of wrapped Floer cohomologies}\label{iso on cohomology}
	By the previous construction and estimates, we obtain the homomorphisms of modules \eqref{act-res b/a}. In order to prove that the module homomorphism \eqref{act-res b/a} is a cochain map, for which purpose we need to understand how Floer differential and the map \eqref{act-res} affect the action of Hamiltonian chords. We know that going along an inhomogeneous pseudoholomorphic strip (Floer trajectory) decreases the action, the differential increases the action (because we are using cohomology), but there cannot be chords of action greater than $a$. So these two truncated graded modules are indeed cochain complexes with respect to the Floer differentials. Thus we get a diagram:
\begin{equation} \label{cochain map}
\begin{CD}
CW^{*}_{(-b, a]}(\mathcal{L}_{0}, \mathcal{L}_{1}; H_{M, N}, J_{M, N}) @>R>> CW^{*}_{(-b, a]}(\mathcal{L}_{0}, \mathcal{L}_{1}; K_{b}, J_{b})\\
@VVm^{1}V	@VVm^{1}V\\
CW^{*}_{(-b, a]}(\mathcal{L}_{0}, \mathcal{L}_{1}; H_{M, N}, J_{M, N}) @>R>> CW^{*}_{(-b, a]}(\mathcal{L}_{0}, \mathcal{L}_{1}; K_{b}, J_{b})
\end{CD}
\end{equation}

\begin{lemma} \label{proof of cochain map}
	The above diagram \eqref{cochain map} is commutative.
\end{lemma}
\begin{proof}
	The proof is based on the observation that the map $R$ is an inclusion (usually identity) map and does not change action of the chords. Since running along a Floer trajectory decreases the action, the differential increases the action (because we are using cohomology). It follows that the image of a generator in $CW^{*}_{(-b, a]}(\mathcal{L}_{0}, \mathcal{L}_{1}; H_{M, N}, J_{M, N})$ (which corresponds to a chord of action between $-b$ and $a$) under the Floer's differential $m^1=m^1_{H_{M, N}, J_{M, N}}$ is a $\mathbb{Z}$-linear combination of $H_{M, N}$-chords of action still bigger than $-b$. Moreover, there are no chords of action greater than $a$, so these chords still have action between $-b$ and $a$, which under the action-restriction map $R$ go to $K$-chords of action between $-b$ and $a$. \par
	Suppose that the two $H_{M, N}$-chords $\Gamma_{0}, \Gamma_{1}$ under the map $R$, which are identical to themselves but regarded as two $K$-chords, are connected by a Floer trajectory $(u, v)$ for $(K, J)$. We have to show that they are in fact connected by a uniquely corresponding Floer trajectory for $(H_{M, N}, J_{M, N})$. Thus it suffices to prove that the Floer trajectory for $(K, J)$ is indeed a Floer trajectory for $(H_{M, N}, J_{M, N})$, i.e. the Floer trajectory does not escape from the region $\{r_1 \le A - \epsilon, r_2 \le A - \epsilon\}$, where $(K, J)$ agrees with $(H_{M, N}, J_{M, N})$. Suppose the contrary, namely that the projection of $(u, v)$ to some factor (either $u$ or $v$) escapes outside of the level $A - \epsilon$. Let us suppose this is the case with $u$. Since the almost complex structure $J$ is split when $r_{1} \le A_{1}, r_{2} \le A_{1}$, the part of $u$ where it lies below the level $r_{1} = A_{1}$ satisfies Floer's equation defined by the datum $(H_{M, 1}, J_{M})$. So maximum principle implies that $u$ has to escape to some place where $r_{1} > A_{1}$, where $u$ might not satisfy Floer's equation as the Hamiltonian and almost complex structure are not of split type. \par
	Without loss of generality, we may assume that the Hamiltonian chords $\Gamma_{0}, \Gamma_{1}$ are non-constant, otherwise they are contained in the compact domain $M_{0} \times N_{0}$ and any Floer trajectory cannot at all escape from that (possibly slightly larger) domain. Also, under the genericity assumption, the projections of the two chords $\Gamma_{0}, \Gamma_{1}$ to $M$ are $H_{M}$-chords $\gamma_{0}, \gamma_{1}$, which can be assumed to be non-trivial. Otherwise if they are $H_{M}$-chords in the interior of $M$, then no inhomogeneous pseudoholomorphic strip connecting them can even escape outside the boundary $\partial M$. In that case, there is nothing to prove. \par
	So let us suppose that $\gamma_{0}, \gamma_{1}$ are non-trivial, and correspond to Reeb chords on some level hypersurfaces. This implies that in the place whenever $u$ satisfies Floer's equation, the intersection $u(Z) \cap (\partial M \times \{r_{1}\})$ is either empty or a non-trivial arc. Since we have assumed that $u$ escapes outside of the level $r_{1} = A_{1}$, we know that for every $r_{1} \in [A, A_{1}]$, $u \cap (\partial M \times \{r_1\})$ is a non-trivial arc. In particular, this is the case with $u \cap (\partial M \times \{A\})$ and $u \cap (\partial M \times \{A_{1}\})$. On the other hand, $u \cap \{A \le r_1 \le A_{1}\}$ is a $J_{M}$-holomorphic curve, because the Hamiltonian $H_{M,1}$ is constant there. Recall that $A_{1} = (\frac{3}{4\sqrt{2}} + \frac{1}{2})A$. The following lemma implies that if $A$ is chosen to be large enough, then the energy of $u$ is very large, which is not possible, because the energy of $u$ is bounded by that of the original Floer trajectory $(u, v)$ in $M \times N$ connecting the given two chords with fixed amount of action.
\end{proof}

\begin{lemma}
	Let $J_{M}$ be an almost complex structure on $M$ of contact type over the cylindrical end, and $d > 1$ a constant. Then there exists a constant $c = c(J_{M}, d) > 0$ depending only on the almost complex structure $J_{M}$, the constant $d$, such that the following holds. Let $S$ be a compact connected Riemann surface with boundary and corners $\partial S = \partial_{l} S \cup \partial_{n} S$, where the two boundary portions $\partial_{l} S$ and $\partial_{n} S$ can meet at the corner points. Let $f: S \to M$ be any $J_{M}$-holomorphic curve, which satisfies
\begin{enumerate}[label=(\roman*)]

\item $f(S) \subset \partial M \times [A, dA]$;

\item $f(\partial_{n} S) \cap (\partial M \times \{A\}), f(\partial_{n} S) \cap (\partial M \times \{dA\})$ are both non-empty;

\item There exists a connected component $C$ of $\partial_{l} S$, such that $f(C)$ is an arc in $\partial M \times [A, dA]$, with its two endpoints lying on $\partial M \times \{A\}$ and $\partial M \times \{dA\}$ respectively.

\end{enumerate}
Then we have
\begin{equation}
Area(f) = E_{J_{M}}(f)=\int_{S} \frac{1}{2}|df|_{J_{M}}^2 \ge cA.
\end{equation}
\end{lemma}
\begin{proof}
This essentially follows from Gromov's monotonicity lemma. Alternatively, this can be proved using inverse isoperimetric inequality.
\end{proof}

	Now we utilize some homological algebra technique to finish the proof of Theorem \ref{invariance for wrapped Floer cohomology in the product}. For this, let us organize the last small modifications in a more conceptual form. The action-restriction map is strictly speaking defined to be the composition of the map we constructed on truncated Floer cochain complexes, with the cochain complex isomorphism provided by conjugating by Liouville flow, and then with the continuation map defined by the $C^2$-small homotopy of the Hamiltonian which is supported in the interior of $M \times N$ bounded by $\Sigma$. Summarizing, we have the following diagram (omitting the obvious choices of almost complex structures):
\begin{equation}
\begin{CD}
CW^{*}_{(-b, a]}(\mathcal{L}_{0}, \mathcal{L}_{1}; H_{M, N}) @>R_{b}>> CW^{*}_{(-b, a]}(\mathcal{L}_{0}, \mathcal{L}_{1}; K_{b}) @>h_{b}>> CW^{*}_{(-b, a]}(\mathcal{L}_{0}, \mathcal{L}_{1}; K)\\
@VViV	@VViV	@VViV\\
CW^{*}_{(-2b, a]}(\mathcal{L}_{0}, \mathcal{L}_{1}; H_{M, N}) @>R_{2b}>> CW^{*}_{(-2b, a]}(\mathcal{L}_{0}, \mathcal{L}_{1}; K_{2b}) @>h_{2b}>> CW^{*}_{(-2b, a]}(\mathcal{L}_{0}, \mathcal{L}_{1}; K)\\
@VViV	@VViV	@VViV\\
\cdots
\end{CD}
\end{equation}
where $R_{b}, R_{2b}, \cdots$ are cochain isomorphisms, $h_{b}, h_{2b}, \cdots$ are cochain homotopy equivalences, the vertical arrows are all natural inclusions. By the nature of our construction, we have \par

\begin{lemma}
	The first square strictly commutes, the second square homotopy commutes. 
\end{lemma}
\begin{proof}
	The commutativity of the first square is clear because $K_{2b}$ agrees with $H_{M, N}$ on a larger region where $K_{b}$ agrees with $H_{M, N}$. The homotopy commutativity of the second square is a general feature of continuation maps.
\end{proof}

	Consider the composition $h_{wb} \circ R_{wb}$, which is a cochain homotopy equivalence for every $w = 1, 2, \cdots$. They fit into a diagram of directed systems of cochain complexes, where the maps homotopy commute. Therefore the homotopy directed limited homomorphism exists and is a cochain homotopy equivalence, which we call the action-restriction map and denote by $R$. \par
 
\subsection{Multiplicative structure} \label{section: proof of preserving multiplicative structure}
	Our last task is to prove Theorem \ref{preserving multiplicative structure}, which says that the action-restriction map $R$ preserves the multiplication structure on the wrapped Floer cohomology. There are two main difficulties. One is that the multiplication adds up the action of the chords plus the energy of the inhomogeneous pseudoholomorphic triangle, therefore yields maps between truncated wrapped Floer complexes with different values of truncation. The other is that there are various choices of Floer data involved. We will give detailed treatment below. \par

	Suppose we are given admissible Lagrangian submanifolds $\mathcal{L}_{0}, \mathcal{L}_{1}, \mathcal{L}_{2} \subset M \times N$. Wrapped Floer cohomology has multiplication structure:
\begin{equation} \label{multiplication map for split Hamiltonian}
\begin{split}
m^{2}&: CW^{*}(\mathcal{L}_{1}, \mathcal{L}_{2}; H_{M, N}, J_{M, N}) \otimes CW^{*}(\mathcal{L}_{0}, \mathcal{L}_{1}; H_{M, N}, J_{M, N})\\
&\to CW^{*}(\mathcal{L}_{0}, \mathcal{L}_{2}; H_{M, N}, J_{M, N}).
\end{split}
\end{equation}
This is defined as the composition
\begin{equation}
\begin{split}
m^{2}&: CW^{*}(\mathcal{L}_{1}, \mathcal{L}_{2}; H_{M, N}, J_{M, N}) \otimes CW^{*}(\mathcal{L}_{0}, \mathcal{L}_{1}; H_{M, N}, J_{M, N})\\
&\to CW^{*}(\phi^{2}\mathcal{L}_{0}, \phi^{2}\mathcal{L}_{2}; \frac{H_{M, N}}{2} \circ \phi^{2}, (\phi^{2})^{*}(J_{M, N})).
\end{split}
\end{equation}
with the isomorphism
\begin{equation}
\begin{split}
&CW^{*}(\phi^{2}\mathcal{L}_{0}, \phi^{2}\mathcal{L}_{2}; \frac{(H_{M, N})}{2} \circ \phi^{2}, (\phi^{2})^{*}(J_{M, N}))\\
&\cong CW^{*}(\mathcal{L}_{0}, \mathcal{L}_{2}; H_{M, N}, J_{M, N}).
\end{split}
\end{equation}
Recall that although the split Hamiltonian $H_{M, N}$ is not admissible in the usual sense on $M \times N$, we still have regularity and compactness for relevant moduli spaces of inhomogeneous pseudoholomorphic triangles, as long as suitable Floer data (of split type) are chosen. The argument is similar to that for defining the wrapped Floer cohomology using this split Hamiltonian, and is essentially because the split Hamiltonian also grows fast enough (but not too fast). To be more precise, we still have the maximum principle for the split Hamiltonian $H_{M, N}$. \par
	More concretely, as before we should choose a family of split Hamiltonians $H_{M, N, S}: S \to \mathcal{H}(M) \times \mathcal{H}(N)$ as part of the Floer data involved for the definition of the multiplication structure, which over the $0$-th strip-like end agrees with the constant family $\frac{H_{M, N}}{4} \circ \phi^{2}$ (constant with respect to the coordinate on $(-\infty, 0]$, and over the $1$-st and $2$-nd strip-like ends agrees with $H_{M, N}$. In order to prove Theorem \ref{preserving multiplicative structure}, this family should be chosen to be "essentially trivial", in the following sense. \par
	
\begin{definition} \label{essentially constant family}
	We say a family of split Hamiltonians $H_{M, N, S} = H_{M, S} + H_{N, S}$ is essentially constant, if the following conditions hold. There exists a smooth function $\rho: S \to [1, 2]$ which takes the value $1$ over the two positive strip-like ends, and $2$ over the negative strip-like ends, such that we can write $H_{M, S, s} = \frac{H_{M}}{\rho(s)^{2}} \circ \phi_{M}^{\rho(s)}$ and $H_{N, S, s} = \frac{H_{N}}{\rho(s)^{2}} \circ \phi_{N}^{\rho(s)}$.
\end{definition}

	It does not harm to consider only these kinds of families of Hamiltonians as part of Floer data for $S$, since transversality of the moduli space of inhomogeneous pseudoholomorphic triangles can be achieved by generically perturbing $J_{S}$ (see Lemma \ref{transversality for strips} and Lemma \ref{transversality for 3-pointed disks}), as long as the family of Hamiltonians is domain-dependent, and each member in this family satisfies non-degeneracy condition on Hamiltonian dynamics and Reeb dynamics, which we assumed at the beginning. The moduli space of inhomogeneous pseudoholomorphic triangles defined using this choice of Floer datum is then regular and can be compatified as before. In particular, we can use it to define the above-mentioned multiplication \eqref{multiplication map for split Hamiltonian}. \par
	Choose the filtration number $a>0$ greater than $\max{(|f_{\mathcal{L}_{0}}| + |f_{\mathcal{L}_{1}}| + |f_{\mathcal{L}_{2}}|)} + \epsilon$. Thus the truncated Floer cochain complexes $CW^{*}_{(-b, a]}(\mathcal{L}_{i}, \mathcal{L}_{j}; H_{M, N}, J_{M, N})$ includes all Hamiltonian chords in the interior of $M \times N$. Regarding the action filtration, we note by the action-energy identity \eqref{action-energy equality for multiplication structure} that it defines the following map on truncated Floer complexes.
\begin{equation}
\begin{split}
m^{2}&: CW^{*}_{(-b, a]}(\mathcal{L}_{1}, \mathcal{L}_{2}; H_{M, N}, J_{M, N}) \otimes CW^{*}_{(-b, a]}(\mathcal{L}_{0}, \mathcal{L}_{1}; H_{M, N}, J_{M, N})\\
&\to CW^{*}_{(-2b, a]}(\mathcal{L}_{0}, \mathcal{L}_{2}; H_{M, N}, J_{M, N}).
\end{split}
\end{equation}
The action of every output chord is still less than or equal to $a$, because there are no chords having higher action by our choice of $H_{M}$ and $H_{N}$. \par
	Similarly, we have multiplication map
\begin{equation}
\begin{split}
m^{2}&: CW^{*}_{(-b, a]}(\mathcal{L}_{1}, \mathcal{L}_{2}; K_{b}, J_{b}) \otimes CW^{*}_{(-b, a]}(\mathcal{L}_{0}, \mathcal{L}_{1}; K_{b}, J_{b})\\
&\to CW^{*}_{(-2b, a]}(\mathcal{L}_{0}, \mathcal{L}_{2}; K_{b}, J_{b}).
\end{split}
\end{equation}
However, we notice a small technical difficulty here: there is only the inclusion
\begin{equation*}
CW^{*}_{(-b, a]}(\mathcal{L}_{0}, \mathcal{L}_{2}; K_{b}, J_{b}) \to CW^{*}_{(-2b, a]}(\mathcal{L}_{0}, \mathcal{L}_{2}; H_{M, N}, J_{M, N}),
\end{equation*}
but that is not defined on $CW^{*}_{(-2b, a]}(\mathcal{L}_{0}, \mathcal{L}_{2}; K_{b}, J_{b})$. The way of resolving this issue is to define a new multiplication map
\begin{equation}
\begin{split}
m^{2}_{b, 2b}&: CW^{*}_{(-b, a]}(\mathcal{L}_{1}, \mathcal{L}_{2}; K_{b}, J_{b}) \otimes CW^{*}_{(-b, a]}(\mathcal{L}_{0}, \mathcal{L}_{1}; K_{b}, J_{b})\\
&\to CW^{*}_{(-2b, a]}(\mathcal{L}_{0}, \mathcal{L}_{2}; K_{2b}, J_{2b}).
\end{split}
\end{equation}
using the cochain homotopy equivalence $h_{b}$. Recall that $h_{b}$ is defined on the whole Floer cochain complexes, not necessarily restricted to the specific action range $(-b, a]$. In particular,
\begin{equation*}
h_{b}: CW^{*}_{(-2b, a]}(\mathcal{L}_{0}, \mathcal{L}_{2}; K_{b}, J_{b}) \to CW^{*}_{(-2b, a]}(\mathcal{L}_{0}, \mathcal{L}_{2}; K, J)
\end{equation*}
is defined, and is also a cochain homotopy equivalence. Consider also the cochain homotopy equivalence
\begin{equation*}
h_{2b}: CW^{*}_{(-2b, a]}(\mathcal{L}_{0}, \mathcal{L}_{2}; K_{2b}, J_{2b}) \to CW^{*}_{(-2b, a]}(\mathcal{L}_{0}, \mathcal{L}_{2}; K, J).
\end{equation*}
Consider the canonical cochain homotopy inverse $g_{2b}$ of $h_{2b}$, which is defined by using the backward homotopy of Hamiltonians that is used to define $h_{2b}$. Then the composition
$$g_{2b} \circ h_{b}: CW^{*}_{(-2b, a]}(\mathcal{L}_{0}, \mathcal{L}_{2}; K_{b}, J_{b}) \to CW^{*}_{(-2b, a]}(\mathcal{L}_{0}, \mathcal{L}_{2}; K_{2b}, J_{2b})$$
is a cochain homotopy equivalence, which gives the desired new multiplication map when composed with the original multiplication map. \par
	The failure of strict commutativity of the multiplication on $CW^{*}_{(-b, a]}(\mathcal{L}_{1}, \mathcal{L}_{2}; H_{M, N}, J_{M, N})$ and the multiplication $m^{2}_{b, 2b}$ is mainly caused by the use of homotopy of Hamiltonians. Fortunately, continuation maps are unique up to cochain homotopy, and any two cochain homotopies are related by a higher order homotopy, canonical up to higher order homotopies. Choosing a smooth deformation from one configuration to the other amounts to a family of geometric data parametrized by an additional parameter besides the domain parameter. Counting rigid elements in the parametrized moduli space defines the desired cochain homotopy:
\begin{equation} \label{cochain homotopy for action-restriction maps}
\begin{split}
R^{2}_{b, 2b}: &CW^{*}_{(-b, a]}(\mathcal{L}_{1}, \mathcal{L}_{2}; H_{M, N}, J_{M, N}) \otimes CW^{*}_{(-b, a]}(\mathcal{L}_{0}, \mathcal{L}_{1}; H_{M, N}, J_{M, N})\\
&\to CW^{*-1}_{(-2b, a]}(\mathcal{L}_{0}, \mathcal{L}_{2}; K_{2b}, J_{2b}),
\end{split}
\end{equation}
between $m^{2}_{b, 2b} \circ (R_{b} \otimes R_{b})$ and $R_{2b} \otimes m^{2}_{H_{M, N}}$. \par
	The detailed construction of the cochain homotopy $R^{2}_{b, 2b}$ is based on the idea of repeating the construction for each member of the domain-dependent family of Hamiltonians $H_{S}: S \to \mathcal{H}(M) \times \mathcal{H}(N)$ to get a domain-dependent smooth family of admissible Hamiltonians $K_{S}$, such that every $K_{S, s}$ is admissible i.e. is convex, increasing in the radial coordinate over the cylindrical end, and quadratic away from a compact set. It suffices to carry out this construction for an essentially trivial family of split Hamiltonians as before. Over each of the two positive strip-like ends, the family is a translation invariant family which agrees with the split Hamiltonian $H_{M, N}$, and we modify it to the admissible Hamiltonian $K_{b}$ as before and think of that as a translation invariant family over the strip-like end. Over the negative strip-like end, the family is also translation invariant and agrees with $\frac{H_{M, N}}{4} \circ \phi^{2}$, and we modify it to the admissible Hamiltonian $\frac{K_{2b}}{4} \circ \phi^{2}$. As for $s \in S$ not in the strip-like ends, the choice of $K_{S, s}$ is more flexible. We will not list all possible choices, but work with a natural one: we modify $H_{S, s} = \frac{H_{M, N}}{\rho(s)^{2}} \circ \phi^{\rho(s)}$ to $\frac{K_{\rho(s)b}}{\rho(s)^{2}} \circ \phi^{\rho(s)}$ (so in fact, this choice also matches up the previous modification over the strip-like ends). We can view this construction as a family version of action-restriction maps. Let us review some basic properties of the function $K_{\rho(s)b}$. Write $A_{s} = \sqrt{\rho(s)}A, A_{1, s} = \sqrt{\rho(s)}A_{1}, B_{s} = \sqrt{\rho(s)}B$ and $C_{s} = (A_{s} - \frac{3\epsilon}{4})^{2}$. For each $s \in S$, the function $K_{\rho(s)b}$ satisfies the following conditions:
\begin{enumerate}[label = (\alph*)]

\item $K_{\rho(s)b}$ is convex and non-decreasing in the radial coordinate $r$ on $\Sigma \times [1, +\infty)$, and agrees with $H_{S, s}$ on the region $\{r_{1} \le A_{s} - \epsilon, r_{2} \le A_{s} - \epsilon \}$;

\item In the region $\{A_{s} \le r_{1} \le A_{1, s}, A_{s} \le r_{2} \le A_{1, s}\}$, $K_{\rho(s)b}$ is constant equal to $2C_{s}$;

\item $K_{S, s}(z, r) = \frac{1}{5}r^{2} + 2C_{s} - (B_{s} + \epsilon)^{2}$ for $r \ge B_{s} + \epsilon$;

\end{enumerate}\par

	We also modify the $S$-dependent family of almost complex structures of split type $J_{M, N, S} = J_{M, S} \times J_{N, S}$, where for every $s$, $J_{M, S, s}$ and $J_{N, S, s}$ may be assumed to be almost complex structures of contact type, to a family of admissible almost complex structures $J_{S}$, in a similar way. But we allow further perturbations of the almost complex structures within the class of the admissible ones. \par
	In this way, we obtain a new Floer datum for the disk with three boundary punctures, which is used to define the map
\begin{equation*}
R_{2b} \circ m^{2}_{H_{M, N}, J_{M, N}} \circ (R_{b}^{-1} \otimes R_{b}^{-1}),
\end{equation*}
and has the same behavior as the one used to define $m^{2}_{b, 2b}$ over the strip-like ends, except that the latter involves additional compactly-supported homotopies of Hamiltonians. This family of data gives rise to a parametrized moduli space of inhomogeneous pseudoholomorphic triangles, and counting elements in the zero-dimensional part of the parametrized moduli space defines the map \eqref{cochain homotopy for action-restriction maps}. Then a standard gluing argument shows that this is a cochain homotopy between $m^{2}_{b, 2b} \circ (R_{b} \otimes R_{b})$ and $R_{2b} \otimes m^{2}_{H_{M, N}}$. \par
	To finish the proof of Theorem \ref{preserving multiplicative structure}, we need an argument from homological algebra. To save space, let us denote $CW^{*}(\mathcal{L}_{0}, \mathcal{L}_{1}; H_{M, N}, J_{M, N})$ by $CW^{*}(\mathcal{L}_{0}, \mathcal{L}_{1}; H)$, similar for other Lagrangian submanifolds and Hamiltonians. We have the following diagram:
\begin{equation} \label{cochain homotopy commutative diagram}
\begin{CD}
CW^{*}_{(-b, a]}(\mathcal{L}_{1}, \mathcal{L}_{2}; H) \otimes CW^{*}_{(-b, a]}(\mathcal{L}_{0}, \mathcal{L}_{1}; H) @>R^{2}_{b, 2b}>> CW^{*-1}_{(-2b, a]}(\mathcal{L}_{0}, \mathcal{L}_{2}; K_{2b}) \\
@VV i \otimes i V	@VV i V\\
CW^{*}_{(-2b, a]}(\mathcal{L}_{1}, \mathcal{L}_{2}; H) \otimes CW^{*}_{(-2b, a]}(\mathcal{L}_{0}, \mathcal{L}_{1}; H) @>R^{2}_{2b, 4b}>> CW^{*-1}_{(-4b, a]}(\mathcal{L}_{0}, \mathcal{L}_{2}; K_{4b})
\end{CD}
\end{equation}
To prove that the homotopy direct limit of $R^{2}_{b, 2b}$ exists as $b \to +\infty$, it suffices to show that $R^{2}_{b, 2b}$ is compatible with $R^{2}_{2b, 4b}$ under the natural inclusions (all denoted by $i$) with respect to the action filtration, i.e. the above diagram homotopy commutes. \par
	This can be easily seen by writing down all other possible diagrams related to the one above:

\begin{equation}
\xymatrix{
CW^{*}_{(-b, a]}(\mathcal{L}_{1}, \mathcal{L}_{2}; H) \otimes CW^{*}_{(-b, a]}(\mathcal{L}_{0}, \mathcal{L}_{1}; H) \ar[r]^-{m^{2}} \ar[d]^-{R_{b} \otimes R_{b}} \ar[rd]^-{R^{2}_{b, 2b}}_{+1} & CW^{*}_{(-2b, a]}(\mathcal{L}_{0}, \mathcal{L}_{2}; H) \ar[d]^{R_{2b}}\\
CW^{*}_{(-b, a]}(\mathcal{L}_{1}, \mathcal{L}_{2}; K_{b}) \otimes CW^{*}_{(-b, a]}(\mathcal{L}_{0}, \mathcal{L}_{1}; K_{b}) \ar[r]^-{m^{2}_{b, 2b}} \ar[d] & CW^{*}_{(-2b, a]}(\mathcal{L}_{0}, \mathcal{L}_{2}; K_{2b}) \ar[d]\\
CW^{*}_{(-2b, a]}(\mathcal{L}_{1}, \mathcal{L}_{2}; H) \otimes CW^{*}_{(-2b, a]}(\mathcal{L}_{0}, \mathcal{L}_{1}; H) \ar[r]^-{m^{2}} \ar[d]^-{R_{2b} \otimes R_{2b}} \ar[rd]^-{R^{2}_{2b, 4b}}_{+1} & CW^{*}_{(-4b, a]}(\mathcal{L}_{0}, \mathcal{L}_{2}; H) \ar[d]^-{R_{4b}}\\
CW^{*}_{(-2b, a]}(\mathcal{L}_{1}, \mathcal{L}_{2}; K_{2b}) \otimes CW^{*}_{(-2b, a]}(\mathcal{L}_{0}, \mathcal{L}_{1}; K_{2b}) \ar[r]^-{m^{2}_{2b, 4b}} & CW^{*}_{(-4b, a]}(\mathcal{L}_{0}, \mathcal{L}_{2}; K_{4b})
}
\end{equation}
To prove the diagram \eqref{cochain homotopy commutative diagram} is homotopy commutative, we recall that the definition of $R^{2}_{b, 2b}$ (and $R^{2}_{2b, 4b}$) mainly depends on the choice of a homotopy between the above-mentioned data involved in the definition of the two different multiplication maps. Two different choices of homotopies can be connected by a higher homotopy. This higher homotopy then defines a parametrized moduli space, using which we can construct the desired cochain homotopy between $R^{2}_{2b, 4b} \circ (i \otimes i)$ and $i \circ R^{2}_{b, 2b}$. \par

\begin{remark}
	Note this is not a statement about homotopy associativity, i.e. the $A_{\infty}$-structure. Rather, this is similar to the general feature of homotopy method in Floer theory - the resulting continuation maps are unique up to homotopy, and homotopies are unique up to higher homotopies. \par
	However, we can also phrase this in terms of $A_{\infty}$-associativity, by extending the maps to $A_{\infty}$-functors. This will be the topic of \cite{Gao}.
\end{remark}

	We then compose $R^{2}_{b, 2b}$ with the cochain homotopy equivalence
\begin{equation}
h_{2b}: CW^{*}_{(-2b, a]}(\mathcal{L}_{0}, \mathcal{L}_{2}; K_{2b}, J_{2b}) \to CW^{*}_{(-2b, a]}(\mathcal{L}_{0}, \mathcal{L}_{2}; K, J)
\end{equation}
and similarly $R^{2}_{2b, 4b}$ with $h_{4b}$. It is clear that $h_{2b}$ and $h_{4b}$ combined with obvious inclusion maps form a homotopy commutative diagram where the homotopy between the two compositions is also a choice arising in homotopy method in Floer theory. This fact, combined with homotopy commutativity of the diagram \ref{cochain homotopy commutative diagram} implies that the homotopy directed limit of $h_{2b} \circ R^{2}_{b, 2b}$ exists, which gives rise to a map of degree $-1$:
\begin{equation}
R^{2}: CW^{*}(\mathcal{L}_{0}, \mathcal{L}_{1}; H_{M, N}, J_{M, N}) \otimes CW^{*}(\mathcal{L}_{0}, \mathcal{L}_{1}; H_{M, N}, J_{M, N}) \to CW^{*-1}(\mathcal{L}_{0}, \mathcal{L}_{1}; K, J).
\end{equation}
Since the action restriction map is also defined as a homotopy direct limit, and at each finite stage $R^{2}_{b, 2b} \circ h_{2b}$ serves as a cochain homotopy between the apparent maps described as before, it follows by definition of homotopy directed limit that $R^{2}$ is a cochain homotopy between $m^{2}_{K, J} \circ (R \otimes R)$ and $R \otimes m^{2}_{H_{M, N}, J_{M, N}}$ as stated in the statement of Theorem \ref{preserving multiplicative structure}. \par
	
\subsection{Concluding remark} \label{section: concluding remark}
	We finally remark that this construction can be extended to studying higher order multiplications on wrapped Floer cochain complexes among an arbitrary but fixed finite collection of conical Lagrangian submanifolds, in principle without any trouble but requiring more careful work. The issue with the entire wrapped Fukaya category is that there are in principle infinitely many Lagrangian submanifolds to consider, while in our construction, the number $a$ chosen to filter the wrapped Floer cochain complex has to be dependent of the Lagrangian submanifolds involved, because we need to control the contribution of the primitives to the action. It does not seem easy to check the $A_{\infty}$-equation. This is the reason that we have only dealt with the cohomology categories of wrapped Fukaya categories so far, although the definition of the chain-level wrapped Fukaya category is quite standard. \par
	If the Liouville manifold is nondegenerate in the sense of \cite{Ganatra}, meaning that every object of the wrapped Fukaya category is in the idempotent image of a finite collection of Lagrangian submanifolds, then we can just work with this finite collection of Lagrangian submanifolds and use homological algebra to get the desired statement for the entire wrapped Fukaya category, or at least the category of perfect complexes over it. This is an evidence that the general statement that the two models of wrapped Fukaya categories are quasi-isomorphic should be true, which will be investigated in the upcoming work \cite{Gao}. \par

\newpage
\appendix

\section{Orientation}\label{appendix: orientation}

\subsection{Statement of the result}
	This appendix is provided to study orientations on the moduli space of quilted maps that we used in subsection \ref{section:representability} in which we prove that the geometric composition represents the quilted wrapped Floer module, in order for the map $gc$ to be defined over $\mathbb{Z}$. Now as the basic ingredients regarding orientations are well-known, we will not repeat those here. \par
	Let $L' = L \circ_{H_{M}} \mathcal{L}$ be the geometric composition, which is assumed to be a proper embedding. Denote by $\mathcal{N}((\gamma^{0}, \gamma^{1}), e; \theta)$ the moduli space of quilted inhomogeneous pseudoholomorphic maps, which asymptotically converge to generalized chord $(\gamma^{0}, \gamma^{1})$ at the positive quilted end, and to $H_{N}$-chord $\theta$ from $L'$ to itself at the negative strip-like end, and to the unique generalized chord $e$ for $(L, \mathcal{L}, L'$ that corresponds to the unit for wrapped Floer cohomology of $L'$, as in subsection \ref{section:representability}. \par

\begin{proposition}\label{coherent orientation on the moduli space of quilts}
	The moduli space $\mathcal{N}((\gamma^{0}, \gamma^{1}), e; \theta)$ carries a canonical orientation, determined by the isomorphism classes of spin structures on $L, \mathcal{L}$ and $L'$. Moreover, this orientation is consistent with the ones on the moduli spaces of inhomogeneous pseudoholomorphic quilted strips $\mathcal{M}(\gamma^{0}_{1}, \gamma^{1}_{1}), (\gamma^{0}, \gamma^{1}))$ as well as the ones on the moduli spaces of inhomogeneous pseudoholomorphic strips $\mathcal{M}(\theta, \theta_{1})$ with boundary on $(L', L')$, which have interaction with $\mathcal{N}((\gamma^{0}, \gamma^{1}), e; \theta)$ via gluing, giving neighborhoods of the boundary strata of the compactified moduli space of those quilted maps.
\end{proposition}

\subsection{The orientation line}
	Given a time-one Hamiltonian chord $\gamma$ from $L_{0}$ to $L_{1}$, there is an associated orientation line $o_{\gamma}$, which is determined by the isomorphism class of the spin structures on $L_{0}$ and $L_{1}$, which we fix at the beginning. The reader is referred to section 11 of \cite{Seidel} for the definitions. \par
	The general theory as in \cite{FOOO2}, and \cite{Seidel} which fits better into our framework, then relates the orientation lines to orientations of moduli spaces of punctured (inhomogeneous) pseudoholomorphic curves to a single symplectic manifold $M$ with asymptotic conditions converging to the given Hamiltonian chords. \par
	Given a generalized chord $(\gamma^{0}, \gamma^{1})$ for a triple $(L, \mathcal{L}, L')$, the associated orientation line is defined as follows. $(\gamma^{0}, \gamma^{1})$ corresponds to a unique time-one Hamiltonian chord in the product $M^{-} \times N$ from $\mathcal{L}$ to $L \times L'$, which is defined as $\Gamma(t) = (\gamma^{0}(1-t), \gamma^{1}(t))$. Now the spin structures on $\mathcal{L}$ and $L \times L'$ determine lifts of their tangent spaces from the Lagrangian Grassmannian to the brane Grassmannian. Choose any path $\tilde{\lambda}: [0, 1] \to \widetilde{\mathcal{LAG}}(T(M^{-} \times N))$ connecting the linear Lagrangian branes which are the lifts of the two Lagrangian subspaces $T_{(\gamma^{0}(1), \gamma^{1}(0))}\mathcal{L}$ and $T_{(\gamma^{0}(0), \gamma^{1}(1))}(L \times L')$. A priori, these Lagrangian subspaces live in different symplectic vector spaces, but can be put into the same one canonically using the parallel transport along the path given by the Hamiltonian chord $\Gamma$ with respect to the symplectic connection on $T(M^{-} \times N)$. Let $\lambda$ be the underlying path in the Lagrangian Grassmannian $\mathcal{LAG}(T(M^{-} \times N))$ connecting $T_{(\gamma^{0}(1), \gamma^{1}(0))}\mathcal{L}$ and $T_{(\gamma^{0}(0), \gamma^{1}(1))}(L \times L')$. \par
	 There is a unique homotopy class of trivialization of the pullback of $T(M^{-} \times N)$ by $\Gamma$ which is compatible with the Calabi-Yau structure, i.e. the trivialization of the canonical bundle $\Lambda^{top}T(M^{-} \times N)$. Let $H$ be the upper half-plane, which we think of as being a one-punctured disk equipped with a negative strip-like end near the puncture. Consider an inhomogeneous pseudoholomorphic map
\begin{equation}
w: H \to M^{-} \times N,
\end{equation}
defined with respect to the split Hamiltonian $H_{M, N}$ and product almost complex structure $-J_{M} \times J_{N}$, which asymptotically converges to $\Gamma$ over the negative strip-like end of $H$. Add an interval at infinity $\{-\infty\} \times [0, 1]$ to compactify $H$. Choose a symplectic trivialization of the pullback bundle $E = w^{*}T(M^{-} \times N)$ over $H$ which extends to the interval at infinity and is compatible with that of $\Gamma^{*}T(M^{-} \times N)$, thought of as a symplectic vector bundle on the interval at infinity. Specify a Lagrangian subbundle $F_{\lambda}$ of $E|_{\partial H}$ by
\begin{equation}
F_{\lambda, s} = \lambda(\psi(s)),
\end{equation}
where $\lambda$ is the path in the Lagrangian Grassmannian as before, while $\psi: \partial H = \mathbb{R} \to [0, 1]$ is a smooth non-decreasing function such that $\psi(s) = 0$ for $s \ll 0$, $\psi(s) = 1$ for $s \gg 1$ and $\psi'(s) > 0$ for all $s$ where $\psi(s) \in (0, 1)$, which can be chosen once-for-all and independent of everything. Denote by $\bar{E}$ the extension of $E$ to the compactification $\bar{H}$ of $H$, and let $\bar{F}$ be the Lagrangian subbundle of $\bar{E}|_{\bar{\partial H}}$ by asymptotic extension of the Lagrangian subbundle $F_{\lambda}$, where $\bar{\partial H}$ is the closure of $\partial H$ in $\bar{H}$, but not $\partial \bar{H}$. \par
	The linearized operator $D_{H, \lambda}$ of the inhomogeneous pseudoholomorphic curve equation for $w$ is then a Cauchy-Riemann operator on the bundle $E$ with Lagrangian boundary condition $F_{\lambda}$, with respect to the pullback almost complex structure $I_{E}$ and the connection $\nabla_{E}$ induced by the Hamiltonian vector field. Now deform the almost complex structure $I_{\bar{E}}$ on the extended bundle $\bar{E}$ to the constant complex structure on the trivial bundle with respect to the chosen trivialization, and deform $\nabla_{\bar{E}}$ to the trivial connection. Then the Cauchy-Riemann operator on the extended bundle $\bar{E}$ can be deformed through Fredholm operators to the standard Dolbeault operator on the trivial bundle over $\bar{H}$ with Lagrangian subbundle $\bar{F}$, now identified with a Lagrangian subbundle of the trivial bundle. For the standard Dolbeault operator we can easily define its determinant line, which should of course be canonically isomorphic to the determinant line of the operator $D_{H, \lambda}$. \par

\begin{definition}
	The orientation line $o_{(\gamma^{0}, \gamma^{1})}$ for the generalized chord $(\gamma^{0}, \gamma^{1})$ is the determinant line $det(D_{H, \lambda})$ of the Cauchy-Riemann operator $D_{H, \lambda}$.
\end{definition}

	This definition of orientation line is independent of choices of $\tilde{\lambda}$, up to canonical isomorphism, as proven in section 11 of \cite{Seidel}. \par
	The orientation line $o_{e}$ for the generalized chord $e$ is defined in a similar way. But let us consider the upper half-plane $\tilde{H}$ thought of as being equipped with a positive strip-like end instead of a negative one. Following the same construction, this would yield the definition of the inverse orientation line $o_{e}^{-}$, i.e. the dual of the orientation line. \par

\subsection{The linearized problem associated to the quilted map}
	We shall now try to adapt the general theory on orientations of moduli spaces of pseudoholomorphic curves to the pseudoholomorphic quilts that we used in subsection \ref{section:representability}. This begins with the study of the linearized operator associated to the inhomogeneous pseudoholomorphic quilted map. \par
	At a point $(u, v) \in \mathcal{N}((\gamma^{0}, \gamma^{1}), e; \theta)$, the tangent space to the moduli space is modeled on the Fredholm complex of a non-degenerate Cauchy-Riemann operator $D$, which is the linearized operator of the Floer's equations for $(u, v)$. More precisely, this Cauchy-Riemann operator acts on the space of pairs of sections of the bundle pair $(u^{*}TM, v^{*}TN)$ which satisfy Lagrangian boundary conditions given by linearizing the Lagrangian boundary conditions $(L, \mathcal{L}, L')$ for the quilted map $(u, v)$, and sends them to pairs of $(0,1)$-forms with values in the bundle pair. For a generic choice of Floer datum, the Fredholm operator is surjective, and therefore
\begin{equation}
\ker(D) \to 0
\end{equation}
is a finite-dimensional reduction and the tangent space is isomorphic to $\ker(D)$. An orientation is a choice of isomorphism class of the top exterior power $\Lambda^{top}\ker D$, i.e. the determinant line $\det(D)$. \par

\subsection{The choice of orientation}
	The orientation lines $o_{(\gamma^{0}, \gamma^{1})}$ and $o_{e}$ associated to the generalized chords $(\gamma^{0}, \gamma^{1})$ and respectively $e$ for $(L, \mathcal{L}, L')$ are defined with respect to the chosen trivializations of the pullbacks of $TM^{-}$ and $TN$. Again, the orientation lines are well-defined and determined by the spin structures on $\mathcal{L}$ and $L \times L'$, where the latter spin structure is the direct product of those on $L$ and $L'$. \par
	Suppose for the moment that we do not have that negative puncture on the second patch of the quilted surface. Then that can be regarded as a quilted disk with one positive quilted end and one negative quilted end. However, we should remember that the boundary of the second patch carries an additional marked point so that there is no translation symmetry on this marked quilted surface. By folding the quilt, we get a map defined on the two punctured disk with two punctures $z_{+}, z_{-}$ equipped with positive/negative strip-like ends to $M^{-} \times N$:
\begin{equation}
f: S = D \setminus \{z_{+}, z_{-}\} \to M^{-} \times N,
\end{equation}
with Lagrangian boundary condition given by the pair $(\mathcal{L}, L \times L')$. Denote by $D_{S, \nu}$ the linearized operator, where $\nu: \partial S \to \mathcal{LAG}(TM)$ is the Lagrangian boundary condition, which is locally constant on the strip-like ends. The last condition says near the puncture $z_{+}$ where the map asymptotically converges to $(\gamma^{0}, \gamma^{1})$, the linear Lagrangian boundary condition $\nu$ is equal to $(T_{(\gamma^{0}(1), \gamma^{1}(0))}\mathcal{L}, T_{(\gamma^{0}(0), \gamma^{1}(1))}(L \times L'))$ on each boundary of the strip-like end, and near $z_{-}$, the linear Lagrangian boundary condition $\nu$ is equal to $(T_{e(0)}\mathcal{L}, T_{e(1)}(L \times L'))$ on each boundary of the strip-like end. Since we have chosen the spin structures, the Lagrangian boundary condition $\nu: \partial S \to \mathcal{LAG}(TM)$ has a preferred lift to the brane Grassmannian $\widetilde{\mathcal{LAG}}(TM)$. \par
	Denote by $\mathcal{R}((\gamma^{0}, \gamma^{1}), e)$ the moduli space of the above inhomogeneous pseudoholomorphic maps $f$. The following proposition follows from the general theory which assigns orientations to inhomogeneous pseudoholomorphic maps from punctured Riemann surfaces (Proposition 11.13 of \cite{Seidel}). \par

\begin{proposition}\label{orientation on the moduli space of folded quilt}
	Let $\bar{C} \cong S^{1}$ be the boundary circle of the compactified surface $\bar{S} = D$, and $z_{+}, z_{-}$ the two punctures. Fix a marked point $p \in \bar{C} \set minus \{z_{+}, z_{-}\}$ mapping to a Lagrangian $\mathcal{L}$ or $L \times L'$. Assuming the moduli space $\mathcal{R}((x, y), e)$ is regular at a point $f$, then
\begin{equation}
\Lambda^{top}(T_{f}\mathcal{R}((\gamma^{0}, \gamma^{1}), e)) \cong \det(D_{S, \nu}) \cong o_{(\gamma^{0}, \gamma^{1})}^{-1} \otimes o_{e}.
\end{equation}
\end{proposition}

	Now going back to the original quilted surface, we need to take into account the additional negative puncture on the boundary of the second patch, where the quilted map asymptotically converges to the time-one Hamiltonian chord $\theta$ from $L'$ to itself. Again, there is a unique homotopy class of trivialization of the pullback of $TN$ by $\theta$ which is compatible with the trivialization of the canonical bundle $\Lambda^{top}TN$. Consider a smooth map
\begin{equation}
g: \tilde{H} \to N,
\end{equation}
which limits to $\theta$ over the positive strip-like end. The linearized inhomogeneous Cauchy-Riemann operator with respect to the chosen trivialization is denoted by $D_{\tilde{H}, \mu}$, where $\mu: \partial \tilde{H} \to \mathcal{LAG}(TM)$ is the linear Lagrangian boundary condition underlying the chosen path in the brane Grassmannian connecting the lifts of $T_{\theta(0)}L'$ and $T_{\theta(1)}L'$. The determinant $D_{\tilde{H}, \mu}$ is then the inverse orientation line $o_{\theta}^{-}$, which is well-defined and independent of the choice of paths in the brane Grassmannian. \par
	We may glue this map $g$ to the quilted map $(u, v)$ to get a map $f$, and correspondingly the linearized Cauchy-Riemann operators $D$ and $D_{\tilde{H}, \mu}$ also glue together to form the operator $D_{S, \nu}$. Applying the gluing formular for Cauchy-Riemann operators as in section (11c) of \cite{Seidel}, we deduce that
\begin{equation}
\det(D_{S, \nu}) \otimes \cong \det(D) \otimes \det(D_{\tilde{H}, \mu}).
\end{equation}
On the other hand, $\det(D_{\tilde{H}, \mu})$ is by definition the inverse orientation line $o_{\theta}^{-}$ of $\theta$, so we conclude that: \par

\begin{corollary}
	The assumptions are the same as those in Proposition \ref{orientation on the moduli space of folded quilt}. At a smooth point $(u, v)$ of the moduli space $\mathcal{N}((\gamma^{0}, \gamma^{1}), e; \theta)$, we have a preferred isomorphism
\begin{equation}
\Lambda^{top}T_{(u,v)}\mathcal{N}((\gamma^{0}, \gamma^{1}), e; \theta) \cong \det(D) \cong o_{(\gamma^{0}, \gamma^{1})}^{-} \otimes o_{e} \otimes o_{\theta}.
\end{equation}
\end{corollary}

	This defines the orientation on the moduli space $\mathcal{N}((\gamma^{0}, \gamma^{1}), e; \theta)$. It is straightforward from the definitions to see that this orientation is consistent with the orientations on the moduli spaces of inhomogeneous pseudoholomorphic quilted strips as well as the one on the moduli space of inhomogeneous pseudoholomorphic strips with boundary on $(L', L')$, which have interaction with $\mathcal{N}((\gamma^{0}, \gamma^{1}), e; \theta)$. For instance, we can see this again using the gluing formula for the Cauchy-Riemann operators. Now the proof of Proposition \ref{coherent orientation on the moduli space of quilts} is complete. \par


\bibliography{wrapped1}
\bibliographystyle{alpha}

\end{document}